\documentclass{amsart}
\usepackage{amssymb}  
\usepackage{amsthm} 
\usepackage{amsmath}
\usepackage{amsfonts}
\usepackage{graphicx}
\usepackage{placeins}
\usepackage[utf8]{inputenc}
\usepackage[english]{babel}
\usepackage[all]{xy}
\usepackage[linktocpage=true,colorlinks=true, allcolors=blue]{hyperref} 
\usepackage{amssymb}  
\usepackage{amsthm}

\theoremstyle{definition}
\newtheorem{definition}{Definition}

\newtheorem{example}{Example}

\newtheorem{remark}{Remark}

\theoremstyle{plain}
\newtheorem{theorem}{Theorem}
\newtheorem{corollary}{Corollary}
\newtheorem{proposition}{Proposition}

\newtheorem*{question}{Question}

\newcommand{\C}{\mathbb{C}}

\newcommand{\Q}{\mathbb{Q}}
\newcommand{\bQ}{\bar{\mathbb{Q}}}
\newcommand{\N}{\mathbb{N}}
\newcommand{\Z}{\mathbb{Z}}

\newcommand{\cD}{\mathcal{D}}

\newcommand{\ie}{\emph{i.e},\,}
\author{Jos\'e Juan-Zacar\'ias \& Alberto Verjovsky}

\address{Instituto de Matem\'aticas Unidad
 Cuernavaca, Av. Universidad s/n. Col. Lomas de Chamilpa C\'odigo Postal 
 62210, Cuernavaca, Morelos.}
 
\email{jose.juan@im.unam.mx, albertoverjovsky@gmail.com}

\thanks{This work was partially supported by PAPIIT (Universidad 
Nacional Aut\'onoma de M\'exico) project \#IN108120.}
\title[Equilateral triangulations of surfaces and Belyi functions]{Some 
remarks on equilateral triangulations of surfaces and Belyi functions} 

\begin{document} 

\maketitle 
\setcounter{tocdepth}{1} 
\begin{abstract} 
	In this paper, following Grothendieck \emph{Esquisse d'un 
	programme}, which was motivated by Belyi's work, we study some 
	properties of surfaces $X$ which are triangulated by (possibly 
	ideal) isometric equilateral triangles of one of the spherical, 
	euclidean or hyperbolic geometries. These surfaces have a natural 
	Riemannian metric with conic singularities. In the euclidean case 
	we analyze the closed geodesics and their lengths. Such surfaces 
	can  be given the structure of a Riemann surface which, considered 
	as algebraic curves, are defined over $\bar\Q$ by a theorem of 
	Belyi. They have been studied by many authors of course. Here we 
	define the notion of connected sum of two Belyi functions and give 
	some concrete examples. In the particular case when $X$ is a torus, 
	the triangulation leads to an elliptic curve and we define the 
	notion of a ``peel" obtained from the triangulation (which is a 
	metaphor of an orange peel) and relate this peel with the modulus 
	$\tau$ of the elliptic curve. Many fascinating questions arise 
	regarding the modularity of the elliptic curve and the geometric 
	aspects of the Taniyama-Shimura-Weil theory.
\end{abstract}

\tableofcontents

\section{Introduction}
One beautiful and fundamental result of the second half of the XX 
century is the result by Belyi which characterizes complex Riemann 
surfaces which, regarded as algebraic curves, can be given by equations 
with coefficients in $\bar{\mathbb Q}$.  Belyi's theorem states: a 
Riemann surface $\Sigma$ is defined over $\bar{\mathbb Q}$ if and only 
if it admits a meromorphic function $f:\Sigma\to\hat\C=\mathbb{P}^1_\C$ 
(the Riemann sphere), with at most three critical values which can be 
taken, without loss of generality, to be $0$ $1$ and $\infty$. 

This theorem fascinated Alexander Grothendieck for its simplicity and 
depth to a degree that it changed his line of investigation and wrote 
the epoch making paper {\em Esquisse d'un Programme} \cite{Esquisse}. 

The function $f$ is called ``a Belyi map''  and topologically expresses 
$\Sigma$ as a branched cover over the Riemann sphere $\hat{\mathbb C}$ 
with branching points a subset of $\{0,1,\infty\}$ (if f has only two 
critical values then, $f$ is up to a change of coordinates, the 
function $f(z)=z^n$, for some $n>1$).  Such a branched covering is 
completely determined by the inverse image 
$f^{-1}([0,1]):=\mathcal{D}_f$. Then $\mathcal{D}_f$ is a bipartite 
graph with colored vertices: ``white'' for the points in 
$f^{-1}(\{0\})$ and ``black'' for the points in $f^{-1}(\{1\})$. The 
graph  $\mathcal{D}_f$ was named by Grothendieck himself \emph{Dessin 
d'enfant} \cite{Esquisse}, but in this paper some times just name it 
\emph{dessin}.  Besides, $\Sigma\setminus\cD_f$ is a union of open sets 
homeomorphic to the open unit disk. This endows $\Sigma$ with a 
decorated (cartographic) map. Reciprocally, given a compact, oriented, 
connected, smooth surface $\Sigma$ with a cartographic map 
$\mathcal{G}$ (\ie an embedded graph) one can endow $\Sigma$ with a 
unique complex structure and a holomorphic map 
$f:\Sigma\to\overline\C=P^1_{\mathbb C}$ with critical values $0$ $1$ 
and $\infty$, such that $\mathcal{D}_f$ is $\mathcal{G}$. Voevodsky and 
Shabat in \cite{VS} have shown that if a closed, connected surface has 
a triangulation by euclidean equilateral triangles, then the flat 
structure with conic singularities gives the surface the structure of 
Riemann surface which as an algebraic complex curve can be defined over 
$\bar\Q$. The same is true if the surface is triangulated with 
congruent hyperbolic triangles of angles $\frac\pi{p}, \frac\pi{q}, 
\frac\pi{r}$ ($p,q,r\in\N,\,\,\frac1{p}+\frac1{q}+\frac1{r}<1$), 
including $p=q=r=\infty$ \cite{CIW}: such surfaces are defined also 
over $\bQ$. The aim of this paper is to study some properties and 
constructions of decorated triangulated surfaces. 
 
\emph{In this paper we have tried to follow the simplicity of the ideas 
that fascinated Grothendieck in the spirit of the expository paper 
\cite{Gui14} and make self-contained.}

\section{Equilateral structures on surfaces} 
\subsection{Canonical equilateral triangulation of the Riemann sphere.} 
Let us consider the equilateral euclidean triangle $\Delta$ with 
vertices $0,1,\omega$, where $\omega=\text{e}^{\frac{2i\pi}6}$ is the 
sixth primitive root of unity. Denote by $\bar\Delta$ the triangle 
which is the reflection of $\Delta$ with respect to the real axis (see 
Figure (\ref{esfera-triangular})). If we identify the boundaries of 
$\Delta$ and $\bar\Delta$ by means of the conjugation map $z\mapsto 
\bar{z}$, we obtain a compact surface of genus 0 with a triangulation 
by two triangles. Such a surface, homeomorphic to the sphere, has a 
natural complex structure which, in the complement of the vertices, is 
given by complex charts with changes of coordinates given by 
orientation-preserving euclidean isometries. \emph{In this paper 
$2\Delta$ will always denote this surface.}

\begin{figure}
	\begin{center}
	\includegraphics[width=7cm]{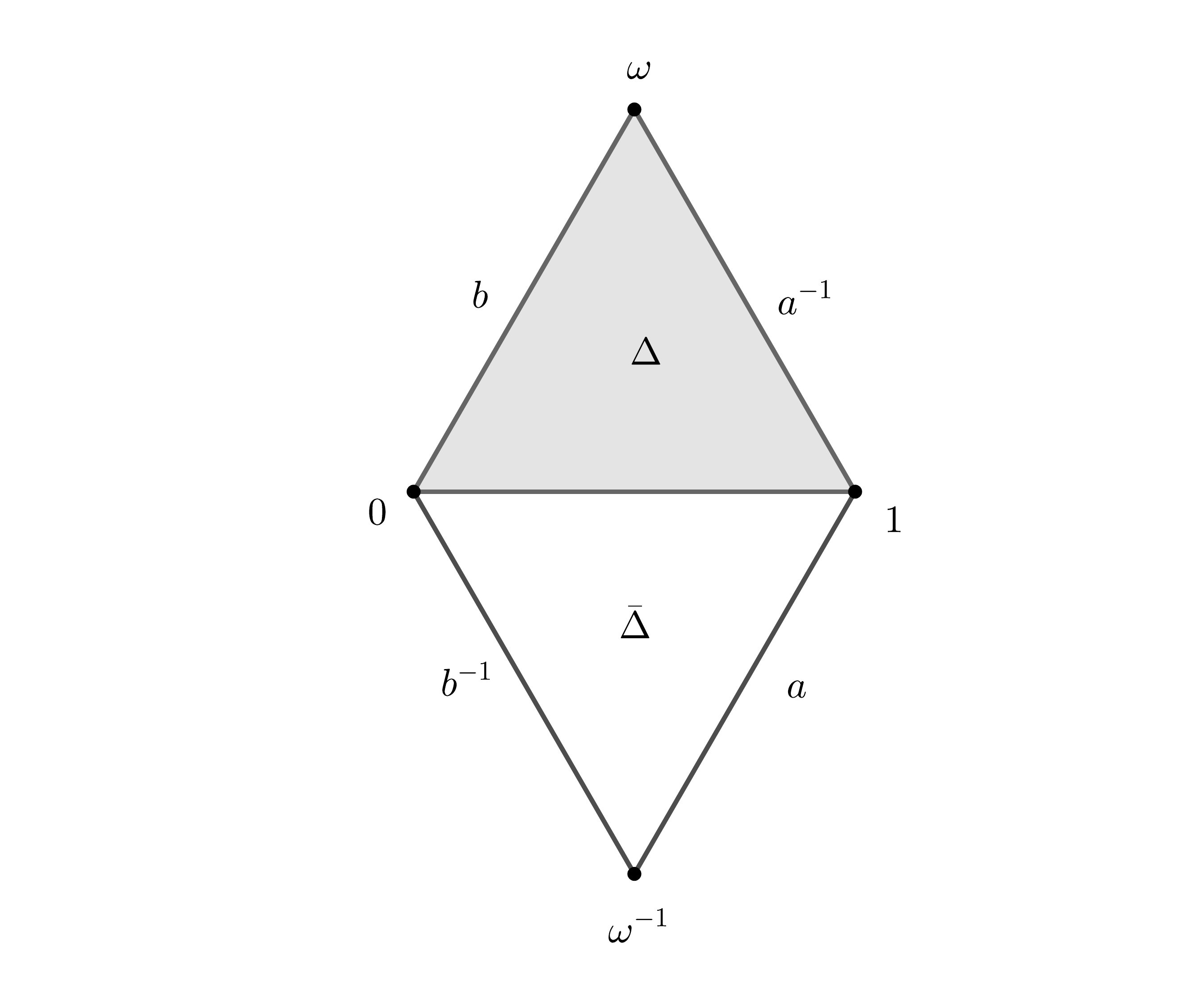} 
	\caption{The sphere as the union of two equilateral triangles.} 
	\label{esfera-triangular} 
	\end{center}
\end{figure}

\par By the uniformization theorem (or Riemann-Roch's theorem) any 
Riemann surface of genus 0 is conformally equivalent to the Riemann 
sphere. In fact, the Riemann surface $2\Delta$ defined in the previous 
paragraph can be uniformized \emph{explicitly} as follows: consider the 
\emph{Schwarz-Christoffel} map $\varphi\colon \bar{\mathbb{H}}\to 
\Delta$ defined by the formula 

\begin{equation}
	\varphi(z)=C\int_{0}^{\zeta}\frac{dw}{w^{2/3}(w-1)^{2/3}}, 
\end{equation}

\noindent  where $C$ is an appropriate  complex constant. The function 
$\varphi$ is a homeomorphism which maps conformally the upper 
half-plane $\mathbb{H}$ onto the interior of $\Delta$, and maps the 
points $0,1,\infty$ to $0,1,\omega$, respectively. 

\par Let $\Omega$ be the region on 
$\mathbb{H}^{+}\cup \mathbb{H}^{-}\cup 
[0,1]$, and define the map $\Phi\colon \Omega\to 
\operatorname{Int}(\Delta\cup \bar\Delta)$ as follows

\begin{equation}
	\Phi(z)=
	\begin{cases}
		\varphi(z),\ z\in \mathbb{H}^{+}\cup [0,1], \\
		\overline{\varphi(\bar{z})},\ z\in \mathbb{H}^{-}\cup [0,1].		
	\end{cases}
\end{equation} 

\noindent  {\em Schwarz reflection principle} implies that $\Phi$ is 
conformal in $\Omega$.  In addition, $\Phi$ induces a homeomorphism 
from the Riemann sphere $\hat{\mathbb{C}}$ to the quotient space. 
Abusing the notation we denote this function also  by $\Phi$. 

\par Applying Schwarz reflection principle to $\Phi$ in local 
coordinates we can verify that $\Phi$ is holomorphic in 
$\hat{\mathbb{C}}-\{0,1,\infty\}$. By Riemann extension theorem this 
map extends to the entire Riemann sphere.

\begin{figure}
	\begin{center}
	\includegraphics[width=12cm]{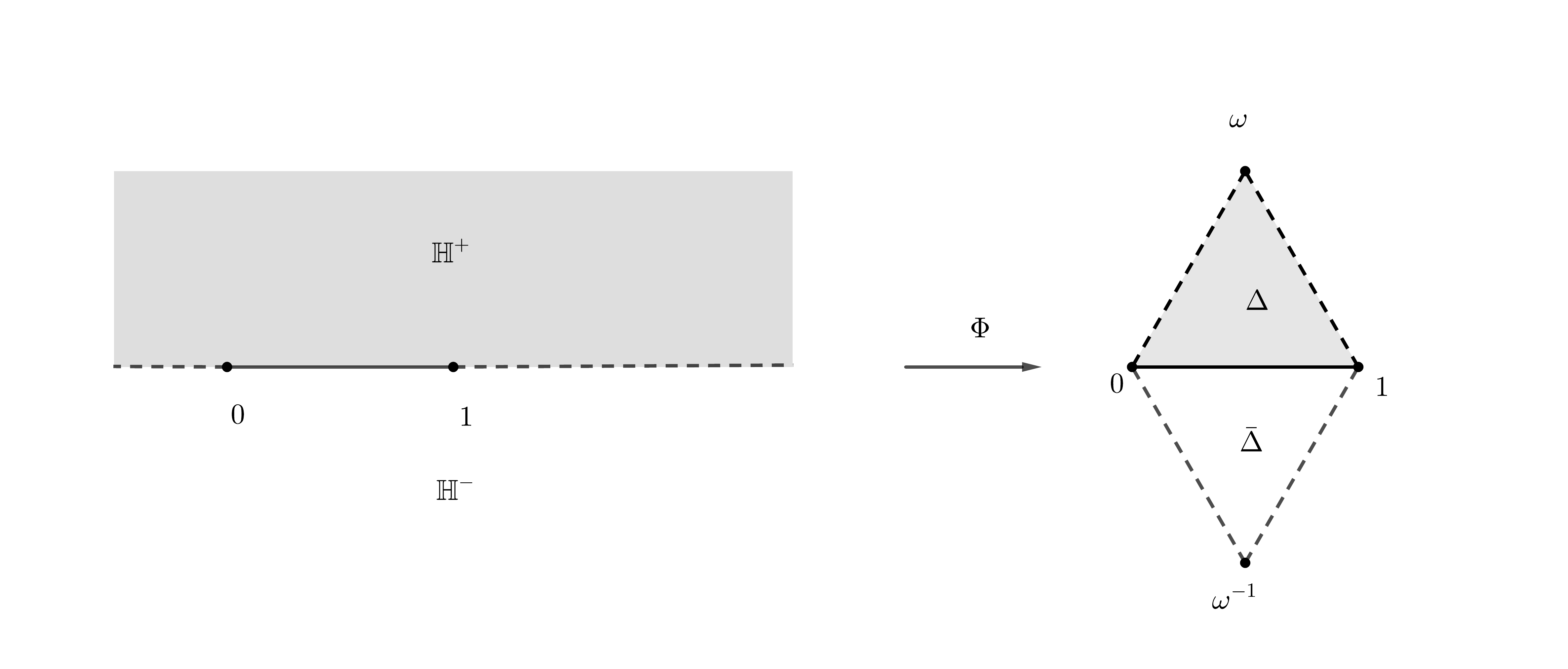} 
	\caption{Schwarz-Christoffel map and its reflection.} 
	\label{uniformizacion-esfera} 
	\end{center}
\end{figure}

\begin{remark}
	The flat metrics of the triangles $\Delta$ and $\bar\Delta$ induce 
	a singular flat metric on $2\Delta$ and the two triangles form an 
	equilateral triangulation of such surface. The vertices become 
	conic singularities. Hence we can pull-back, via $\Phi$, this flat 
	singular Riemannian metric to $\hat{\mathbb{C}}-\{0,1,\infty\}$. 
	With this metric the upper and lower half-planes with marked points  
	$0,1,\infty$ are isometric to equilateral euclidean triangles.
\end{remark}

\subsection{Equilateral euclidean triangulations on surfaces.} 

\begin{definition}(Euclidean triangulation)\label{def-triang-eucli}
	Let $X$ be a compact oriented surface. Then, a \emph{euclidean 
	triangulation}  of $X$ is given by a set of homeomorphisms 
	$\varphi_{i}\colon \Delta_{i}\to T_{i}$, where $\Delta_{i}$ are 
	euclidean triangles with the additional condition that if $T_{i}$ 
	and $T_{j}$ share an edge $a$, then the metrics on $a$ induced by 
	$\varphi_{i}$ and $\varphi_{j}$ coincide. If all the triangles 
	$\Delta_{i}$ are equilateral the triangulation is called 
	\emph{equilateral triangulation}. 
\end{definition}

\noindent Euclidean (respectively equilateral) triangulations on a 
surface will be also be referred as \emph{euclidean (respectively 
equilateral) triangulated structure} on the surface. 

\begin{remark}
	We don't assume that the triangles of our triangulations meet at 
	most in one common edge, we allow the triangles to meet in more 
	than one edge.
\end{remark}

\begin{remark}
	(a) If $a$ is a common edge of $T_{i}$ and $T_{j}$, then 
	$\varphi_{i}\colon \Delta_{i}\to T_{i}$ and $\varphi_{j}\colon 
	\Delta_{j} \to T_{j}$ induce the same linear metric on $a$ if and 
	only if $\varphi_{i}^{-1}(a)$ and $\varphi_{j}^{-1}(a)$ have the 
	same length and the barycentric coordinates in $a$, induced by 
	$\varphi_i$ and $\varphi_j$, are the same.\\
	
	\noindent (b) In an equilateral triangulation all the edges have 
	the same length.
	
\end{remark}

A triangulation of a surface $X$ with normal barycentric coordinates 
which is coherently oriented has a canonical complex structure. The 
charts around points which are not vertices are given as follows:

\begin{itemize}

	\item[(i)] Interior points. Let
	 $\psi_{i}$ denote the inverse of the homeomorphism
	 $\varphi_{i}:\Delta_{i}\to T_{i}$, then the map
	 $\psi_{i}\colon \operatorname{Int}T_{i}\to \operatorname{Int} 
	\Delta_{i}$ is a chart for all points in the interior of $T_{i}$. 
	
	\item[(ii)] Points on the edges. Suppose that $a$ is a common edge 
	of triangles the $T_{i}$ and $T_{j}$, let $R$ be the reflection 
	with respect to the edge $\psi_{i}(a)$. There exists an 
	orientation-preserving euclidean isometry $\mu_{j}$, which send 
	$\Delta_{j}$ onto $R(\Delta_{i})$ and such that $\mu_{j}\circ 
	\psi_{j}=\mu_{i}$ in $a$ (see Figure (\ref{carta-arista})).

	Define the homeomorphism $\Psi\colon T_{i}\cup T_{j}\to 
	\Delta_{i}\cup R(\Delta_{i})$ as follows 

	\begin{equation*}
		\Psi(x)=\begin{cases}\psi_{i}(x),\quad x\in T_{i}\\
		\mu_{j}\circ \psi_{j}(x), \quad x\in T_{j}.\end{cases}	
	\end{equation*}

	\noindent Then we can choose as a chart for the points in the 
	interior of the edge $a$ the previous map restricted to 
	$\operatorname{Int}(T_{i}\cup T_{j})$. One can verify that this map 
	is compatible with the charts defined in (i).

\begin{figure}
	\begin{center}
	\includegraphics[width=12cm]{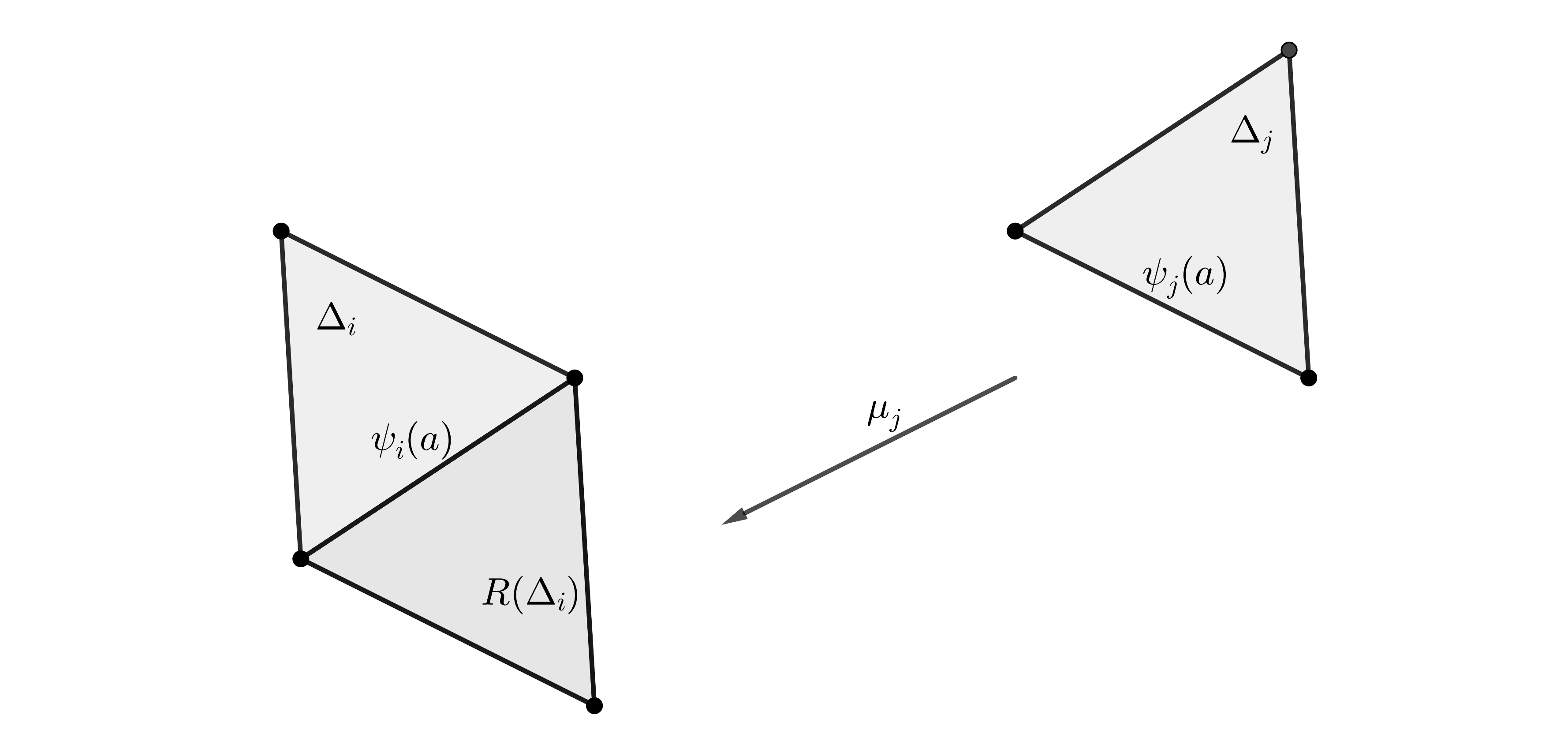} 
	\caption{The Isometry $\mu_{j}$.} 
	\label{carta-arista} 
	\end{center}
\end{figure}

\end{itemize}

\noindent One can show (see, for instance, \cite{Springer, Troyanov}) 
that the complex structure extends to the vertices.

\par Since a euclidean triangulated structure on a surface is given by 
normal barycentric coordinates with a coherent orientation we have the 
following proposition:

\begin{proposition} 
	A euclidean triangulated structure on $X$ induces a complex 
	structure on $X$ which turns it into a Riemann surface and 
	therefore (by Riemann) is an algebraic complex curve. 
\end{proposition}
	 
\begin{remark}
	The complex structure does not change if we rescale the euclidean 
	triangles, so we may assume in all that follows that the edges have 
	length one.
\end{remark}


\subsection{Belyi functions}

Let $X$ be a compact Riemann surface. A meromorphic function $f\colon X 
\to \hat{\mathbb{C}}$ with critical values contained in the set 
$\{0,1,\infty\}$ will be called a \emph{Belyi function}. If a compact 
Riemann surface $X$ has a Belyi function $f$, in its field of rational 
functions, the surface is called \emph{Belyi surface} or \emph{Belyi 
curve}, and $(X,f)$ wil be called a \emph{Belyi pair}.

Two Belyi pairs $(X_1,f_1)$ and $(X_2,f_2)$ are said to be 
\emph{equivalent} if they are isomorphic as branched coverings over the 
Riemann sphere \ie , there exists a conformal map $\Phi\colon X_1\to 
X_2$ such that the following diagram:

\begin{equation}
	\xymatrix{X_1\ar_{f_1}[dr]\ar^{\Phi}[rr] & & X_2\ar^{f_2}[dl]\\ 
	&\hat{\mathbb{C}} & },
\end{equation} 
is commutative.

We say that a Riemann surface  $X$ is \emph{defined over a field 
$K\subset \mathbb{C}$} if there exists a polynomial $P(x,y)\in 
K[x,y]\subset\mathbb{C}[x,y]$ such that $X$ is conformally equivalent 
to the normalization $\widetilde{C}$, of the curve $C$, defined by the 
zeroes of $P$ in $\mathbb{C}^2$:

\begin{equation}
	C=\{(x,y)\in \mathbb{C}^2\colon P(x,y)=0\}.
\end{equation} 

If $C^*$ is the nonsingular locus of $C$ then the compact Riemann 
Riemann surface $\widetilde{C}$ contains $C^*$ (\ie contains a 
holomorphic copy of $C^*$) and $\tilde{C}-C^*$ is a finite set. In 
addition,  $\widetilde{C}$ is unique up to a conformal isomorphism.

In 1979 Belyi \cite{Belyi} gave a criterium to determine if a compact 
Riemann surface is defined over an algebraic number field:

\begin{theorem}[Belyi's Theorem]
	Let be a compact Riemann surface $X$. The following statements are 
	equivalent: 
	
	\begin{itemize}
		\item[(i)] $X$ is defined over $\bar{\mathbb{Q}}$, the field of 
		algebraic numbers. 
		
		\item[(ii)] There exists a meromorphic function $f\colon X\to 
		\hat{\mathbb{C}}$, such that its critical values belong to the 
		set $\{0,1,\infty\}$. 
	\end{itemize}
\end{theorem}

This is the main motivation of the present paper. The proof can be 
consulted in \cite{Belyi}, \cite{Girondo-Gonzalez} or \cite{JW16}. In 
the following sections of this paper we will discuss several examples 
of Belyi functions.

\begin{definition}
	A \emph{dessin d'enfant}, or simply \emph{dessin}, is a pair 
	$(X,\mathcal{D})$ where $X$ is a compact oriented topological 
	surface, and $\mathcal{D}\subset X$ is a finite graph such that:
	
	\begin{itemize}
		\item[(i)] $\mathcal{D}$ is connected.
		
		\item[(ii)] $\mathcal{D}$ is bicolored, \ie its vertices are 
		colored with the colors white and black in such a way that two 
		vertices connected by an edge have different colors. 
		
		\item[(iii)] $X-\mathcal{D}$ is a finite union of open 
		topological 2-disks named \emph{faces}. 
	\end{itemize}
\end{definition}

Two dessins $(X_{1},\mathcal{D}_{1})$, $(X_{2},\mathcal{D}_{2})$ are 
considered as equivalent if there exists an orientation-preserving 
homeomorphism $h:X_{1}\to{X_{2}}$ such that 
$h(\mathcal{D}_{1})=\mathcal{D}_{2}$ and the restriction of $h$ to 
$\mathcal{D}_{1}$ induces an isomorphism of the colored graphs 
$\mathcal{D}_{1}$ and $\mathcal{D}_{2}$.

The following proposition tells us how to associate a dessin d'enfant 
to a Belyi function.

\begin{proposition}
	Let $(X,f)$ be a Belyi function, and consider 
	$\mathcal{D}_f=f^{-1}([0,1])$ as an embedded bicolored graph in $X$ 
	where the white (respectively black)  vertices are the points of 
	$f^{-1}(0)$ (respectively $f^{-1}(1)$). Then $\mathcal{D}_f$ is a 
	dessin d'enfant in $X$.
\end{proposition}


Reciprocally, given a dessin $(X,\mathcal{D})$ one can associate a 
complex structure to the topological surface $X$ and a Belyi function 
$f\colon X \to \hat{\mathbb{C}}$ that realizes the dessin \ie  
$\mathcal{D}=f^{-1}([0,1])$:

\begin{proposition}
	Let $(X,\mathcal{D})$ be a dessin, then there exists a complex 
	structure on $X$ which turns it into a compact Riemann surface 
	$X_{\mathcal{D}}$ and a Belyi function $f_{\mathcal{D}}\colon 
	X_\mathcal{D} \to \hat{\mathbb{C}}$ which realizes the dessin. In 
	addition, the  Belyi pair $(X_{\mathcal{D}},f_{\mathcal{D}})$ is 
	unique, up to equivalence of ramified coverings.
\end{proposition}

In the following section we will describe in detail such a construction 
in the case of a surface with a decorated equilateral triangulation.

The following theorem establishes that the previous correspondences 
within their set of equivalence classes are inverse to each others. 

\begin{proposition}
	The correspondence 
	\begin{equation}
		(X,f)\longmapsto (X,\mathcal{D}_f)
	\end{equation} 
	induces a one-to-one correspondence between the set of equivalence 
	classes of dessins d'enfant and the set of equivalence classes of 
	Belyi pairs. The inverse function is induced on equivalence classes 
	by the correspondence:	
	\begin{equation}
		(X,\mathcal{D})\longmapsto (S_D,f_D).
	\end{equation}
\end{proposition}

The proof can be consulted in  \cite{Girondo-Gonzalez}.

Given a degree $d$ Belyi function $f\colon X\to \hat{\mathbb{C}}$ 
one has the associated monodromy map 
$\mu\colon \pi_1(\hat{\mathbb{C}}-\{0,1,\infty\},1/2)\to S_d$ of the unbranched covering
 (in \S 
\ref{seccion-geodesica-cerrada-primitiva} we recall the definition of this homomorphism). 
If $\gamma_0$ and $\gamma_1$ are loops in $\C-\{0,1\}$
 based in $1/2$ and with winding number 1  around  $0$ and $1$ (respectively) then the
permutations $\sigma_0=\mu(\gamma_0)$ and $\sigma_1=\mu(\gamma_1)$ 
generate a transitive subgroup of $S_d$.

Reciprocally, by the theory of covering spaces, given a pair
of permutations $(\sigma_0,\sigma_1)$ there exists a connected ramified covering of the sphere
$f\colon X\to \hat{\mathbb{C}}$ where $X$ is a topological surface which ramifies only at 
 $0,1,\infty$ and such that 
$\mu(\gamma_0)=\sigma_0$ and $\mu(\gamma_1)=\sigma_1$.
We can endow $X$ with the complex structure that makes $f$ a meromorphic function 
(\ie the pull-back of the complex structure of the sphere). Then $f$ becomes a Belyi function.

\par One can show that two Belyi pairs 
 $(X_1,f_1)$, $(X_2,f_2)$ are isomorphic if and only if its associated permutations
  $(\sigma_0,\sigma_1)$ and $(\sigma_0',\sigma_1')$, respectively, are 
conjugated \ie there exists  $\tau\in S_d$ such that $\sigma_0'=\tau 
\sigma_0 \tau^{-1}$ and $\sigma_1'=\tau \sigma_1 \tau^{-1}$.

We remark that we can obtain directly  $\sigma_0$ and $\sigma_1$ 
from the dessin: first we label the edges of the
dessin; we draw a small topological disk around each white vertex
and we define $\sigma_0(i)=j$ if $j$ is the consecutive edge of 
$i$, under a positive rotation. Analogously, we define 
 $\sigma_1$ by the same construction but now using the black vertices. 

Summarizing the previous discussion we have the following theorem:

\begin{theorem}
	There exists a natural one-to-one correspondence between the following sets: 
	\begin{itemize}
		\item[(i)] Belyi pairs $(X,f)$, modulo equivalence of ramified coverings.

		\item[(ii)] Dessins $(X,\mathcal{D})$, modulo 
		equivalence.
		
		\item[(iii)] Pairs of permutation $(\sigma_0,\sigma_1)$ 
		such that $\langle\sigma_0,\sigma_1\rangle$ is a transitive subgroup 
		of the symmetric group  $S_d$, modulo conjugation.
	\end{itemize}
\end{theorem}

\subsection{Belyi function of a decorated equilateral 
triangulation.}\label{funcion-triangulacion-equilatera}

Let $X$ be a compact, connected and oriented surface with an 
equilateral triangulation $\Delta=\{\varphi_{i}\colon \Delta_{i}\to 
T_{i}\}$, endowed with its associated complex structure and its 
Riemannian metric outside of the vertices which turns each triangle 
equilateral of the same size. We could, in fact, include the vertices 
and consider the metric which induces on each triangle a metric 
isometric with a euclidean equilateral triangle of a fixed size. Such a 
metric is called a \emph{singular euclidean metric}. With this metric 
$X$ becomes a \emph{length metric space}. The singularities are the 
vertices which have neighborhoods isometric to flat cones over a 
circle. Now suppose that the vertices of the triangulation are 
decorated by the symbols $\circ$, $\bullet$ and $\ast$, in such a way 
that \emph{no two adjacent vertices have the same decoration.} Such a 
decoration allows us to assign colors to the triangles in the following 
way: first, a Jordan region on the oriented surface has two 
orientations a positive one if it agrees with the orientation of $X$ 
and negative otherwise. This, in turn, induces an orientation in the 
boundary of the region, which is a Jordan curve, and the orientation is 
determined by triple of ordered points in this curve. Thus the 
orientation of the Jordan region is determined by an ordered triple of 
points in the boundary. 
 
By hypothesis the vertices of each triangle of the triangulation are 
decorated with the three different symbols, therefore the ordered 
triples  $(\circ,\bullet,\ast)$ and $(\circ,\ast,\bullet)$ determine an 
orientation on each triangle. Then, we can color with \emph{black} the 
triangles with triple $(\circ,\bullet,\ast)$ and we color with 
\emph{white} the ones corresponding to  $(\circ,\ast,\bullet)$. Hence 
such surfaces are obtained by gluing the edges of the triangles by 
isometries  which respect the decorations (see Figure 
(\ref{triangulos-decorados})).

\begin{figure}
	\begin{center}
	\includegraphics[width=12cm]{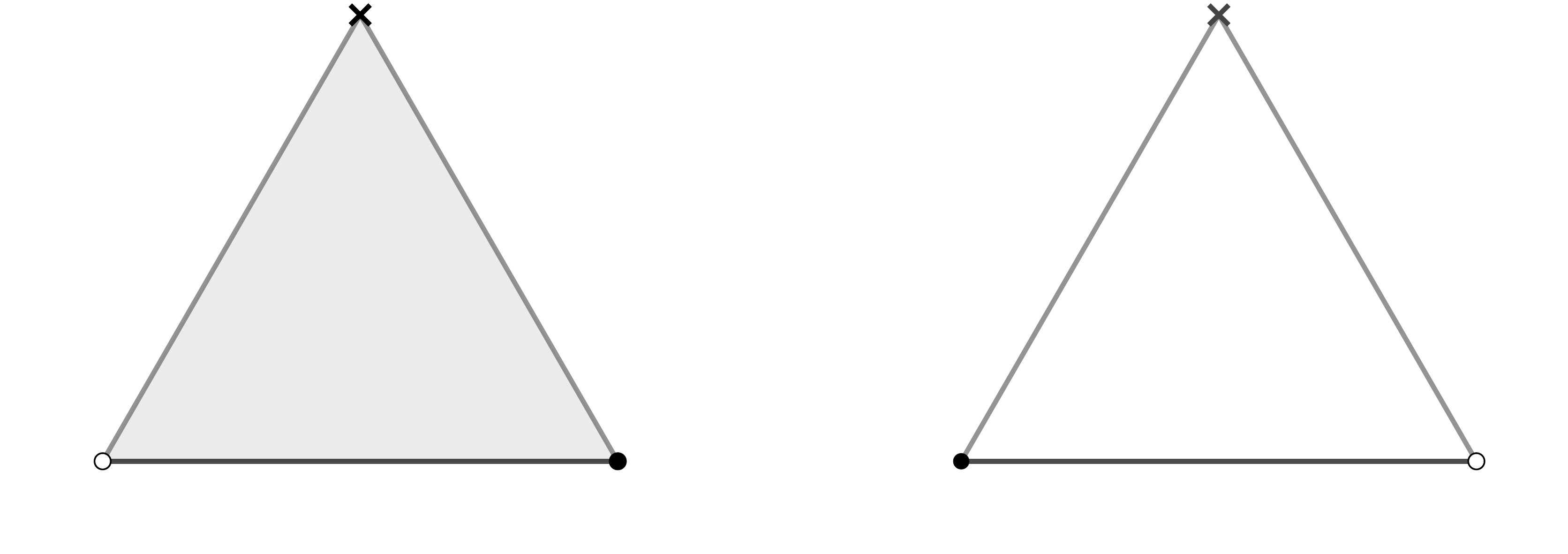} 
	\caption{Oriented equilateral triangles as building blocks.} 
	\label{triangulos-decorados} 
	\end{center}
\end{figure}

\begin{remark}\label{triangulo-observacion}
	(i) For each pair of adjacent triangles  
	$T_{i}$ and $T_{j}$ with the common edge
	$a$ there exists a reflection
	$R_{a}\colon T_{i}\cup T_{j}\to T_{i}\cup T_{j}$, which interchanges $T_{i}$ with
	$T_{j}$, \ie  $R_{a}$ is an orientation-reversing isometry and fixes each point of the edge
	$a$. \\  
	
	\noindent (ii) Let $T_{i}$ be a black triangle and 
	$\phi_{i}\colon T_{i}\to \overline{\mathbb{H}}^{+}$ a homeomorphism
	which preserves boundaries and is conformal in the interior and sends vertices
	$\circ, \bullet,\ast$ onto $0,1,\infty$, respectively. 
	
	If $T_{j}$ is a white triangle which is adjacent to
	$T_{i}$ with common edge  $a$, we can extend continuously $\phi_{i}$ to 
	a function $\Phi\colon T_{i}\cup T_{j}\to \hat{\mathbb{C}}$, 
	which is defined in $T_{j}$ as
	$\Phi(x)=\overline{\phi_{i}\circ R_{a}(x)}$. 
	By Schwarz reflection principle this function is holomorphic in the interior of
	$T_{i}\cup T_{j}$. We remark that it is a homeomorphism when restricted to
	 $T_{j}$ and respects the decoration of its vertices 
	 ($\circ$ are zeros, $\bullet$ are in the pre-image of 
	 1 and $\ast$ are poles).
	 Analogously we can extend any homeomorphism
	  $\phi_{j}\colon T_{j}\to \overline{\mathbb{H}}^{-}$, 
	  defined on a white triangle to a black triangle which is adjacent to it.
	   
\end{remark}

\subsection*{Construction of the Belyi function.} 
A surface with an equilateral decorated triangulation has associated a \emph{Belyi function}
constructed as follows: choose a black triangle $T_{i}$ and a homeomorphism 
$\phi_{i}\colon T_{i}\to \overline{\mathbb{H}}^{+}
$ satisfying the hypothesis of Remark
 \ref{triangulo-observacion} (ii).  Using the Schwarz reflection principle
 as in Remark \ref{triangulo-observacion} (i), we extend the function 
 $\phi_{i}$ to the adjacent triangle
 to $T_{i}$ which shares a vertex decorated with the symbol $\circ$. 
 We continue the process to another adjacent triangle until we use all the triangles 
 which have the vertex $\circ$ in common (see Figure (\ref{extension-vertice})). The function obtained by this process will be denoted as $\Phi$. 

\begin{figure}
	\begin{center}
	\includegraphics[width=12cm]{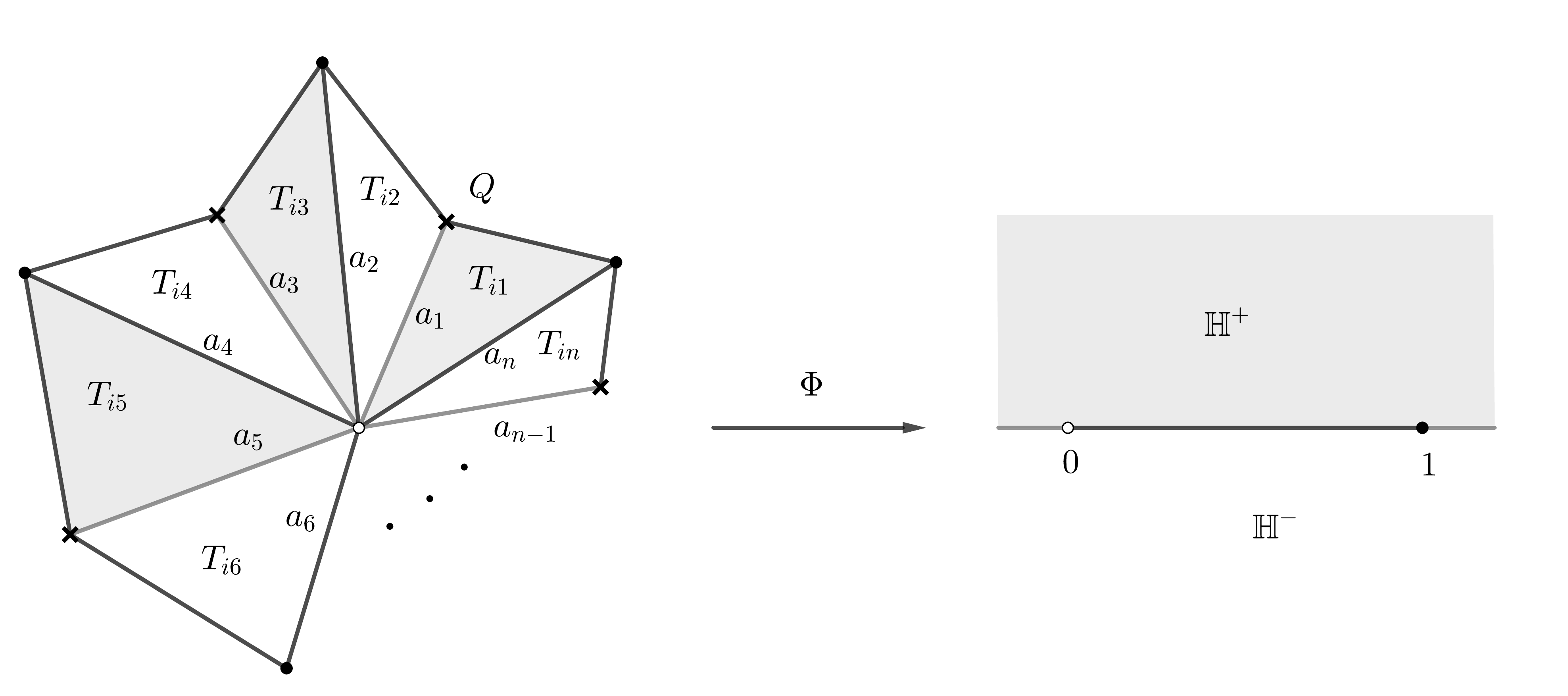} 
	\caption{Extension $\Phi$ in a neighborhood of vertex $\circ$ of
	$T_{i}=T_{i1}$.}  
	\label{extension-vertice} 
	\end{center}
\end{figure}

The function $\Phi$ is well-defined and holomorphic in the interior of
$T_{i1}\cup \ldots \cup T_{in}$. It is also well-defined in the edge
$a_{n}$ of the triangle $T_{i1}=T_{i}$ with vertices $\circ$ and $\bullet$.
 If $x\in a_n$, $\Phi(x)=\phi_{i}\circ R_{a_{1}}\circ\cdots \circ R_{a_{n-1}}(x)=\phi_{i}(x)$. 
 The function $\phi_i$ extends continuously to a neighborhood of the white vertex
$\circ=T_{i1}\cap \cdots \cap T_{in}$, and the extension is holomorphic in the interior.
By construction $\Phi$ is conformal in the interior of each triangle.

\par Analogously, we make the extension around the black vertex $\bullet$ of
 $T_{i}$ and also the vertex $\ast$. We continue with this process to all triangles which
 are adjacent to those triangles where the function 
 $\Phi$ has been defined until we do it for all the triangles of the euclidean decorated structure.
 This way we obtain a meromorphic function 
 $\Phi\colon X\to \hat{\mathbb{C}}$ with critical values at the points
 $\{0,1,,\infty\} $. 
 The construction of $\Phi$ implies that the pre-images of
 $0,1,\infty$  are decorated with the symbols
$\circ,\bullet,\ast$, respectively.
The black triangles correspond to the pre-images of the upper half-plane
and the white triangles to the pre-images of the lower half-plane. The 1-skeleton 
of the triangulation corresponds to the pre-image of
 $\bar{\mathbb{R}}=\mathbb{R}\cup \{\infty\}$. We summarize all of the above 
 by means of the following theorem (compare \cite{Bost,SV, VS}):

\begin{theorem} 
Let $X$ be a Riemann surface obtained by a decorated equilateral triangulation.
        Then there exists a Belyi function $f\colon X\to \hat{\mathbb{C}}$ such that the corresponding
        decorated triangulation coincides with the original. 
\end{theorem}

This meromorphic function is equivalent to the continuous function 
$f\colon X\to 2\Delta$ which sends each triangle of $X$ to one of the two triangles of
$2\Delta$, preserving the decoration.

\subsection{Belyi function of a symmetric decorated triangulated structure} 
We observe that the proofs of the results of
Section \ref{funcion-triangulacion-equilatera}
remain valid if we weaken the condition that the triangles are equilateral
to demanding only that condition (ii) in Remark
 \ref{triangulo-observacion} holds. Thus two triangles which share an edge 
 are obtained from each other by a reflection along the edge. Such euclidean triangulations 
will be called \emph{symmetric}. 
If the vertices of the triangulation are decorated by the symbols
$\circ,\bullet, \ast$, then its triangles can be colored in a similar way to the previous case.
The symmetric triangulations with such a decoration are called
 \emph{decorated symmetric triangulation} (compare \cite{CIW}).

\begin{proposition}\label{belyi-simetrica-decorada} (Compare \cite{CIW} Theorem 2.)
If $X$ is a compact Riemann surface given by a decorated symmetric triangulation,
then there exists a Belyi function $f\colon X\to \hat{\mathbb{C}}$ which realizes the triangulation.
\end{proposition}

\begin{corollary}\label{subdivision-baricentrica} Let $X$ be a compact Riemann surface obtained from
an equilateral triangulated structure. Then, given the barycentric subdivision there is a metric on $X$
in which all triangles in the subdivision are equilateral and the induced 
complex structure coincides with the original.

\end{corollary}
\begin{proof}
         By definition the complex structure given by the barycentric subdivision
	regarded as a euclidean triangulated structure coincides with the complex structure of $X$.
	
	\par Let us decorate the vertices of the subdivision as follows:
	put the symbol $\bullet$ to all of the original vertices, assign the symbol $\circ$
	to the midpoints of the edges and put the symbol $*$ to
	the barycenters. Then this triangulation is symmetric and therefore
	we have a decorated symmetric triangulation. By proposition
	\ref{belyi-simetrica-decorada} there exists a Belyi function 
	 $f:X\to \hat{\mathbb{C}}$ which realizes the triangulation. 
	
	\par If we consider the equilateral triangulated structure of the Riemann sphere by two
	triangles constructed at the beginning we can, using $f$, provide $X$ with a metric which 
	renders all the triangles of the barycentric subdivision equilateral. 
	This, in turn, induces a complex structure on $X$, since $f$ is conformal in the interior of each triangle. 
	 This complex structure coincides with the original.
	\end{proof}

\begin{example}(The $j$-invariant)
	The following rational function of degree 6
	\begin{equation}
		j(\lambda)=\frac{4}{27} 
		\frac{(\lambda^{2}-\lambda+1)^{3}}{\lambda^{2}(\lambda-1)^{2}},
	\end{equation} 
	
	\noindent is called the $j$-invariant. The $j$-invariant has 8 critical points
	two zeros of order 3, 3 preimages of 1 with multiplicity 2 and three double poles.
	Its critical values are $\{0,1,\infty\}$, in other words it is a Belyi function.
	Its {\em dessin d'enfant} consists of two line  segments which connect 
	$\omega, 1/2$ and $\omega^{-1}, 1/2$ and also by two circular arcs of radius 1 centered in
	0 and 1 (see Figure (\ref{dessinj})).
	
	\begin{figure}
		\begin{center}
		\includegraphics[width=10cm]{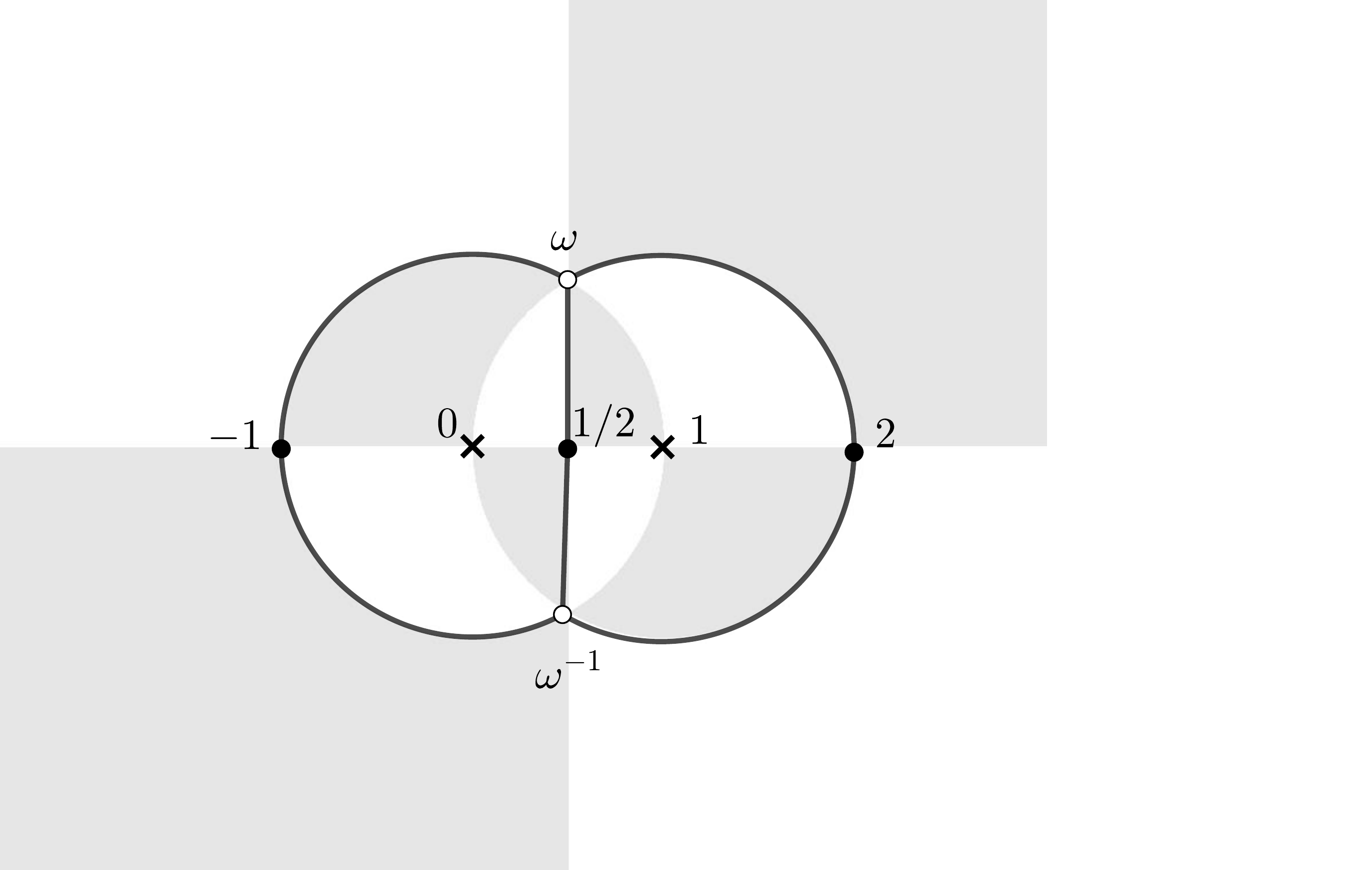}
		\caption{Dessin d'enfant of the $j$ invariant, where $\omega=\exp(2\pi 
		i/6)$.}  
		\label{dessinj} 
		\end{center}
	\end{figure}	
	
	The function $j$ induces, via pull-back, a metric on the sphere
	in which all triangles are equilateral. This metric can be obtained by 
	gluing euclidean equilateral triangles of the same size
	as shown in Figure (\ref{TriangulacionEquilateraj}).

	\begin{figure}
		\begin{center}
		\includegraphics[width=10cm]{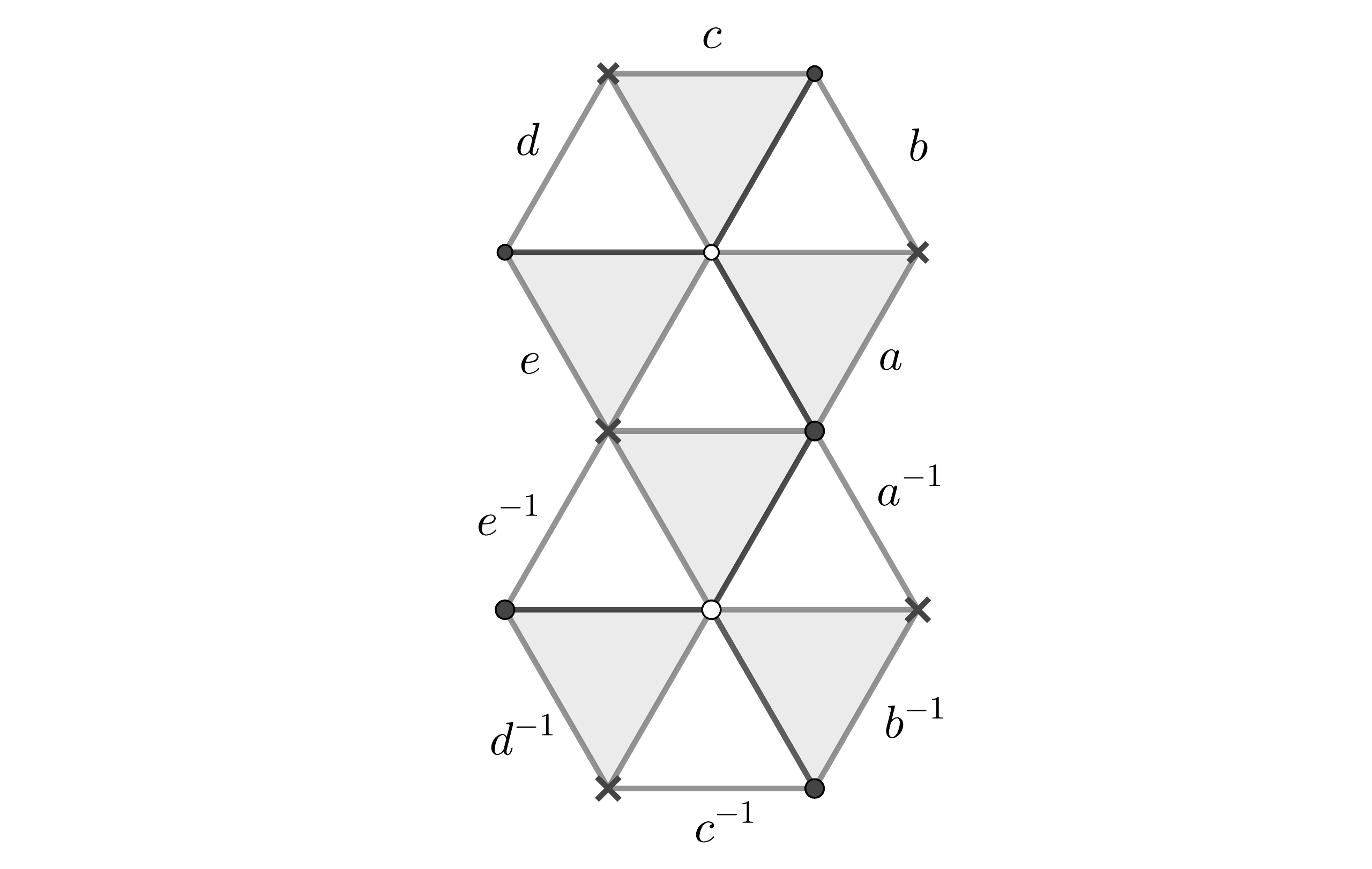} 
		\caption{Equilateral triangulation of the Riemann sphere associated to the $j$-invariant.}   
		\label{TriangulacionEquilateraj} 
		\end{center}
	\end{figure}	
\end{example}
In the previous paragraphs we constructed a Belyi function starting from an equilateral triangulation
without assuming that the triangulation was decorated. If the triangulation is decorated we had a Belyi function
$f$ that realizes the triangulation. Now we will show that the Belyi's function 
associated to the barycentric subdivision is equivalent to the function $1-1/j(f)$. In other words it is obtained by post 
composing $j(f)$ with the M\"obius function $z\mapsto1-1/z$.

Let us start by describing the
dessin d'enfant of $1-1/j(f)$. 
From the dessin of $j$ (see
Figure (\ref{dessinj})) we see that the decorated triangulation of
$1-1/j$ is as in the dessin in Figure (\ref{dessin1-jinv}).

\begin{figure}
	\begin{center}
	\includegraphics[width=10cm]{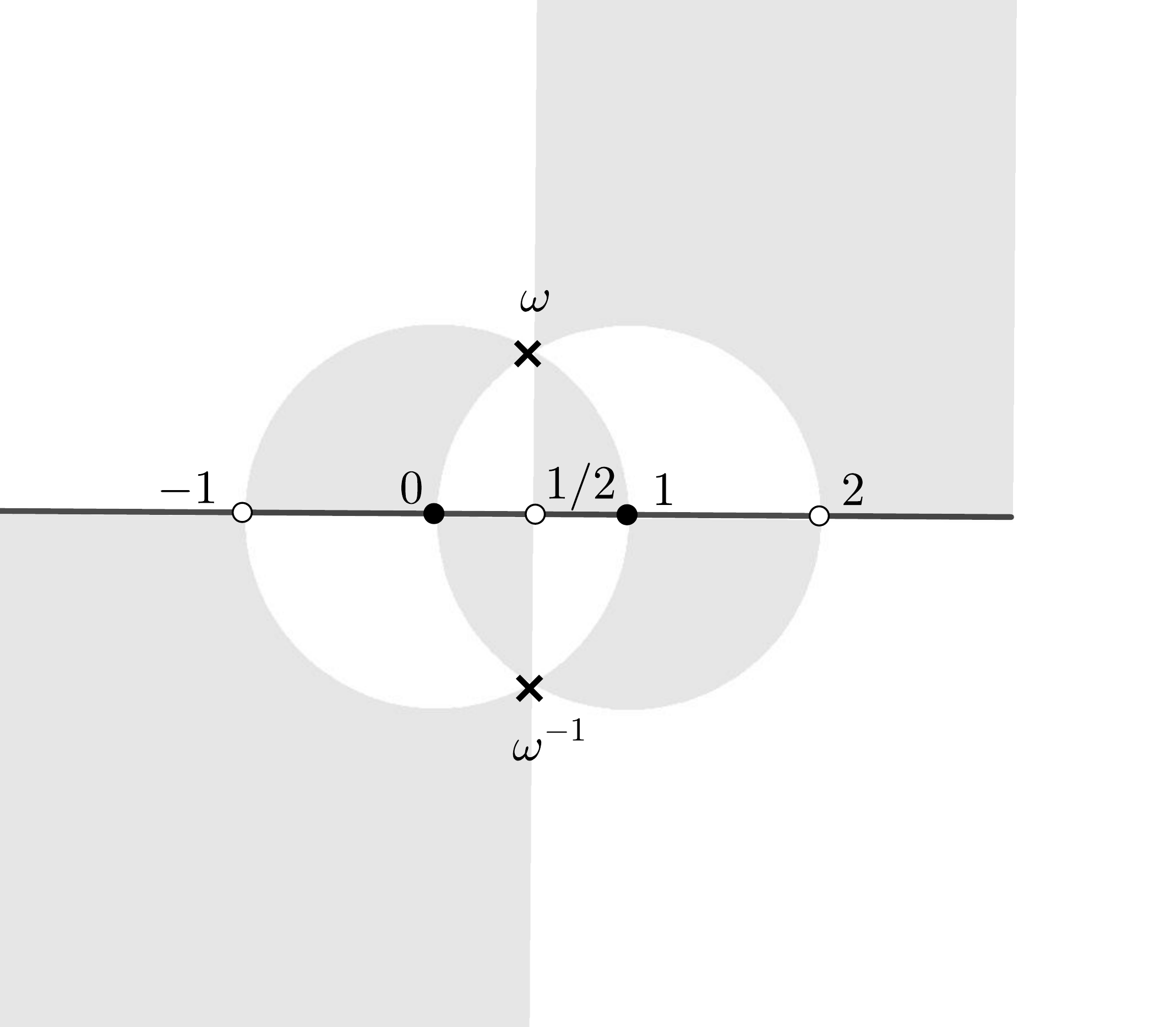} 
	\caption{Dessin d'enfant of $1-1/j$.}   
	\label{dessin1-jinv}
	\end{center}
\end{figure}

Let $\mathcal{D}$ denote the dessin of $f$ and $\mathcal{T}$ the corresponding triangulation. 
The following simple observations about the
 dessin and triangulation of
 $1-1/j(f)$, are obtained by analyzing Figure
(\ref{dessin1-jinv}):

\begin{itemize} 
	\item[(i)] The vertices of the triangulation have the symbol $\bullet$ in the new dessin 
	(including the poles)
        
        \item[(ii)] One adds a vertex with symbol $\circ$ in the interior of each edge of
         $\mathcal{D}$ (see Figure (\ref{dessin1-1jf})).
	
	\item[(iii)] One adds a pole $\ast$ in each original triangle
	and the pole is connected with vertices in the boundary of the triangle with symbols $\bullet$ and $\circ$	
	described in (i) and (ii). 
\end{itemize}

\begin{figure}
	\begin{center}
	\includegraphics[width=12cm]{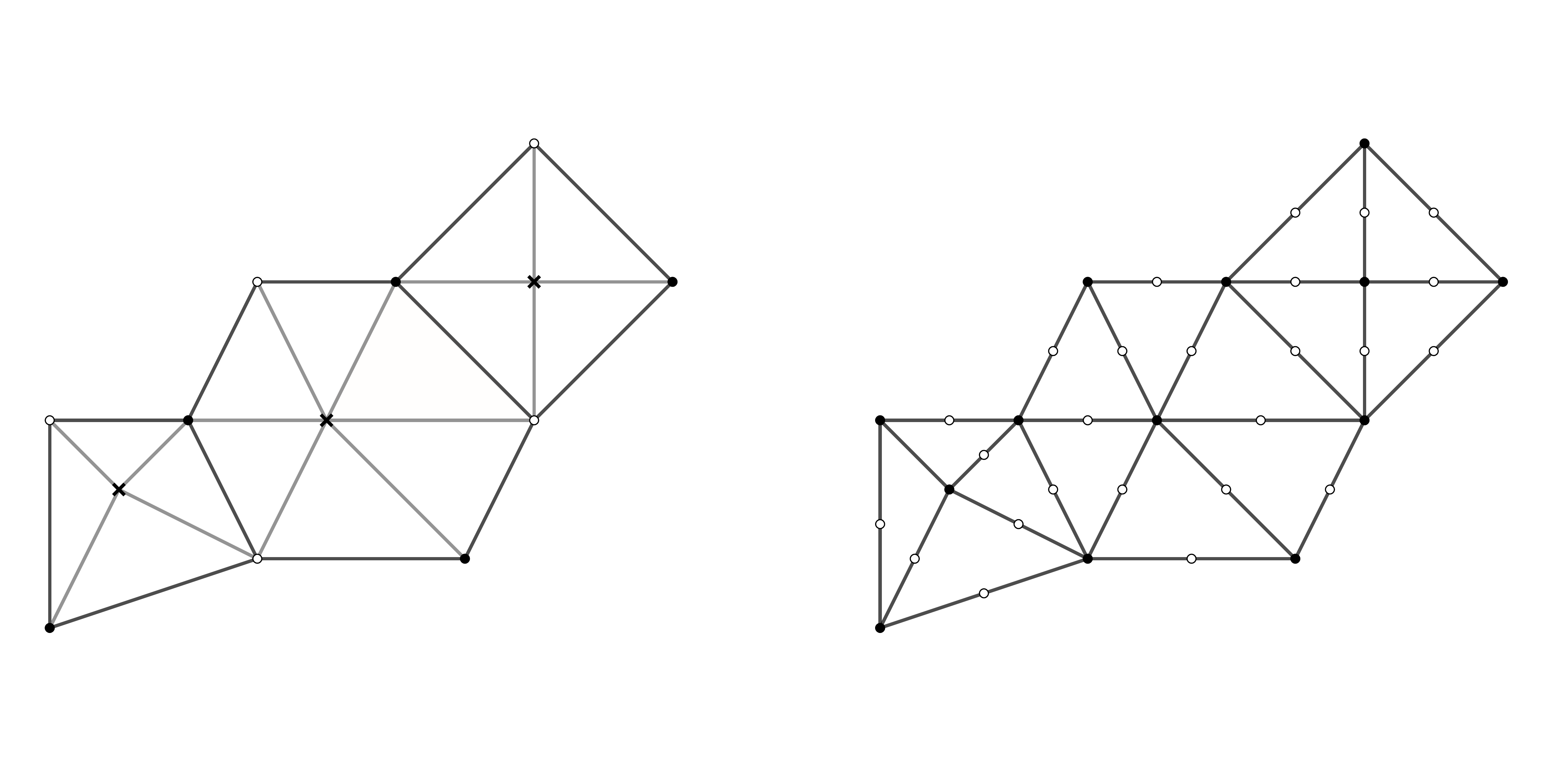} 
	\caption{Dessin d'enfant of $1-1/j(f)$.}   
	\label{dessin1-1jf}
	\end{center}
\end{figure}

Therefore the dessin of $1-1/j(f)$ is precisely the dessin of the Belyi function corresponding to the barycentric
subdivision. Hence such functions are equivalent. Summarizing we have the following proposition:

\begin{proposition}
	Let $X$ be a compact, connected Riemann surface obtained from a decorated
	 equilateral triangulation and let $f\colon X\to 
	\hat{\mathbb{C}}$ be the corresponding Belyi function. 
	Then the Belyi function associated to the barycentric subdivision
	 (decorated as in Corollary \ref{subdivision-baricentrica}) is 
	equivalent to $1-1/j(f)$.
\end{proposition}

\section{Connected sum of two Belyi functions}

\subsection{Definition of the connected sum of two Bely functions} 
Let $f_1\colon X_1\to \hat{\mathbb{C}}$ and $f_2\colon X_2\to 
\hat{\mathbb{C}}$ be two Belyi functions and consider their corresponding equilateral decorated triangulations.
Let $T_1^+$ be a black triangle of the triangulation of
$f_1$ and $T_2^-$ a white triangle of the triangulation of  $f_2$. 
We can construct the connected sum of
 $X_1$ with $X_2$ by removing the interior of $T_1^+$ and the interior of $T_2^-$ 
and glue the boundaries of  $T_1^+$ and $T_2^-$ by the homeomorphism
 $f_2^{-1}\circ f_1\colon \partial T_1^+\to \partial 
T_2^{-}$. Denote such surface as
  $X_{1}\,{}_{_{T_1^+}}\#_{_{T_2^-}}\,X_2$.

The function $f_1\,{}_{_{T_1^+}}\#_{_{T_2^-}}\,f_2\colon 
X_1\,{}_{_{T_1^+}}\#_{_{T_2^-}}\,X_2\to \hat{\mathbb{C}}$ given by the formula 

\begin{equation}\label{suma-conexa-funcion}
	f_1\,{}_{_{T_1^+}}\#_{_{T_2^-}}\,f_2(x)=
	\begin{cases}
		f_1(x), \ x\in X_1-\operatorname{int} T_1^+\\
		f_2(x),\ x\in X_2-\operatorname{int} T_2^-,
	\end{cases}
\end{equation}

\noindent is well-defined and continuous.

\par By construction, the surface
 $X_1\,{}_{_{T_1^+}}\#_{_{T_2^-}}\,X_2$ has a decorated equilateral structure
 induced by the equilateral structures of
 $X_1$ and $X_2$ and with this structure it becomes a Riemann surface. 
The function 
(\ref{suma-conexa-funcion}) is holomorphic outside of $\partial 
T_1^+=\partial T_2^-$ in $X_1\,{}_{_{T_1^+}}\#_{_{T_2^-}}X_2$, 
and by continuity is holomorphic in the whole connected sum.
Furthermore the function (\ref{suma-conexa-funcion}) is the function that realizes
the triangulation in the connected sum.

\begin{definition}
	Consider the Belyi functions $f_1$ and $f_2$ and their associated triangulations.
	Given an edge $a_1$ of the dessin of
	$f_1$ and an edge $a_2$ of the dessin of $f_2$, 
	there exists a unique black triangle
	$T_1$ which contains $a_1$  and a unique white triangle which contains
	 $a_2$, therefore we can define $X_1\,{}_{_{a_1}}\#_{_{a_2}}\,X_2$ 
	 as $X_1\,{}_{_{T_1}}\#_{_{T_2}}\, X_2$. 
\end{definition}

\begin{remark}
	We will see in \ref{monodromia-suma} that the connected sum of
	Belyi functions depends upon the chosen triangles used to perform the connected sum,
	in other words the ramified coverings corresponding to the connected sums
	could be non-isomorphic if we change the decorated triangles.
	
	We don't know if the conformal class of the Riemann surface
	$X_1\,{}_{_{T_1^+}}\#_{_{T_2^-}}\,X_2$ changes if we change the two decorated triangles 
	along which the connected sum is performed.
	
\end{remark}

\subsection{Monodromy of the connected sum of two Belyi
         functions}\label{monodromia-suma} 
         Let $\sigma_0$, $\sigma_1$ be the cyclic permutations
         around the vertices of type $\circ$ and $\bullet$, respectively, for
         the function $f_1$, and let $\sigma_0'$, $\sigma_1'$ the corresponding permutations
         of $f_2$. Let us recall that such permutations determine the monodromy of the Belyi functions.
         Denote by
         $\sigma_0\,{}_{_{T_1^+}}\#_{_{T_2^-}}\,\sigma_0'$ the permutation of the vertices $\circ$ of
          $f_1\,{}_{_{T_1^+}}\#_{_{T_2^-}}\,f_2$, and by
         $\sigma_1\,{}_{_{T_1^+}}\#_{_{T_2^-}}\,\sigma_1'$ for the vertices $\bullet$. 
         
         Let us compute this permutations in order to determine the monodromy of
          $f_1\,{}_{_{T_1^+}}\#_{_{T_2^-}}\,f_2$.

\par 
Let $P_0$ and $Q_0$ the set of vertices with labels $\circ$ and $\bullet$, respectively,
of the dessin determined by $f_1$ which are also vertices of  $T_1^+$. 
Analogously let $P_0'$ and $Q_0'$ be the vertices with labels $\circ$ and $\bullet$, respectively,
determined by the dessin of $f_2$ which are also vertices of triangle
$T_2^-$ (see Figure
(\ref{permutacion-vertices})).

\begin{figure}
		\begin{center}
		\includegraphics[width=12cm]{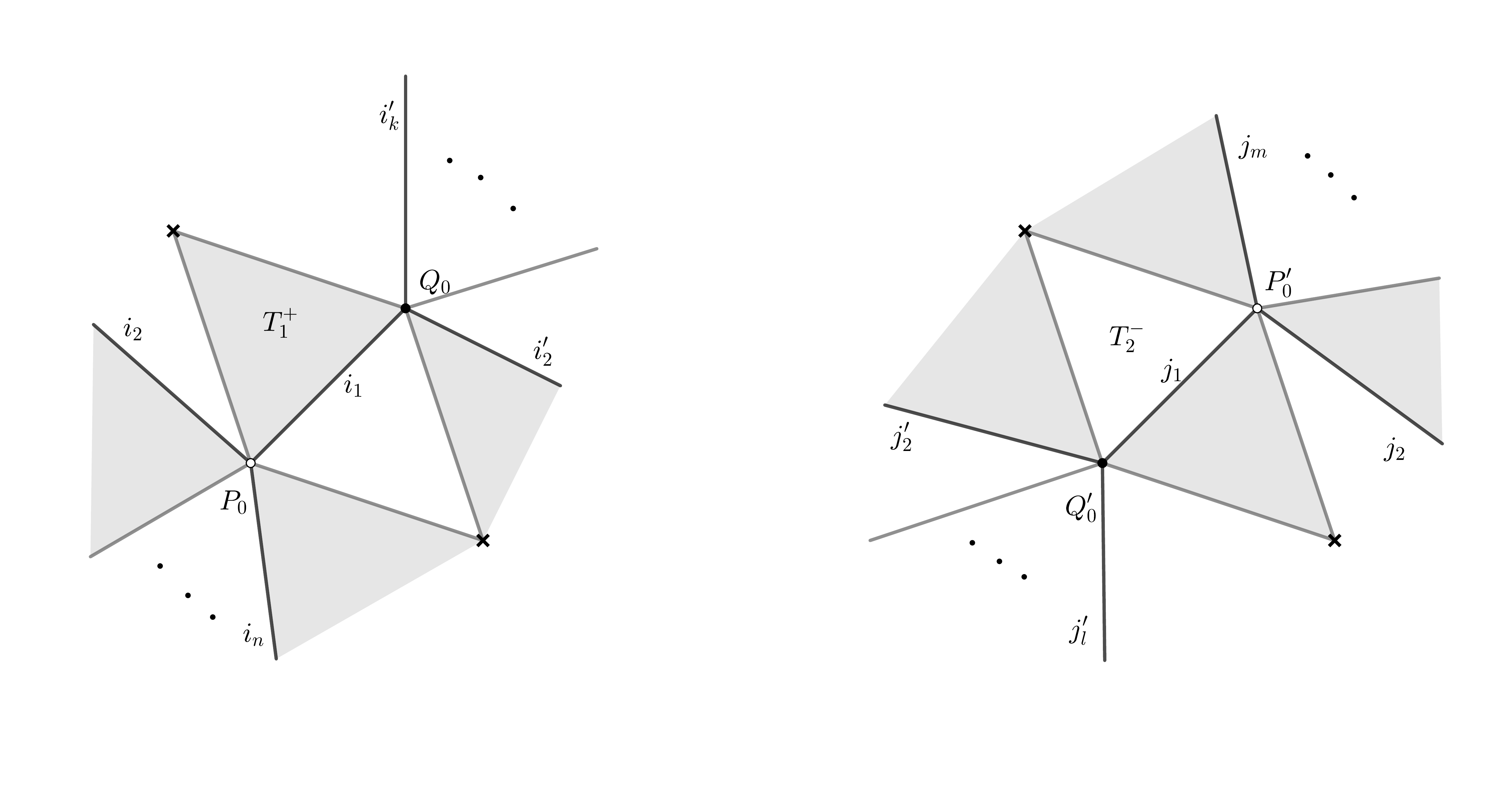} 
		\caption{Permutation around the vertices.}  
		\label{permutacion-vertices} 
		\end{center}
\end{figure}	

In Figure (\ref{permutacion-vertices}) it is indicated the edges around the vertices
   (the positive orientation is counter-clockwise) and thus the figure also indicates
   the cyclic permutations around the vertices.
   Then we obtain:
   
 \begin{itemize}
\item[(i)] The cycle around
 $P_0=P_0'$, in the connected sum, 
 is given by the equation:
 
 \begin{equation*}
	(j_1\ j_2\ \ldots\ j_m\ i_2 \ \ldots\ i_n).
\end{equation*}

Since the other vertices with label $\circ$ remained unchanged when we make the connected sum
the rest of the cycles of
$\sigma_0\, {}_{_{T_1^+}}\#_{_{T_2^-}}\,\sigma_0'$ remain the same. 

\item[(ii)] 
Analogously for $Q_0=Q_0'$ in the connected sum, the cycle around it is

\begin{equation*}
	(i_1\ i_2'\ \ldots\ i_k'\ j_2'\ldots j_l'),
\end{equation*}
and the rest of the cycles of $\sigma_1\,{}_{_{T_1^{+}}}\#_{_{T_2^-}}\,\sigma_1'$ remain the same. 
\end{itemize}

Hence the permutations and monodromy
of the  Belyi function of the connected sum is completely described.

\begin{remark}

(i) From the previous analysis we obtain the following formulas 
of the multiplicities and degree:

\begin{eqnarray*} 
\operatorname{mult}_{P_0}(f_1\,{}_{_{T_1^+}}\#_{_{T_2^-}}\,f_2) 
& = & \operatorname{mult}_{P_0}f_1+\operatorname{mult}_{P_0'}f_2-1\\ 
\operatorname{mult}_{Q_0}(f_1\,{}_{_{T_1^+}}\#_{_{T_2^-}}\,f_2) 
& = & \operatorname{mult}_{Q_0}f_1+\operatorname{mult}_{Q_0'}f_2-1\\ 
\deg(f_1\,{}_{_{T_1^+}}\#_{_{T_2^-}}\,f_2) & =& \deg f_1+\deg f_2-1
\end{eqnarray*} 
\end{remark}

The connected sum of two Belyi functions has as associated 
dessin d'enfant, it is  obtained by the fusion of the corresponding edges 
of $T_1^+$ and $T_2^-$ which belong to the corresponding dessins.
\begin{example}\label{si-depende}

Consider the polynomials

\begin{equation}
	f_1(z)=\frac{4^4}{3^3}z(1-z)^3,\quad f_2(z)=4z(1-z)
\end{equation}

whose dessin are in Figure (\ref{dos-polinomios}). In Figure 
(\ref{dos-dessin-diferentes}) we show the graphs corresponding to
$f_1\,{}_{_{T_1^+}}\#_{_{T_2^-}}\,f_2$ and $f_1\,{}_{_{T_2^+}}\#_{_{T_2^-}}\,f_2$. 

\begin{figure}
		\begin{center}
		\includegraphics[width=12cm]{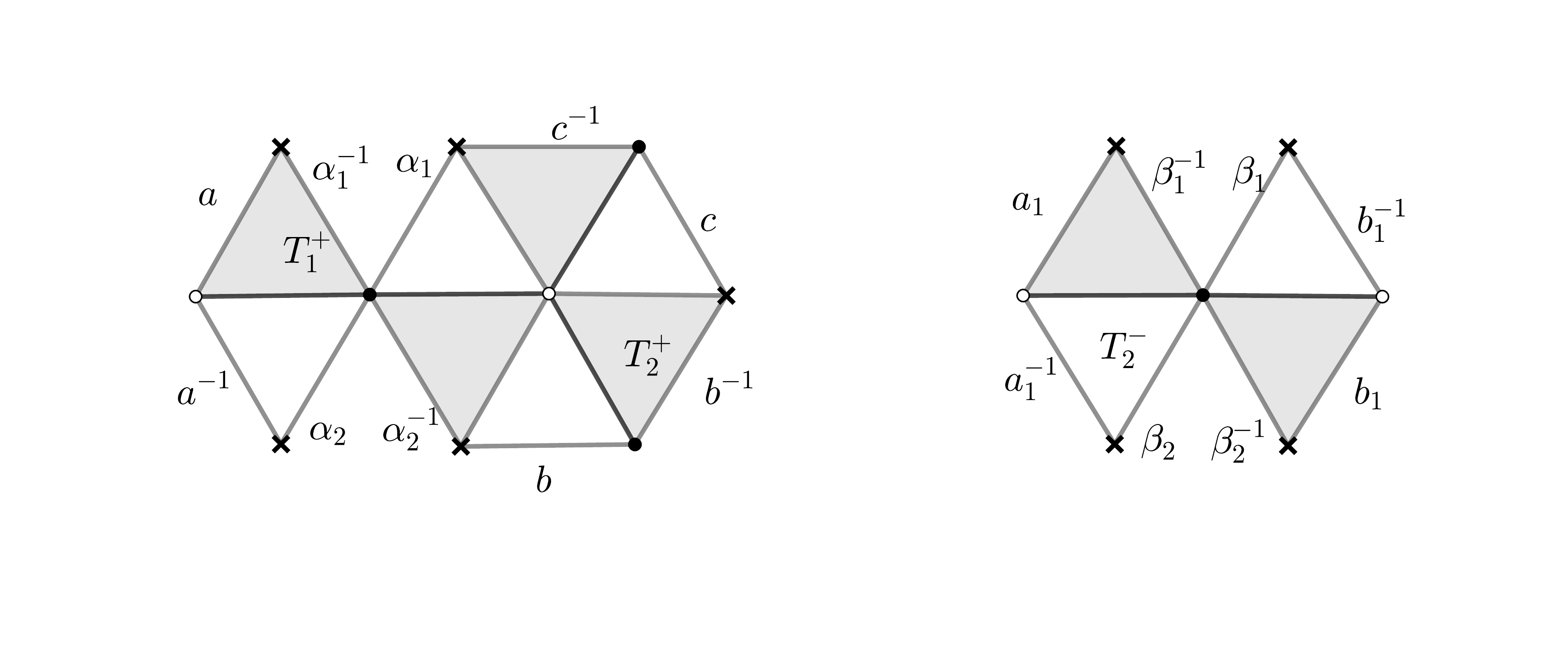} 
		\caption{Dessins of $f_1(z)=\frac{4^4}{3^3}z(1-z)^3,\quad 
		f_2(z)=4z(1-z)$.}   
		\label{dos-polinomios}
		\end{center}
\end{figure}

\begin{figure}
		\begin{center} 
		\includegraphics[width=10cm]{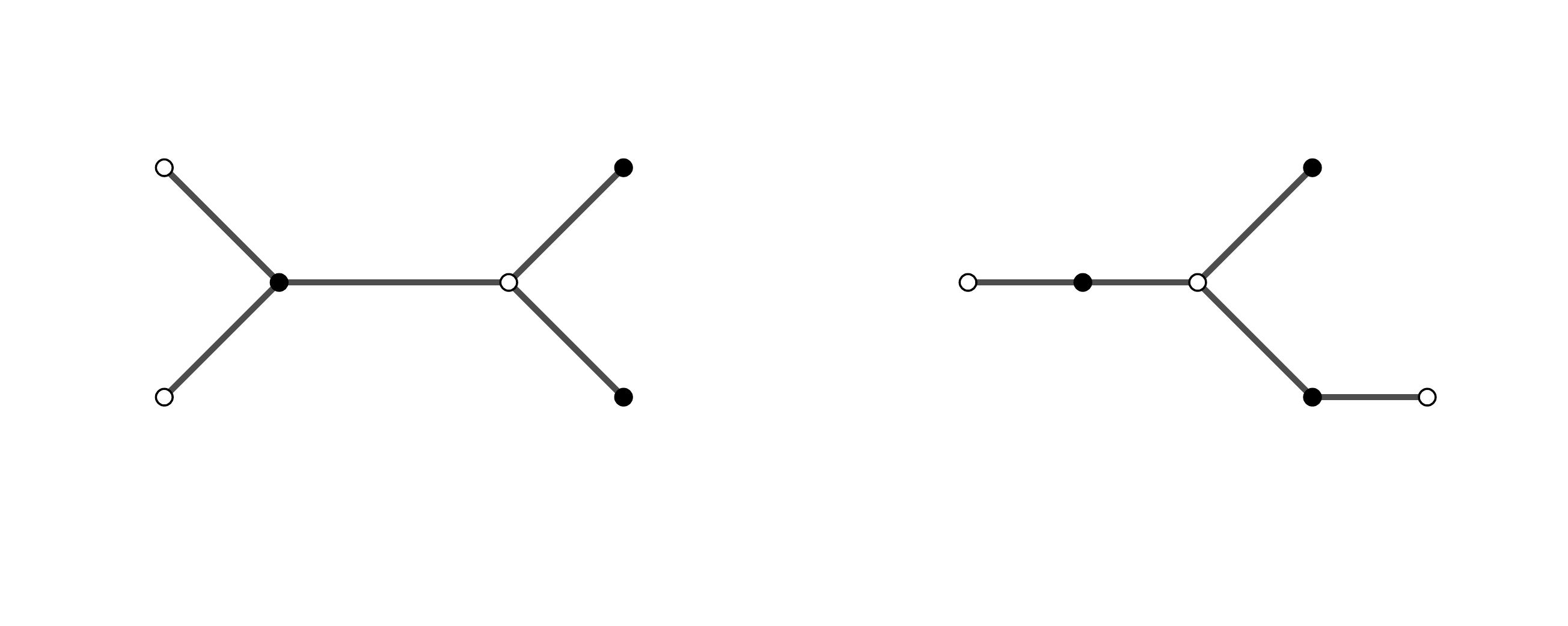} 
		\caption{Graphs corresponding, respectively, to the dessins of $f_1\,{}_{_{T_1^+}}\#_{_{T_2^-}}f_2$ and 
		$f_1\,{}_{_{T_2^+}}\#_{_{T_2^-}}\,f_2$.} 
		\label{dos-dessin-diferentes} 
		\end{center}
\end{figure}
\end{example}

From the previous formulas we see that, in general, 
the connected sum of Belyi functions depends upon the chosen triangles used in the connected sum (see Example 
\ref{si-depende}). In \ref{suma-conexa-Galois} we give a condition in order to have the connected sum independent of the
 triangles.

\par In the following example we describe the connected sum of the $j$-invariant with itself.

\begin{example}[Connected sum $j\# j$]
Consider the 
dessin of the $j$-invariant (see Figure 
(\ref{TriangulacionEquilateraj})), and the connected sum
$j\,{}_{_{T_1^+}}\#_{_{T_2^-}}\, j$, where $T_1^+$ and $T_2^-$ are indicated in 
Figure (\ref{suma-conexa-j}) with dotted lines. 

\begin{figure}
		\begin{center}
		\includegraphics[width=12cm]{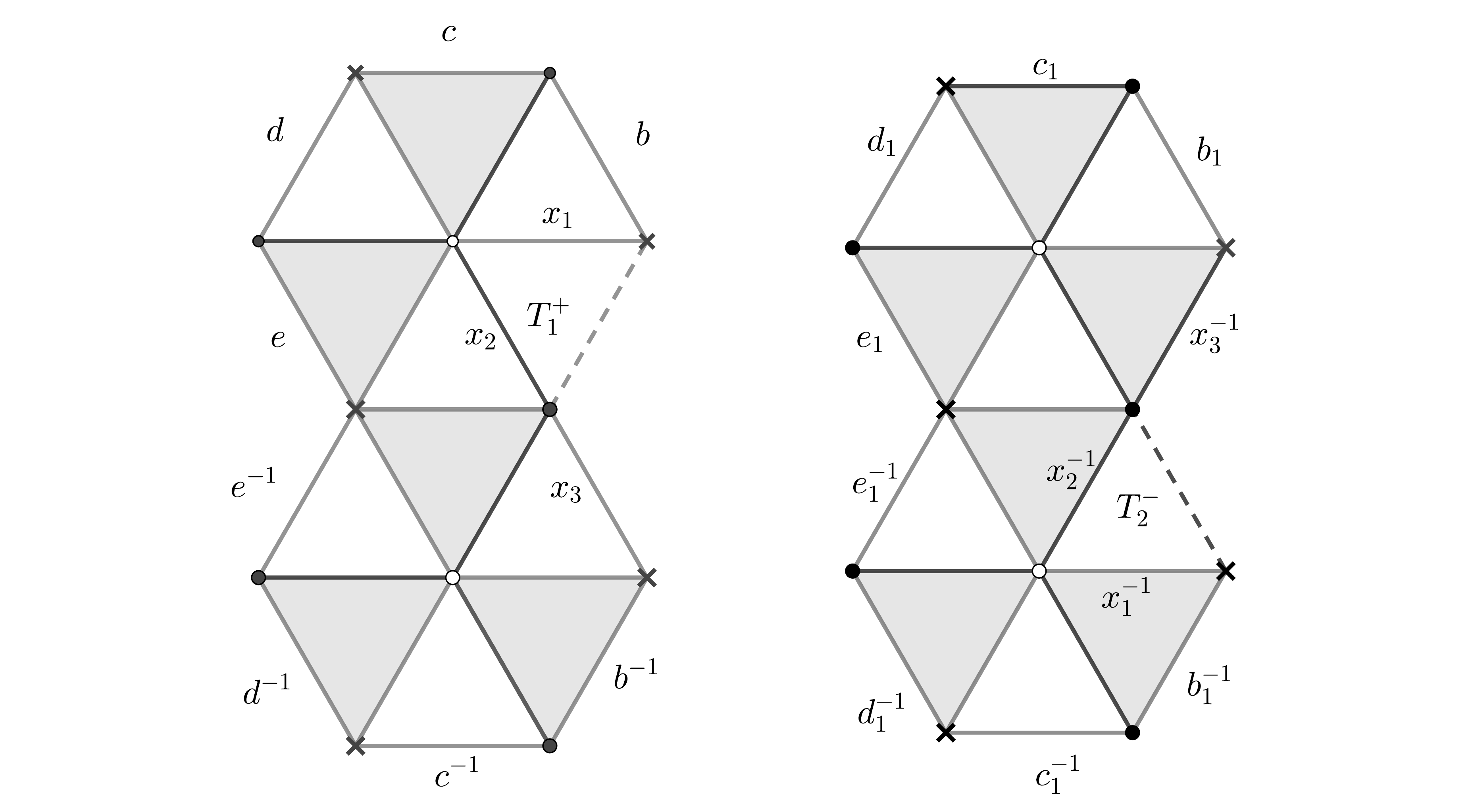} 
		\caption{Dessin of $j\,{}_{_{T_1^+}}\#_{_{T_2^-}}\, j$.}   
		\label{suma-conexa-j}
		\end{center}
\end{figure}
 When we glue to obtain the connected sum we obtain topologically the Riemann sphere
 with the decoration shown in Figure (\ref{suma-conexa-j-top}). 

\begin{figure}
		\begin{center}
		\includegraphics[width=10cm]{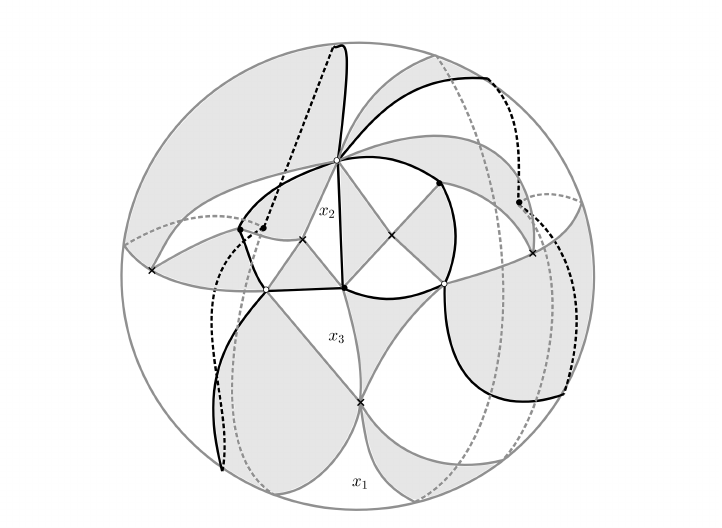} 
		\caption{Visualization in the sphere of the connected sum $j\,{}_{_{T_1^+}}\#_{_{T_2^-}}\, j$.}   
		\label{suma-conexa-j-top}
		\end{center}
\end{figure}

\end{example}

\subsection{Connected sum of Belyi functions which are Galois.}\label{suma-conexa-Galois}
Let us recall that a meromorphic function
 $f\colon X\to \hat{\mathbb{C}}$ is 
of \emph{Galois} type or \emph{regular} if its group of deck transformations
 $\operatorname{Deck}(f)$ acts transitively on the fibers of the regular values of
  $f$ (see \cite{Forster}). Equivalently the unbranched covering, when we remove both
  the critical values and their pre-images, is a regular (or Galois) covering.

\begin{remark}\label{Observacion-Galois}
	If  $f\colon X \to \hat{\mathbb{C}}$ is a Galois Belyi function, then 
	the group of deck transformations $\operatorname{Deck}(f)$ acts transitively
	on the triangles of the same color and preserves the decoration of the vertices.
	
\end{remark}

\begin{proposition}
	If $f_1\colon X_1\to \hat{\mathbb{C}}$ and $f_2\colon X_2\to 
	\hat{\mathbb{C}}$ are two Belyi functions which are of Galois type
	Then, given $T_1^+$, $T_1'^{+}$ two black triangles of
	$X_1$ and  $T_2^-$, $T_2'^{-}$ two white triangles of $X_2$, then there exists a
	biholomorphism 
	
	\begin{equation}
		\Phi\colon X_1\, {}_{_{T_1^+}} \#_{_{T_2^-}}\, X_2\to X_1\,{}_{_{T_1'^{+}}} 
		\#_{_{T_2'^{-}}}\,X_2	
	\end{equation}  
	
	\noindent for which the following diagram is commutative	
	\begin{equation}\label{diagrama-suma-conexa}
		\xymatrix{X_1\,{}_{_{T_1^+}}\#_{_{T_2^-}}\,X_2\ar^{\Phi}[rr]
		\ar_{f_1\,{}_{_{T_1^{+}}} 
		\#_{_{T_2^{-}}}\,f_2}[rd] & & X_1\,{}_{_{T_1'^{+}}} 
		\#_{_{T_2'^{-}}}\,X_2\ar^{f_1\,{}_{_{T_1'^{+}}} \#_{_{T_2'^{-}}}\,f_2}[ld]\\ & 
		\hat{\mathbb{C}}&}
	\end{equation}
\end{proposition}

\begin{proof}
	By Remark \ref{Observacion-Galois} there exists $\sigma_1\in 
	\operatorname{Deck}(f_1)$ and $\sigma_2\in \operatorname{Deck}(f_2)$ 
	such that $\sigma_1(T_1^+)=T_1'^+$ and $\sigma_2(T_2^-)=T_2'^-$ which preserve 
	the decoration of the vertices (see Figure 
	(\ref{deck-triangulos})).
	
	\begin{figure}
		\begin{center}
		\includegraphics[width=12cm]{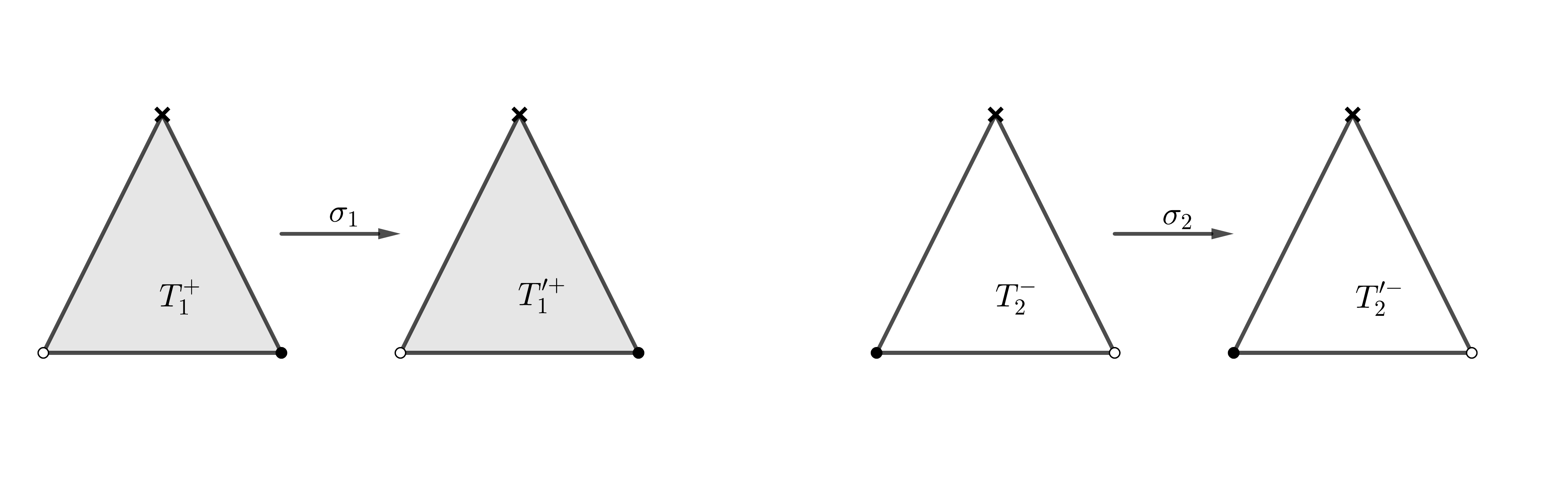}
		\caption{Deck transformations preserve the decoration.}  
		\label{deck-triangulos} 
		\end{center}
	\end{figure}
	
	Let us define the map
	\begin{equation}
		\sigma_1\,{}_{_{T_1^+}}\#_{_{T_2^-}}\,\sigma_2\colon X_1\,{}_{_{T_1^+}} 
		\#_{_{T_2^-}}\, X_2 \to X_1\,{}_{_{T_1'^{+}}} 
		\#_{_{T_2'^{-}}}\,X_2
	\end{equation}
	as follows:
	 \begin{equation}
		\sigma_1\,{}_{_{T_1^+}}\#_{_{T_2^-}}\,\sigma_2(x)=
		\begin{cases}
			\sigma_1(x),\ x\in X_1-\operatorname{int} T_1^+\\
			\sigma_2(x),\ x\in X_2-\operatorname{int} T_2^-. 
		\end{cases}
	\end{equation} 
	
	\noindent This function is well-defined because if
	 $x\in \partial T_1^+$, $y\in \partial T_2^-$ are such that
	 tal que $x\sim y$, then
	$f_1(x)=f_2(y)$. Hence $f_1(\sigma(x))=f_2(\sigma(y))$, therefore 
	$\sigma(x)\sim \sigma(y)$. In addition,  
	$\sigma_1\,{}_{_{T_1^+}}\#_{_{T_2^-}}\,\sigma_2$ is a homeomorphism. 
	
	Clearly $\sigma_1\,{}_{_{T_1^+}}\#_{_{T_2^-}}\,\sigma_2$ is conformal in the complement of
	 $\partial T_1^+=\partial T_2^-$ and, by continuity, such function can be extended
	 conformally to all of the connected sum.

	\par The diagram (\ref{diagrama-suma-conexa}) commutes because
	\begin{eqnarray*}
		(f_1\,{}_{_{T_1'^{+}}} \#_{_{T_2'^{-}}}\,f_2)\circ 
		(\sigma_1\,{}_{_{T_1^+}}\#_{_{T_2^-}}\,\sigma_2) (x) & = & (f_1 \circ  
		\sigma_1)\,{}_{_{T_1^{+}}} \#_{_{T_2^{-}}}\,(f_2\circ \sigma_2)(x)\\ & = & 
		f_1\,{}_{_{T_1^{+}}} \#_{_{T_2^{-}}}\,f_2(x).
	\end{eqnarray*}
\end{proof}

\subsection{Connected sum of the polynomials $z^m$.}
Since the monomial functions of the form 
$z^m$, with $m\geq 1$, are of Galois type
their connected sum does not depend on the triangles chosen to do the connected sum. 
Thus in this case we denote the connected sum simply as $z^m\# z^n$. 

\par From what was discussed before one has that the connected sum
$z^m\# z^n$ has the same monodromy as that of the function
 $z\mapsto{z^{m+n-1}}$. Hence, up to isomorphism of ramified coverings
 one has: 

\begin{equation} 
z^m\# z^n= z^{m+n-1}.
\end{equation}

\par Therefore we have the following properties:

\begin{itemize}
	\item[(i)] Associativity: 
	$(z^m\# z^n)\# z^k=z^{m+n+k-2}=z^m\#(z^n\# 	z^k)$.
	\item[(ii)] There exists a neutral element: the polynomial $z$.
	\item[(iii)] Commutativity.
\end{itemize}

Then the Belyi functions of the form $z^m$ is, under connected sum, a
 {\em commutative monoid} (under equivalence relations of ramified coverings). 
 Figure (\ref{ejemplo-z3-z2}) illustrates the connected sum of $z^3$ with $z^2$.

\begin{remark}
	The connected sum of two polynomials of the form $z^m$, with
	$m$ odd is a closed operation.
\end{remark}

\begin{figure}
		\begin{center}
		\includegraphics[width=12cm]{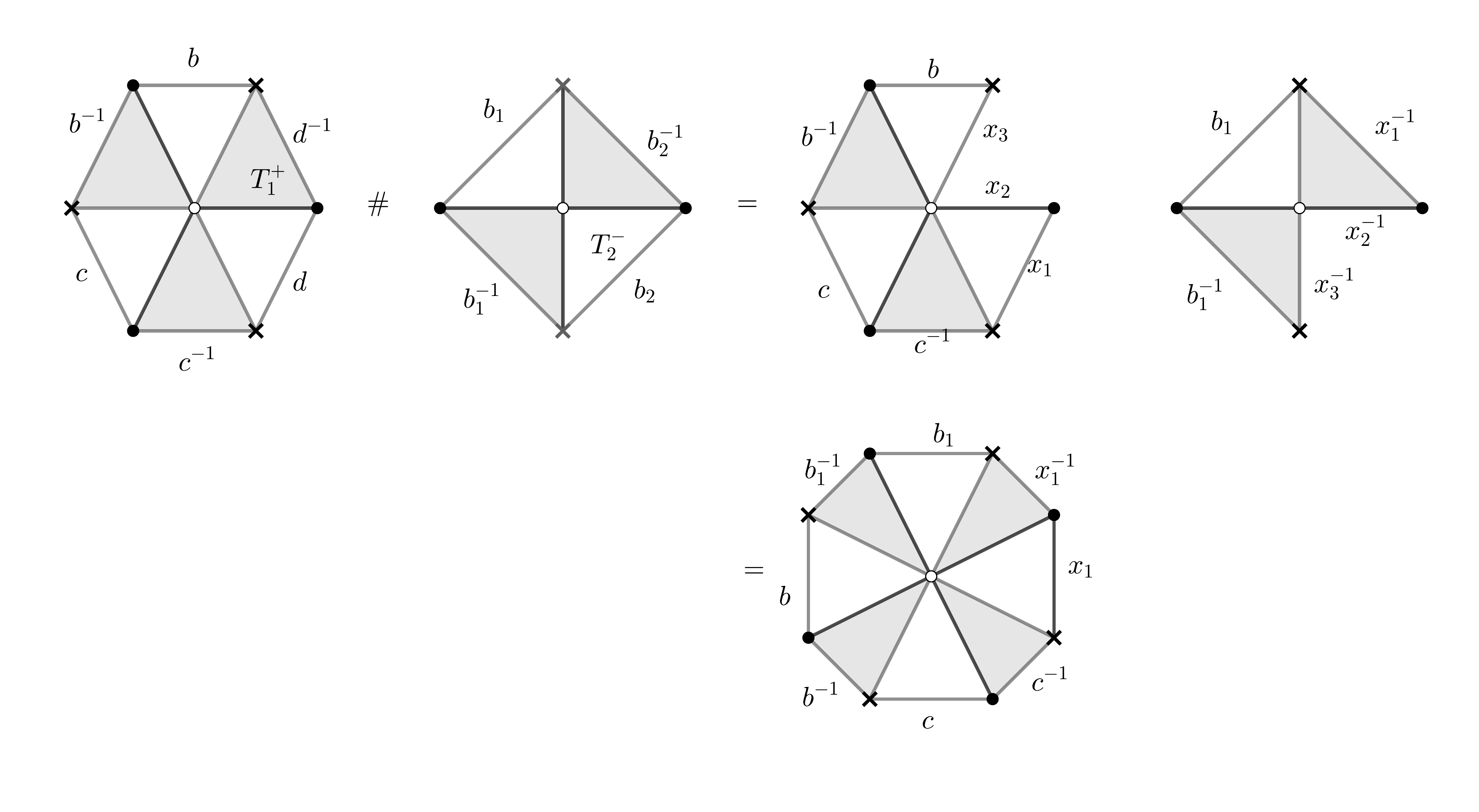} 
		\caption{Connected sum of $z^3$ with $z^2$, represented 
		topologically by identifying polygons.}  
		\label{ejemplo-z3-z2} 
		\end{center}
\end{figure}


\begin{remark}[Shabat polynomials]
	A polynomial which is a Belyi function has as dessin d'enfant associated
	a bicolored tree and reciprocally, given a bicolored tree there exists a Belyi function that realizes it.
	Such polynomials are called \emph{Shabat polynomials} (see \cite{ShZv94}). 
	
	\par The dessin of a connected sum of two Shabat polynomials is
	again a bicolored tree since the connected sum does not create cycles
	in the dessin of the connected sum. Hence the connected sum of two
	Shabat polynomials is again a Shabat polynomial (under the isomorphism class of branched coverings)
	
\end{remark}

\subsection{Connected sum of double-star polynomials.} Double-star polynomials
are graphs like the one in self-explicatory Figure (\ref{double-star}). 
These dessins can be realized by polynomials of the form:

\begin{figure}
		\begin{center}
		\includegraphics[width=10cm]{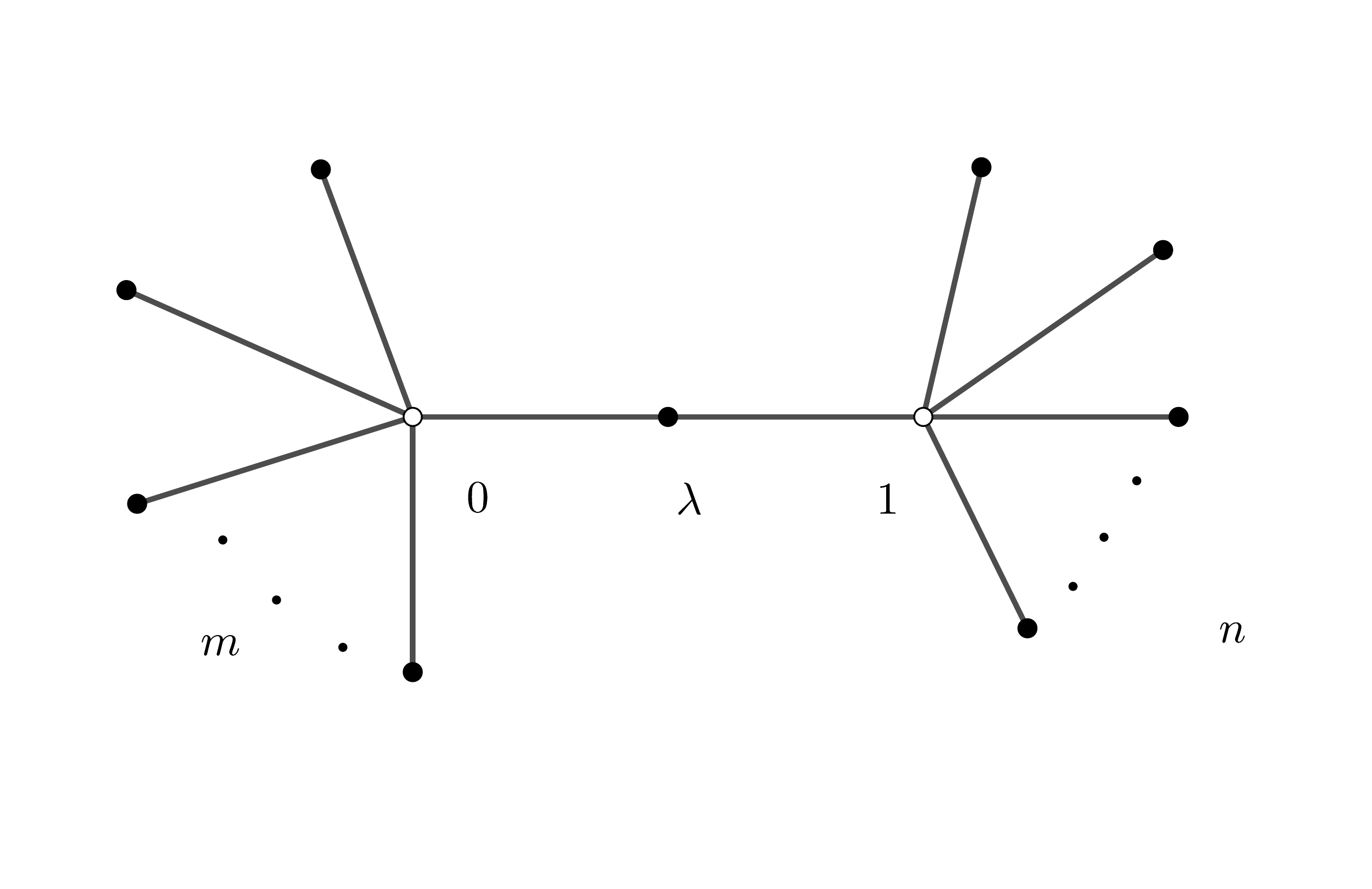} 
		\caption{Double-star tree corresponding to $P_{m,n}$.}  
		\label{double-star} 
		\end{center}
\end{figure}

\begin{equation}
	P_{m,n}(z)=\mu z^m(1-z)^n,\quad 
	\mu=\frac{(m+n)^{m+n}}{m^mn^n},\quad \lambda=\frac{m}{m+n},\quad 
	m,n\geq 1.
\end{equation}

Let $a$ and $b$ denote the intervals $[0,\lambda]$ and
$[\lambda,1]$, respectively, and $a_1$, $b_1$ the intervals 
$[-1,0]$ and $[0,1]$ respectively. Then:
\begin{equation}
z^m\,{}_{a}\#_{a_1}\,(1-z^2)\,{}_{b_1}\#_{b}\, z^n=P_{m,n}(z)  \quad \text{(because they have the same dessin.)}
\end{equation}

\subsection{Connected sum of Tchebychev polynomials}
Recall that \emph{Tchebychev polynomials} are those polynomials $\bar{T}_n\,\, (n\in\mathbb{Z}, n\geq0)$ that satisfy
the functional identity:
\begin{equation}
	\bar{T}_n(\cos\ z)=\cos \ nz.
\end{equation}

Using De Moivre's formula it is possible to calculate them explicitly. 
In fact they have the satisfy the recursive formula
for $n=0$, $\bar{T}_0(x)=1$. For 
$n\geq 1$

\begin{equation}
	\bar{T}_n(x)=2x\bar{T}_{n-1}(x)-\bar{T}_{n-2}(x).
\end{equation}
This implies that $\bar{T}_n$ has degree $n$.

\par The critical points of the polynomial 
 $\bar{T}_n(x)$ are $\cos 
k\pi/n$, with $k=1,\ldots,n-1$, therefore the critical values are

\begin{equation}
	\bar{T}_n\left(\cos\ \frac{k\pi}{n}\right)=\cos\ k\pi=
	\begin{cases}
		-1, \ k\ \text{odd}\\
		1,\ k\ \text{even}, \quad k=1,\ldots, n-1.
	\end{cases}
\end{equation}

Hence $\bar{T}$ has only three critical values. The function

\begin{equation}
T_n(z) = \frac{1}{2}(\bar{T}_n(z)+1)
\end{equation}

is a Belyi function for each $n\geq 1$ with critical values 0, 1 and $\infty$.

The dessins of these functions are depicted in Figure (\ref{linear-trees}).

\begin{figure}
		\begin{center}
		\includegraphics[width=12cm]{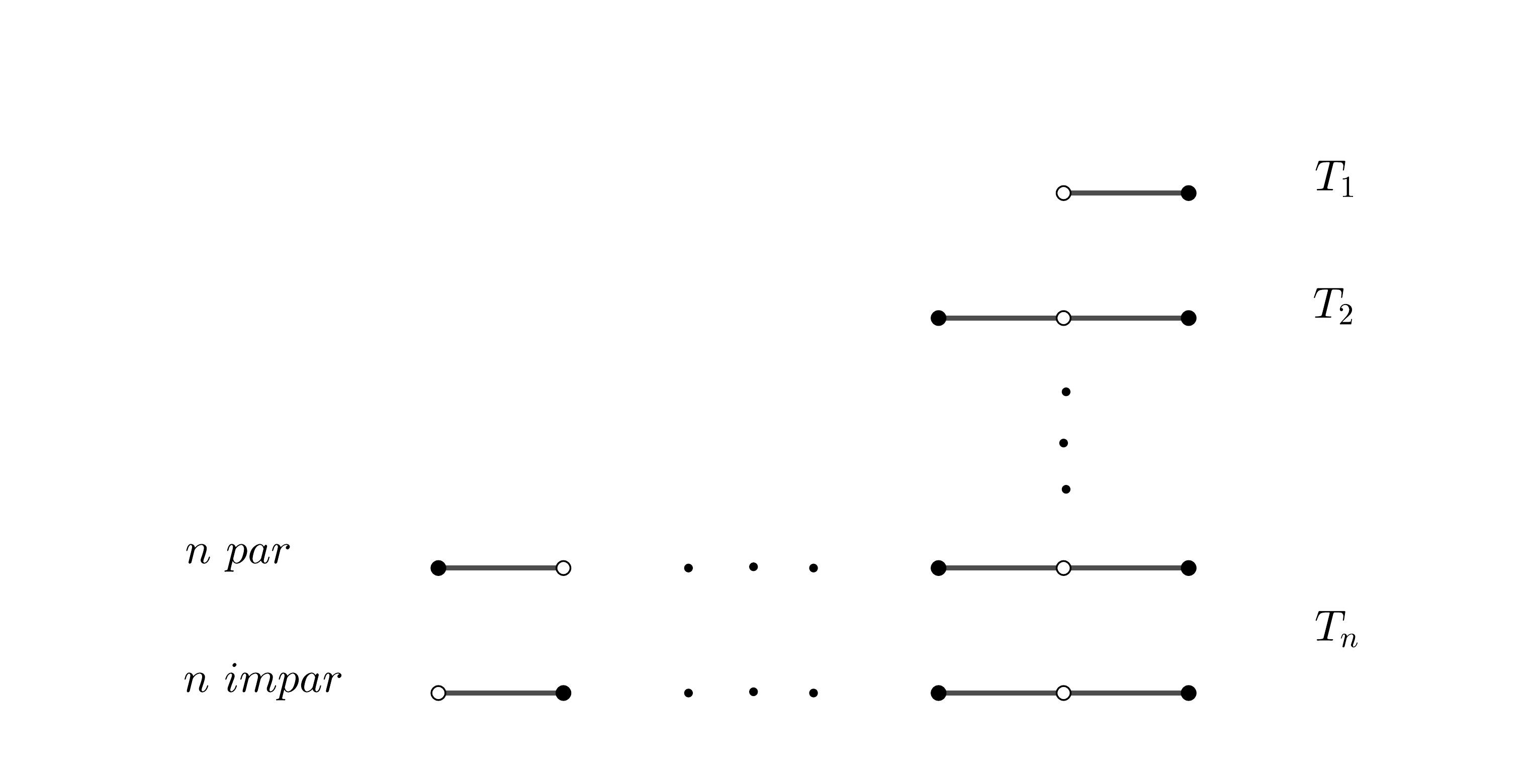} 
		\caption{Dessins of Tchebychev polynomials ${T}_n$. 
		All these 
		dessins have as extreme points
		$-1$ and $1$, and vertices 
		in $cos\ k\pi/n$.}
		\label{linear-trees} 
		\end{center}
\end{figure}

The dessins of the Belyi Tchebychev polynomials
$T_k$, with $n\geq 2$, have two distinguished edges, namely those at the extremes.
From these dessins of each connected sum we obtain the following identities:

\begin{equation}\label{identidades-Tchebychev}
\begin{cases}
	T_l\# T_m = T_{l+m-1}, \quad m\ \text{odd},\\
	(1-T_l)\# T_m = T_{l+m-1},\quad m\ \text{even},\\
	T_l\# (1-T_m) = 1-T_{l+m-1}, \quad m \ \text{even},
\end{cases}
\end{equation}

\noindent where the connected sum is taken with respect to the extremal edges;
the one on the right for the first function and the one on the left for the second function. 
From (\ref{identidades-Tchebychev}) one obtains:

\begin{proposition}
	If $n$ is odd $T_2\# T_n=T_{n+1}$, if $n$ even $(1-T_2)\# 
	T_n=T_{n+1}$. Hence, 
	
	\begin{equation}
		 \begin{cases} 
		T_n=(1-T_2)\# \cdots \# T_2\# (1-T_2)\# T_2, \ n\ 
		\text{impar},\\
		T_n=T_2\# \cdots \# T_2\# (1-T_2)\# T_2, \ n\ 
		 \text{par}, 
		\end{cases} 
	\end{equation}
	where both sums consist of
	 $n-1$ summands. Equivalently, 
	
	\begin{equation}\label{Tchebychev-potencias}
	\begin{cases}
		T_n= ((1-T_2)\# T_2)^{\#\frac{n-1}{2}}=T_3^{\#\frac{n-1}{2}}, \ n\ 
		\text{impar},\\
		T_n = T_2\# T_{n-1}=T_2\# (T_3^{\#\frac{n-2}{2}}),\ n\ \text{par}.
	\end{cases}
	\end{equation}
\end{proposition}


\begin{remark}
	From (\ref{identidades-Tchebychev}) it follows that the connected sum is a closed operation
	for the Tchebychev polynomials of odd degree. In addition they form a commutative monoid. 
	 From (\ref{Tchebychev-potencias}) it follows that this monoid is generated by $T_3$.
\end{remark}

\subsection{Connected sum of surfaces with an equilateral triangulation}

Let $X_1$ and $X_2$ be Riemann surfaces obtained by equilateral triangulations. Let $T_1$ and $T_2$ 
be two triangles in $X_1$ and
$X_2$, respectively. Suppose the triangles are decorated with the symbols
 $(\circ,\bullet,\ast)$ in such a way that such a decoration induces a positive orientation
in $T_1$ and a negative orientation in $T_2$. 
Consider the connected sum
$X_1\,{}_{_{T_1}}\#_{_{T_2}}\,X_2$,
where the gluing homeomorphism is the isometry from the boundary of
 $T_1$ to the boundary of $T_2$ which preserves the decoration.
 As explained before,
  $X_1\, {}_{_{T_1}}\#_{_{T_2}}\,X_2$ has also a decorated equilateral structure, 
which gives the connected sum the structure of an arithmetic Riemann surface. 
\begin{figure}
	\begin{center} 
	\includegraphics[width=12cm]{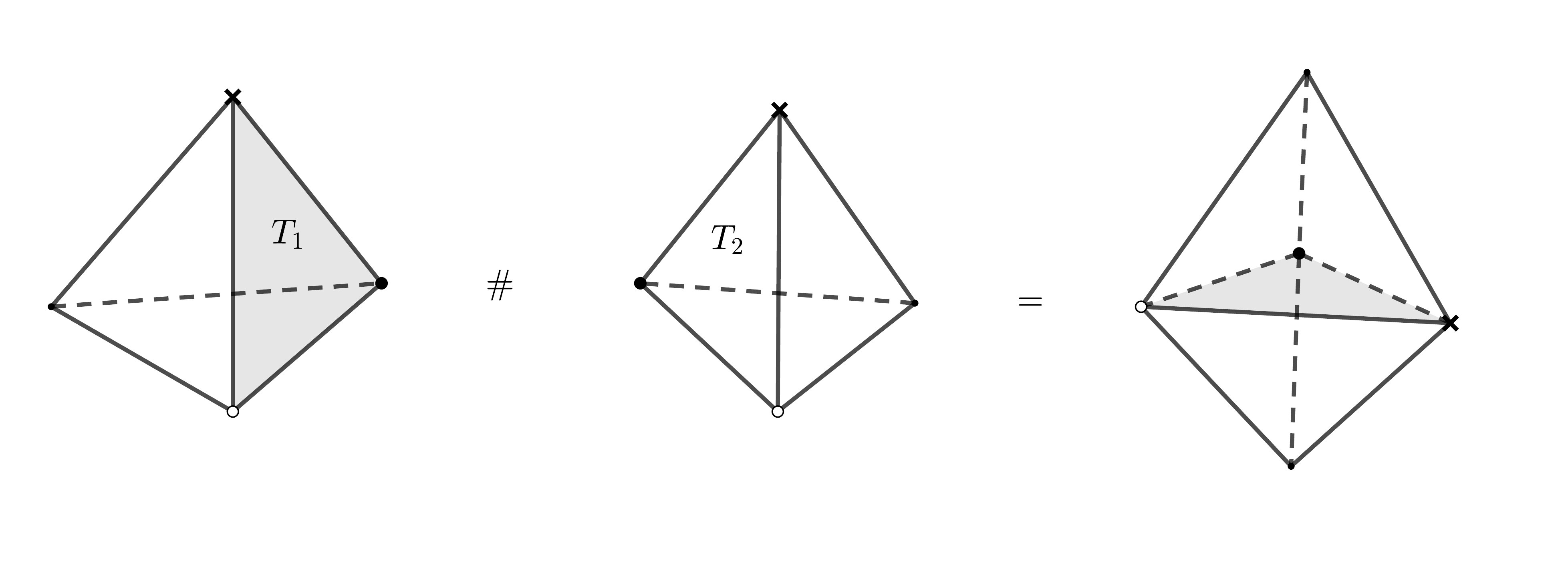} 
	\caption{Connected sum of two tetrahedra.} 
	\label{suma-dos-tetraedros} 
	\end{center}
\end{figure}

The connected sum defined in the previous paragraph is compatible 
with the definition of connected sum of Belyi functions since the gluing map
of $f_2^{-1}\circ f_1$ in the connected sum preserves the decoration
and it is an isometry which reverses the orientation of the boundaries since $f_2^{-1}\circ f_1 
(x)=f_2^{-1}\circ \bar{f_1}(x)$, for all $x\in \partial T_1$. 

Figure (\ref{suma-dos-tetraedros}) shows an example of the connected sum 
of two surfaces with equilateral triangulations.

\subsection{Flips} \label{flips}Another way to define the connected sum of two triangulated surfaces
can be done if instead of removing the interior of a triangle in each surface 
we remove rhombi \ie removing pairs of adjacent triangles.
Let $R_1$ and $R_2$ be two rhombi in $X_1$ and $X_2$, respectively, and
suppose that a triangle in each of the rhombi is decorated with the 
symbols $(\circ,\bullet,\ast)$ in such a way that the decoration induces a positive orientation in
$R_1$ and a negative orientation in $R_2$,  then
we can construct the connected sum  $X_1\,{}_{_{R_1}}\#_{_{R_2}}\, X_2$, 
where the gluing map is an isometry from
 $\partial{R_1}$ to $\partial{R_2}$ which respects the decoration of the triangles.  
 
Given a decorated equilateral triangulated structure on a surface $X$ 
if we perform the connected sum with a tetrahedron along a rhombus the triangulation of the
connected sum is obtained from the triangulation of $X$ by flipping the diagonals of the rhombus
(figure \ref{Flip-suma-conexa} shows an example). Any two simplicial triangulations of $X$ with the same, and sufficiently
large, number of vertices can be transformed into each other,
up to homeomorphism, by a finite sequence of diagonal flips \cite{BNN95}. If $P(X,n)$ denotes 
the set of simplicial triangulations of $X$ with $n$ vertices these triangulations are permuted by flippings. 
 \begin{question} What are the fields of definition of the algebraic curves corresponding to the triangulated surfaces of the different permutations under fliping?
 \end{question}

\begin{figure}
	\begin{center} 
	\includegraphics[width=12cm]{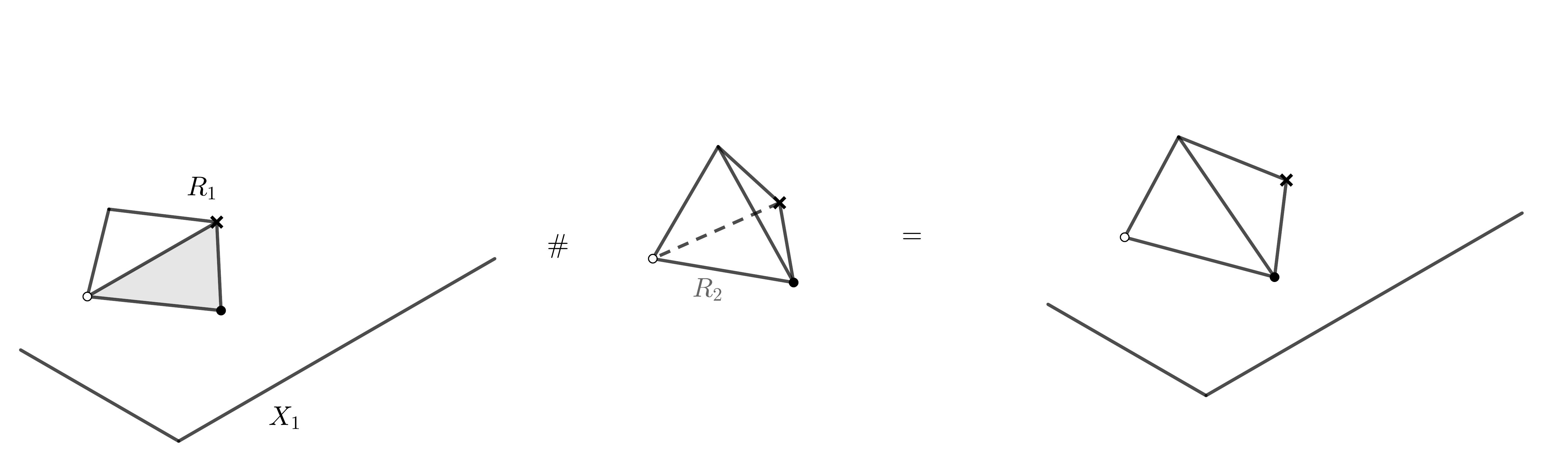} 
	\caption{Connected sum of a surface and a regular tetrahedron using rhombi.} 
	\label{Flip-suma-conexa} 
	\end{center}
\end{figure}

\subsection{Connected sum with tetrahedra and elementary sub-divisions (starrings)}\label{tetraedros-starrings}

In this subsection we will make some elementary observations on the relations between 
\emph{starrings} and connected sums of tetrahedra.
A triangulation on a surface can be subdivided (or refined) by means of some elementary operations
called \emph{starrings}. Starting from an equilateral triangulation 
we can apply this refinement and obtain a new equilateral triangulation 
(by declaring all triangles equilateral of the same size)
which induces a new flat metric with conic singularities. 

\begin{remark} In general we don't know how the complex structure changes nor the algebraic 
subfield of definition, as a subfield of $\bar{\mathbb Q}$.
\end{remark}

Let $X$ be a regular tetrahedron with its complex structure induced by the triangulation. 
Suppose that one of its triangles $T_2$ is negatively oriented with respect to the decoration 
$(\circ,\bullet,\ast)$. 
Let $X_1$ be a Riemann surface obtained from an equilateral structure
and let $T_1$ be a triangle which  positively oriented
with the decoration $(\circ,\bullet,\ast)$. Then, combinatorially, the
Riemann surface $X_1\,{}_{_{T_1}}\#_{_{T_2}}\, X_2$, 
is obtained from $X_1$ by starring $T_2$ by an interior point (see Figure (\ref{starring-interior})).  

\begin{figure}
	\begin{center} 
	\includegraphics[width=10cm]{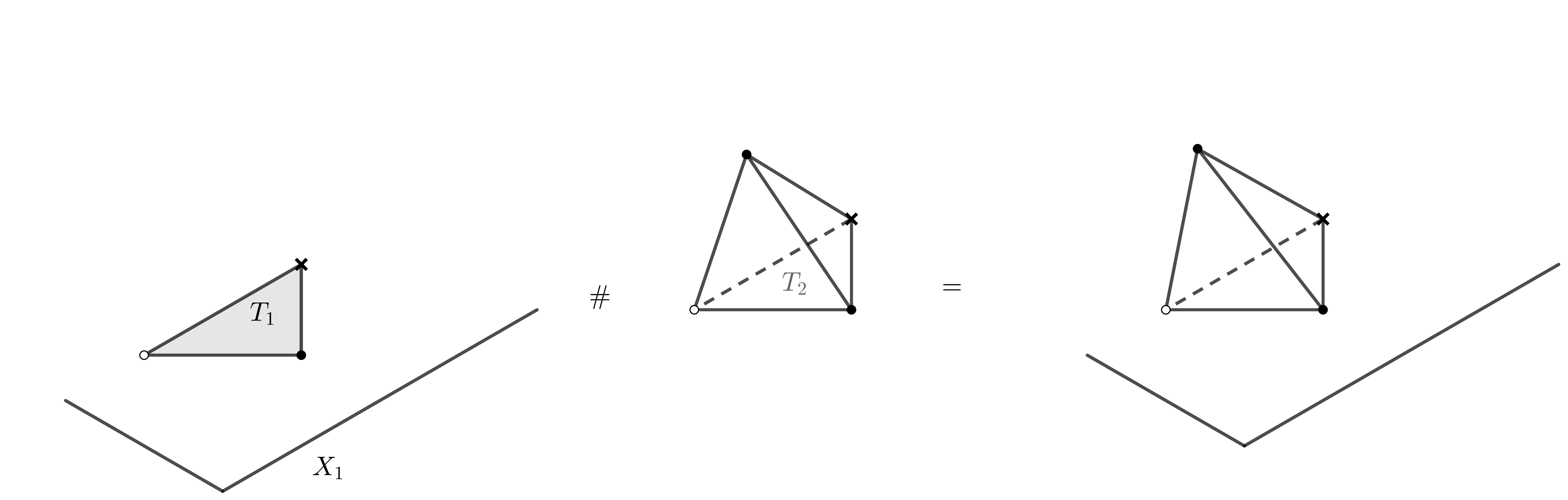} 
	\caption{Connected sum of a triangulated surface with a tetrahedron. 
	Combinatorially this corresponds to starring with respect to a point in the interior of a triangle.} 
	\label{starring-interior} 
	\end{center}
\end{figure}

\par 

Analogously, if we make the connected sum of two tetrahedra
we obtain a doble pyramid with six triangles $2P$ (see Figure (\ref{suma-dos-tetraedros})) and consider a rhombus $R_2$ 
obtained from two triangles with a triangle in each tetrahedron in this connected sum.
If $R_1$ is a rhombus in $X_1$ obtained from two adjacent triangles
we can do the connected sum $X_1\,{}_{_{R_1}}\#_{_{R_2}}\, 2P$
to obtain a triangulated Riemann surface which is combinatorially the starring of
 the triangulation of $X_1$ with respect to the midpoint of an edge of $T_1$ (see Figure (\ref{starring-arista})). 

\begin{figure}
	\begin{center} 
	\includegraphics[width=10cm]{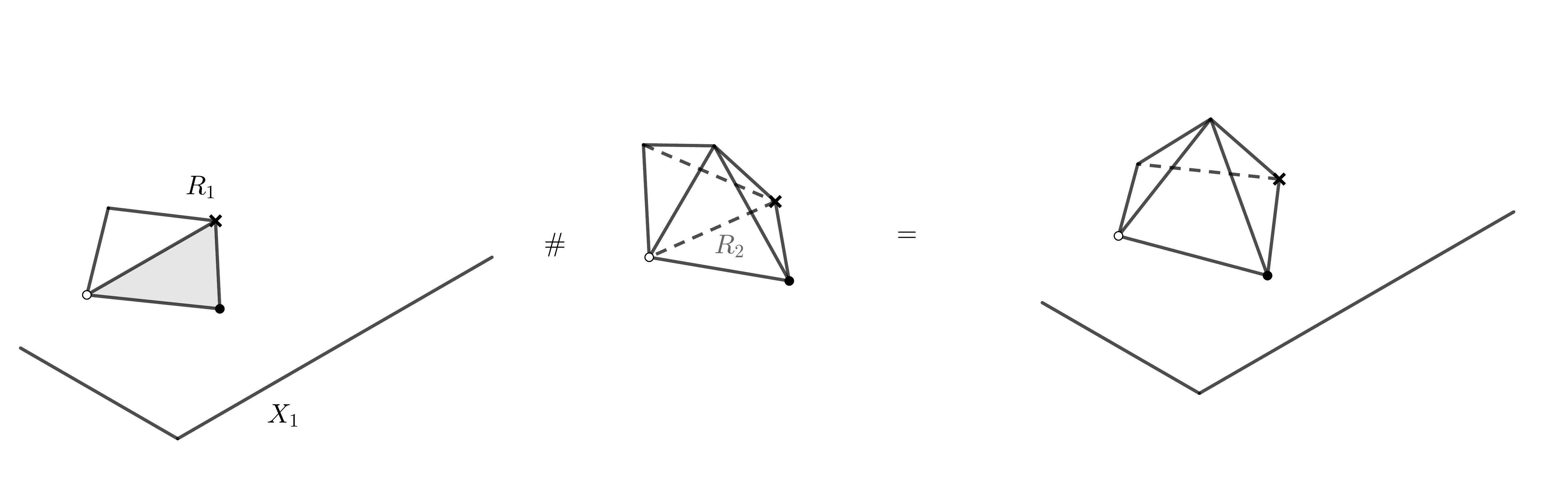} 
	\caption{Connected sum of a triangulated surface with a tetrahedron using rhombi. 
	Combinatorially,  this corresponds to starring with respect to a point in the interior of an edge.} 
	\label{starring-arista} 
	\end{center}
\end{figure}

\section{Elliptic curves with a hexagonal decomposition}

\subsection{Coverings of elliptic curves and orbits of 
$\operatorname{PSL}(2,\mathbb{Q})$ in the upper half-plane.}

\begin{proposition}\label{cubrientes-subreticulas}
	Let $\Lambda$ be a lattice in $\mathbb C$. Then a Riemann surface 
	 $X$, is a non-ramified holomorphic covering of
	 $\mathbb{C}/\Lambda$ if and only if $X$ is conformally 
	equivalent to $\mathbb{C}/\Gamma$ for some sub-lattice $\Gamma\leq \Lambda$.
\end{proposition}

\begin{proof}
	If $\Gamma\leq\Lambda$ is a sub-lattice one has the following commutative diagram 
		
	\begin{equation}\label{diagrama-cambio-clase}
		\xymatrix{& \mathbb{C}\ar^{\pi_\Lambda}[rd]\ar_{\pi_\Gamma}[ld] 
		& \\ 		 
		\mathbb{C}/\Gamma\ar^{\Phi}[rr] &  &\mathbb{C}/\Lambda,}
	\end{equation}

\noindent where $\pi_\Gamma$ and $\pi_\Lambda$
are the natural projection to the quotients and
 $\Phi$ is the function that changes the equivalence relation. 
 It follows from (\ref{diagrama-cambio-clase}) that $\Phi$ is 
a holomorphic map and, by Riemann-Hurwitz formula is not ramified.

\par Reciprocally, if there exists $f\colon X\to \mathbb{C}/\Lambda$ 
a non-ramified holomorphic covering the group of deck transformations
of the covering $\operatorname{Deck}(f)$ is a sub-lattice of $\Lambda$ and
 $X$ is conformally equivalent to
$X/\operatorname{Deck}(f)$. Then we can take 
$\Gamma=\operatorname{Deck}(f)$. 

\end{proof}


\begin{proposition}\label{PSL(2,Q)-cubrientes}
	If $\tau$ and $\tau'\in \mathbb{H}$, then $X_{\tau'}$ is a holomorphic covering of
	 $X_\tau$ if and only if there exists $\gamma\in 
	\operatorname{PSL}(2,\mathbb{Q})$such that $\tau'=\gamma(\tau)$.
\end{proposition}

\begin{proof}
	If $X_{\tau'}$ is a holomorphic covering of $X_\tau$, there exists a
	sub-lattice $\Gamma=\langle \omega_1,\omega_2\rangle$ of $\langle 
	1,\tau \rangle$,  such that $X_{\tau'}$ is conformally equivalent
	to $\mathbb{C}/\Gamma$. We can assume that
	$\omega_1/\omega_2\in 
	\mathbb{H}$. Then there exist $a,b,c,d\in \mathbb{Z}$ such that 
	
	\begin{equation}\label{matriz-SL2Z}
	\left(\begin{matrix}\omega_1\\ 
	\omega_2\end{matrix}\right)=\left(\begin{matrix} a & b \\ c & d 
	\end{matrix}\right) \left(\begin{matrix} 1 \\ \tau 
	\end{matrix}\right),\quad ad-bc\neq 0,
	\end{equation}
	
	\noindent therefore $\gamma(z)=(az+b)/(cz+d)$ is in
	$\operatorname{PSL}(2,\mathbb{Q})$ and 
	$\gamma(\tau)=\omega_1/\omega_2$. 
	On the other hand,
	$\omega_1/\omega_2$ is in the orbit of $\tau'$ under
	$\operatorname{PSL}(2,\mathbb{Z})$, because the associated elliptic curves
        are conformally equivalent. Hence $\tau'$ 
	is in the orbit of $\tau$ under
	$\operatorname{PSL}(2,\mathbb{Q})$. \\
	
	\par Reciprocally, suppose there exists $\gamma\in 
	\operatorname{PSL(2,\mathbb{Q})}$ such that $\gamma(\tau)=\tau'$. 
	Clearing denominators we can assume that
		
	\begin{equation}
		\gamma(z)=\frac{az+b}{cz+d}, \quad \text{con}\ a,b,c,d\in 
		\mathbb{Z}\ \text{y}\ ad-bc\neq 0.
	\end{equation}
	
	Therefore if we define $\omega_1=a\tau+b$ and $\omega_2=c\tau+d$ we have 
	$\Gamma=\langle \omega_1,\omega_2\rangle \leq \langle 
	1,\tau\rangle$. Hence $\mathbb{C}/\Gamma$ is a holomorphic covering of
	$X_\tau$. Since $\tau'=\omega_1/\omega_2$, 
	$\mathbb{C}/\Gamma$  is isomorphic to $X_{\tau'}$
	and the result follows.	
	\end{proof}
By the previous results and Belyi's Theorem we obtain:

\begin{corollary}
	If $\tau\in \mathbb{H}$ is such that $X_\tau$ is an arithmetic elliptic curve
	(\ie defined over $\bar{\mathbb Q})$ then
	for any  $\gamma\in 
	\operatorname{PSL}(2,\mathbb{Q})$, $X_{\gamma(\tau)}$ is 
	also arithmetic. 	
 \end{corollary}
 
 \begin{corollary}\label{cubriente-tau}
	If $\tau$ and $\tau'\in \mathbb{H}$, then $X_{\tau'}$ is 
	a holomorphic covering of $X_\tau$ if and only if $X_{\tau}$ is 
	a holomorphic covering of $X_{\tau'}$.
 \end{corollary}

 Corollary \ref{cubriente-tau} can also be deduced as follows:
 suppose that $X_{\tau'}$ is a holomorphic covering of 
 $X_\tau$, then $X_{\tau'}\cong \mathbb{C}/\Gamma$ for some 
 sub-lattice $\Gamma=\langle \omega_1,\omega_2\rangle\leq \langle 
 1,\tau\rangle$. By Remark \ref{k-subreticula} there exists $k\in 
 \mathbb{Z}$ such that $k\langle 1,\tau \rangle \leq \langle 
 \omega_1,\omega_2\rangle$. Since the elliptic curve
 $X_{\tau}$ is 
 isomorphic to $\mathbb{C}/\langle k,k\tau\rangle$, it follows that $X_{\tau}$ is 
a holomorphic covering of $X_{\tau'}$.

\begin{remark}\label{k-subreticula}
	If $\langle \omega_1,\omega_2\rangle\leq  \langle 1,\tau\rangle$ is 
	una sub-lattice, then there exists  $k\in \mathbb{Z}$ such that
	$k\langle 1,\tau \rangle \leq \langle \omega_1,\omega_2\rangle$. 
	This is because there is a matrix which satisfies  (\ref{matriz-SL2Z}), 
	then if we take $k=ad-bc$ we obtain:	
	\begin{equation} 
		k\left(\begin{matrix} 1 \\ \tau 
		\end{matrix}\right)=\left(\begin{matrix} d & -b \\ -c & a 
		\end{matrix}\right) \left(\begin{matrix}\omega_1\\ 
		\omega_2\end{matrix}\right). 
	\end{equation}

 \end{remark}

\subsection{Surfaces which admit a regular hexagonal decomposition}

\begin{proposition}\label{libre-torsion-236}
	If $\Gamma$ is a torsion-free subgroup of the triangular group  
	$\Gamma_{2,3,6}$, then $\Gamma$ is a sub-lattice of
	$\Lambda_{1,\omega}$, where $\omega$ is equal to $\exp(2 \pi i /6)$.
\end{proposition}

\begin{proof}
	Consider the action of $\Gamma_{2,3,6}$ on the complex plane
         $\mathbb{C}$ and its associated tessellation (see Figure 
	(\ref{236-limpio4-4})). We note that the only points in 
	$\mathbb{C}$ with nontrivial  stabilizer are the vertices of the tessellation. 
	Therefore, if $T\in \Gamma_{2,3,6}$ has a fixed point, 
	$T$ must belong to the stabilizer of some vertex, however these stabilizer are cyclic
	hence  $T$ is a torsion element.

	\begin{figure}
	\begin{center} 
	\includegraphics[width=14cm]{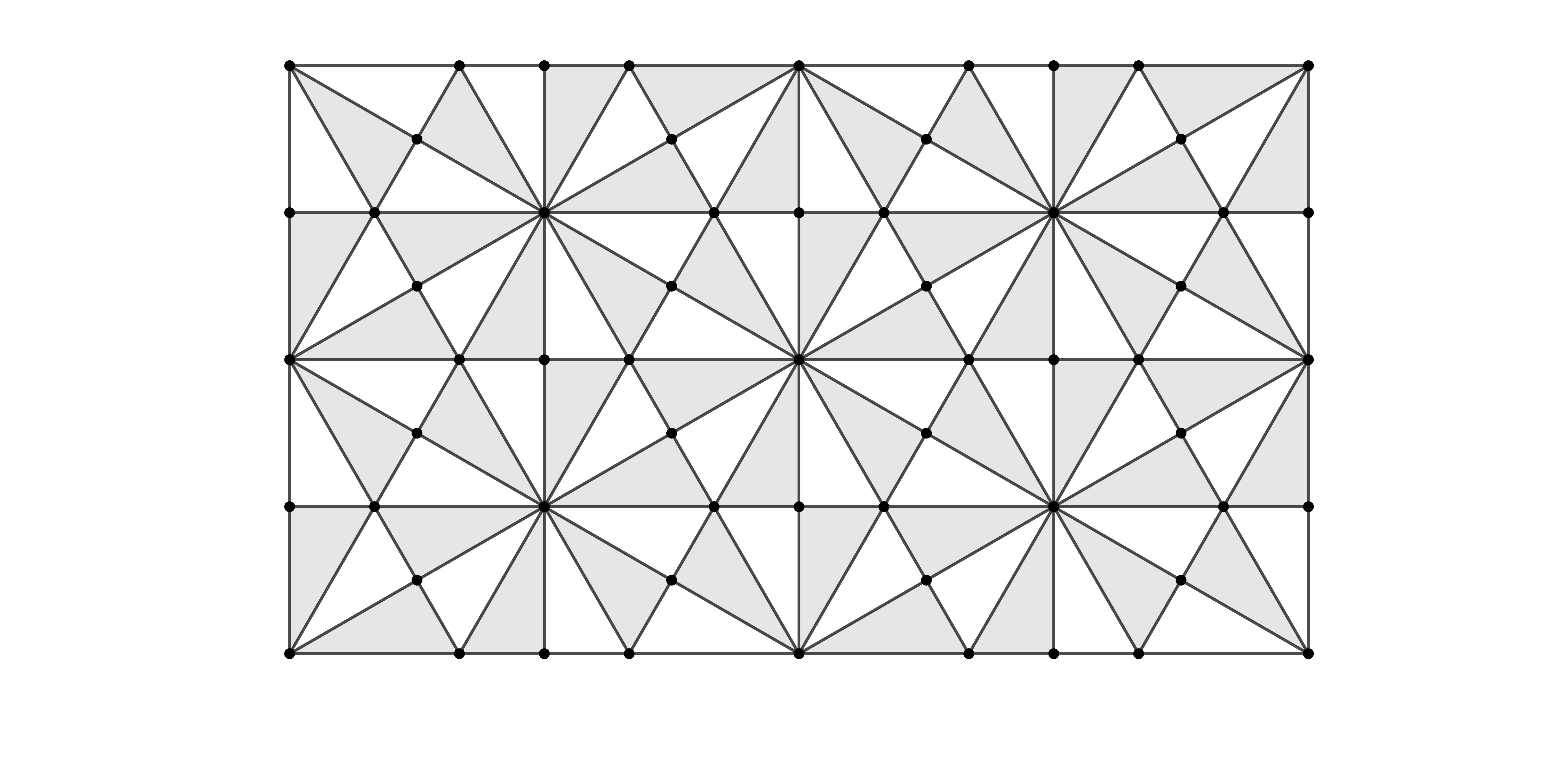} 
	\caption{Tessellation of $\mathbb C$ by the triangular group
	$\Gamma_{2,3,6}$.} 
	\label{236-limpio4-4} 
	\end{center}
	\end{figure}
	
	\par On the other hand, if $T\in \Gamma_{2,3,6}$ is a translation
	$T$ then it sends vertices of the tessellation into vertices of the tessellation 
	and it preserves its valencies, hence it preserves the triangular lattice
	$\Lambda_{1,\omega}$ (the points of these triangular lattice are the vertices with valence 12).
	
	\par It follows that if $\gamma$ is in the torsion-free subgroup $\Gamma$, 
	then $\gamma$ must be a translation that 
	preserves the triangular lattice.  Hence $\gamma\in \Lambda_{1,\omega}$.
\end{proof}

\begin{proposition}\label{hexagonales-reticulas}
	Suppose that $X$ is a compact, connected Riemann surface. Then $X$ 
	admits a flat structure given by a finite union of regular hexagons
        (\ie the valence of each vertex is 3) 
	if and only if $X$ is conformally equivalent to
	$\mathbb{C}/\Gamma$, where $\Gamma$ is a sub-lattice of 
	$\Lambda_{1,\omega}$.
\end{proposition}

\begin{proof}
	Let us suppose that  $X$ is given by a cellular decomposition by
	regular hexagons, with vertices of valence 3 so that the surface doesn't have singularities.
	Let us decorate with the symbol $\bullet$ the vertices (the 0-cells),
	with the symbol $\circ$ the midpoints of the edges. Let us define new edges
	dividing in half each edge of the hexagons. This cartographic map is a 
	uniform dessin d'enfant 
	with valence $(2,3,6)$. Hence $X$ is conformally
	equivalent to $\mathbb{C}/\Gamma$ for some torsion-free subgroup 
	(see \cite{Girondo-Gonzalez} section 4.4.3)
	$\Gamma\leq \Gamma_{2,3,6}$. By Proposition 
	\ref{libre-torsion-236} this subgroup must be a sub-lattice of
	$\Lambda_{1,\omega}$.  
	
	
	\par Reciprocally, let us suppose that $\Gamma$ is a sub-lattice
	of $\Lambda_{1,\omega}$ with $\omega = \exp(2\pi{i}/6)$. By 
	Proposition \ref{cubrientes-subreticulas} the map
	$\Phi\colon \mathbb{C}/\Gamma\to \mathbb{C}/\Lambda$ is a holomorphic covering map.
	 Therefore any regular hexagonal
	 decomposition of $\mathbb{C}/\Lambda$ lifts, by means of $\Phi$,
	 to endow $\mathbb{C}/\Gamma$ with a regular hexagonal
	 decomposition.The complex structure induced by the hexagonal structure coincides 
	with original complex structure since $\Phi$ is holomorphic.
 \end{proof}
 
 \begin{remark}
	(i) The elliptic curve $\mathbb{C}/\Lambda_{1,\omega}$, with
	$\omega = \exp(2\pi i/6)$, admits hexagonal decompositions, 
	it even admits decompositions such that two hexagons intersect in at most one edge, see Figure 
	(\ref{tiling-hexagonal}).
	
	\begin{figure}
	\begin{center} 
		\includegraphics[width=10cm]{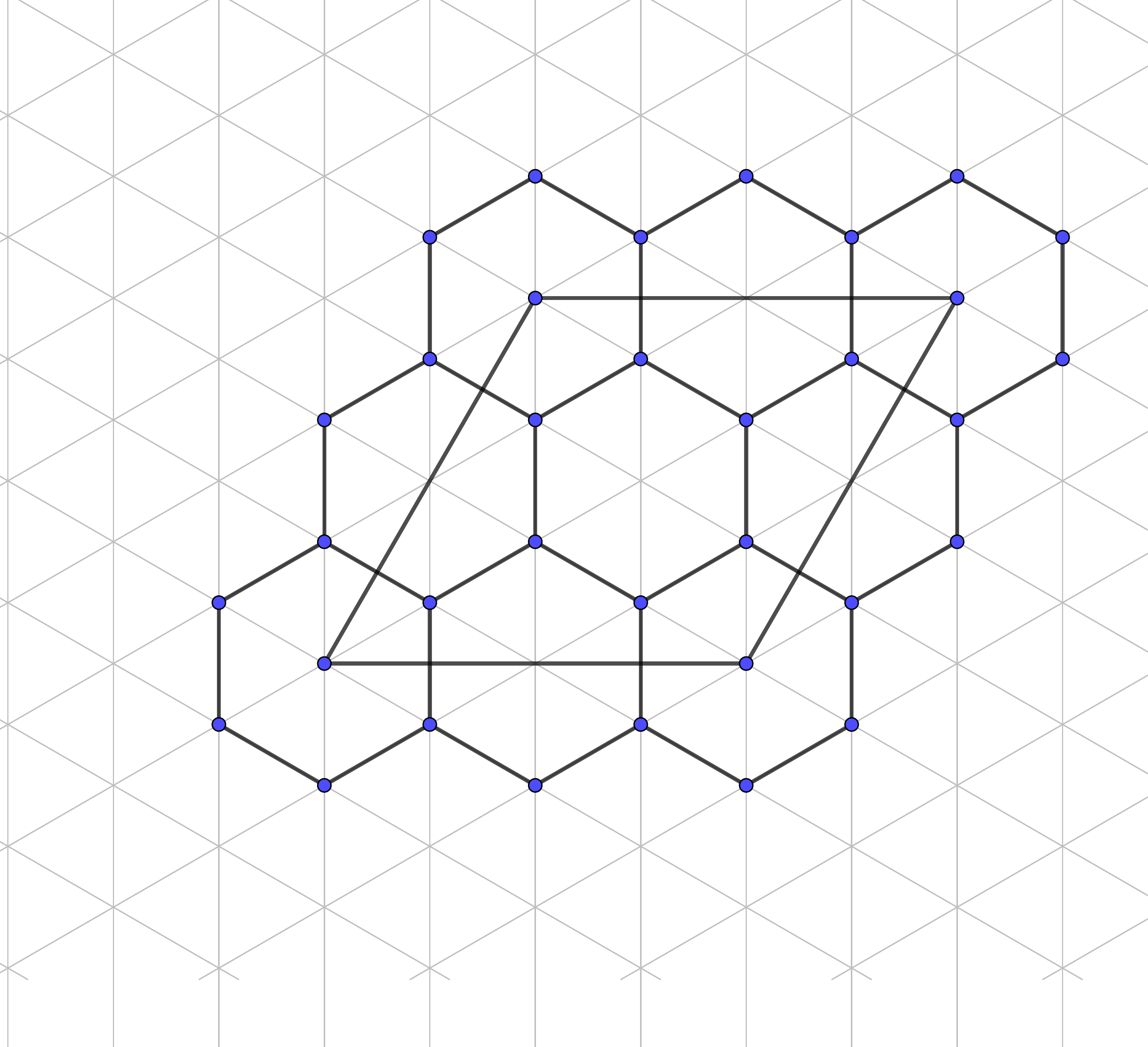} 
		\caption{Hexagonal decomposition of 
		$\mathbb{C}/\Lambda_{1,\omega}$.} 
		\label{tiling-hexagonal} 
	\end{center}
	\end{figure}
	
	\par (ii) In the proof of proposition \ref{hexagonales-reticulas}
	we didn't need to suppose that the hexagons meet at most at one edge.
	For instance the decomposition could consist of only one hexagon.
	
	 \end{remark}
 
  Propositions  \ref{cubrientes-subreticulas} 
 \ref{PSL(2,Q)-cubrientes}, \ref{hexagonales-reticulas} imply the following corollary. 
 
 \begin{corollary}
	Let $\tau\in \mathbb{H}$, then $X_{1,\tau}$ admits a	flat structure
	without conical points given by a finite union of regular hexagons
	with each vertex of valence 3 if and only if there exists
	 $\gamma \in \operatorname{PSL}(2,\mathbb{Q})$   
	such that $\tau=\gamma(\omega)$.
 \end{corollary}
 
  \section{Geodesics in surfaces with a decorated equilateral triangulation}

\par As mentioned before we will denote by $\Delta$ the euclidean 
equilateral triangle with vertices at $0,1,\omega$, with $\omega=\exp(2 
\pi i/6)$. We decorate these vertices with the symbols $\circ, \bullet$ 
and $\ast$, respectively. All through this paper we will assume that 
the vertices of the tessellation of the triangular group 
$\Gamma_{3,3,3}$, generated by $\Delta$ are decorated in a compatible 
way with respect to $\Delta$ (see Figure (\ref{333-deco})).

\begin{figure} 
	\begin{center} 
	\includegraphics[width=10cm]{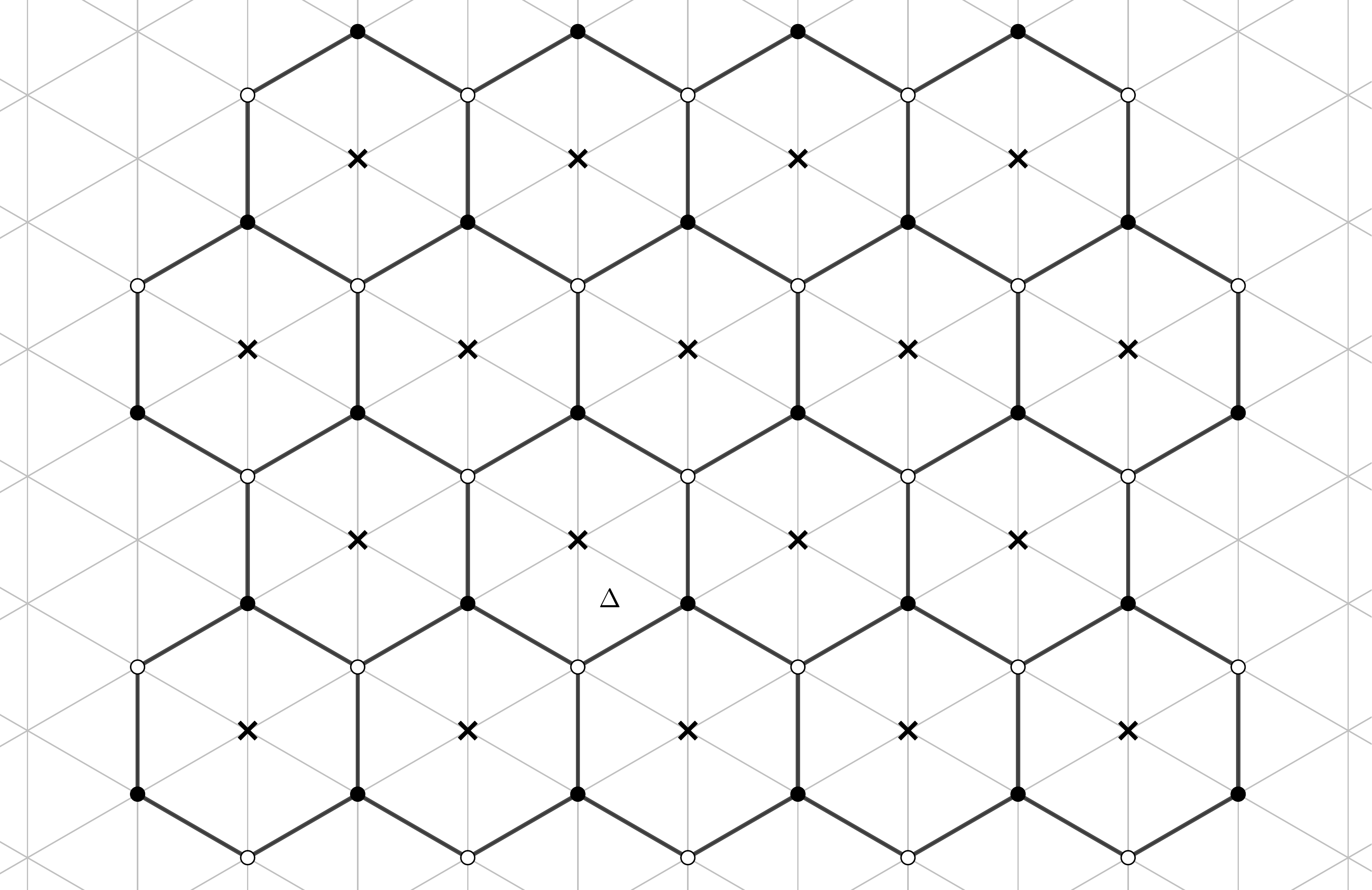} 
	\caption{Decorated tessellation of $\Gamma_{3,3,3}$, generated by $\Delta$.} 
	\label{333-deco} 
	\end{center} 
\end{figure}

\par All through this section we will suppose that $X$ is a Riemann 
surface obtained from a decorated euclidean equilateral triangulation, 
(in other words $X$ is determined from a dessin d'enfant). Let  
$f\colon X\to 2\Delta$ be its associated Belyi function. Let us recall 
that, by definition, this function is the one that sends each triangle 
of $X$ to one of the two triangles of $2\Delta$, preserving the 
decoration. 

\par This section is devoted to the study of geodesics 
with the respect to the conic flat metric
of the triangulated surface $X$ and which do not pass through the  of the vertices of the triangulation.   
 
\subsection{Geodesics defined by straight lines}\label{geodesica-marcada}
Let $\ell$ be a straight line in $\mathbb C$ which  passes through $1/2$ 
and does not contain any point of the lattice
$\Lambda_{1,\omega}$.
Let $P$ be the union of all the triangles which intersect $\ell$. Then since
$\ell$ does not contain any point of the lattice 
 $\ell$ is contained in the interior of $P$. In all that follows we 
 will assume that $P$ has its vertices decorated 
according to the decoration given to the tessellation associated to 
$\Gamma_{3,3,3}$ (see Figure (\ref{triangulo-recta})).
 
\begin{figure} 
	\begin{center} 
	\includegraphics[width=8cm]{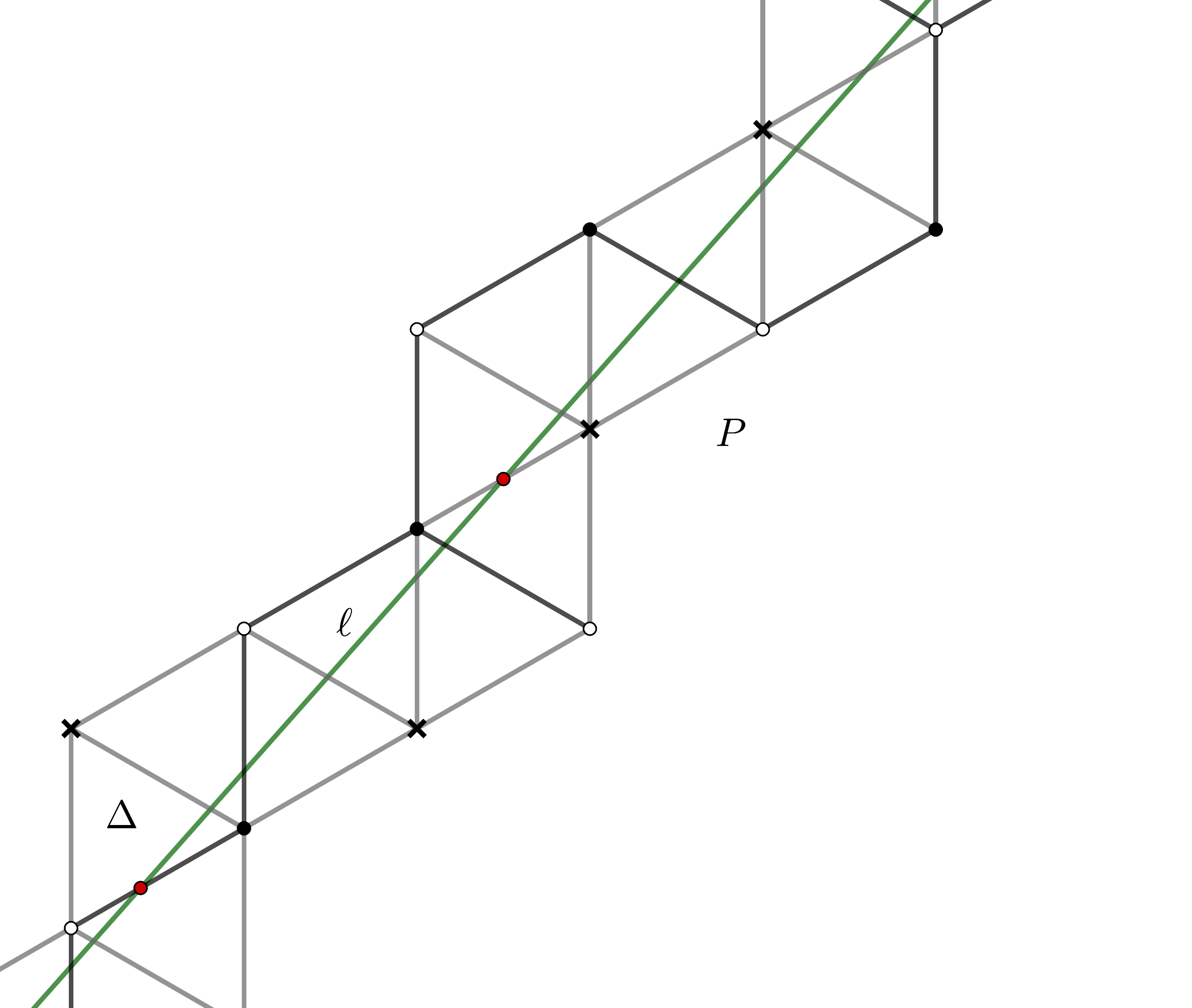} 
	\caption{Decorated triangle and straight line $\gamma$ passing through $1/2$.} 
	\label{triangulo-recta} 
	\end{center} 
\end{figure}

\par Let $p\in{X}$ be the midpoint of an edge of the dessin of
 $X$ determined by $f$, \ie  $p\in f^{-1}(1/2)$. Then there exist a unique 
 continuous map $\Psi\colon P\to X$ satisfying the following conditions:

\begin{itemize}
\item[(i)] $\Psi(1/2)=p$. 

\item[(ii)] $\Psi$ sends isometrically triangles onto triangles
through an isometry that preserves the decoration. 
\end{itemize}
 
\par Then the images under $\Psi$ of two adjacent triangles of $P$, 
are two adjacent triangles in $X$. Therefore  
$\Psi$ is a local isometry from the interior of $P$ to $X$. We will call the map $\Psi$ 
the \emph{development} of $\ell$ in $X$ with respect to $p$. 

\par By the previous observations  $\Psi(\ell)$ is a geometric geodesic in 
$X$ (\ie a non parametrized geodesic). We we call such a curve $\Psi(\ell)$ 
\emph{the geometric geodesic induced} by $\ell$ in $X$ .

\begin{remark}\label{proyeccion2Delta}
	We observe that $2\Delta$ is the quotient $\mathbb{C}/\Gamma_{3,3,3}$. 
	So that the complex plane $\mathbb{C}$ is a branched covering of $2\Delta$ 
	with branching points the vertices of the triangles determined by $\Gamma_{3,3,3}$.
	The natural projection  $\pi\colon \mathbb{C}\to \mathbb{C}/\Gamma_{3,3,3}$ is an  
	isometry when restricted to each triangle and it preserves the decorations.  This implies that
	$\pi(\ell)$ is a geometric geodesic in $2\Delta$. In fact if
	 $X=2\Delta$ and $p=1/2\in 2\Delta$ then $\pi_{|_{P}}=\Psi$. 
\end{remark}

\begin{remark} 
Let us denote  $\Psi(\ell)$ by $\tilde{\gamma}$ and $\pi(\ell)$ by $\gamma$. 
By Remark  \ref{proyeccion2Delta} it follows that locally 
 $f\colon X\to 2\Delta$ can be written as
$f=\pi\circ\Psi^{-1}$. Hence, $\tilde{\gamma}$ is the lifting, based in $p$, of
 $\gamma$, with respect to $f$.
\end{remark}

\subsection{Closed and primitive geodesics in $2\Delta$}

Let us consider the parametrized straight line $\ell(t)=1/2+t\omega_0$, with $\omega_0\in 
\Lambda_{1,\omega}$ (the parametrization is by a multiple of arc-length). We study in this subsection the condition we must impose to 
$\omega_0$, which guarantee that $\ell(t)$ induces a smooth closed geodesic  in 
$2\Delta$ which does not contain any vertex of $2\Delta$.

The following simple proposition, whose proof we leave to the reader, determines the condition that garantes that
$\ell$ does not contain any point of the lattice $\Lambda_{1,\omega}$.

\begin{proposition} 
	If $\omega_0\in \Lambda_{1,\omega}$ is expressed as
	$\omega_0=m_0+n_0\omega$, with $n_0\neq 0$, then the parametrized line $\ell(t)=1/2+t\omega_0$ 
	does not contain any point of the lattice $\Lambda_{1,\omega}$ if and only if 
	the diophantine equation $n_0=(2n_0)x+(-2m_0)y$ has a solution with $x,y\in{\mathbb Z}$.
\end{proposition}

\begin{corollary}
	If $\omega_0=m+n\omega\in \Lambda_{1,\omega}$ with $\text{gcd}(m,n)=1$, then
	$\ell(t)=1/2+t\omega_0$ does not contain a point in the lattice
	$\Lambda_{1,\omega}$ if and only if $n$ is odd.
\end{corollary}

\begin{remark}
 Let $Hex=\Gamma_{3,3,3}\cap \Lambda_{1,\omega}$. 
 We observe that $Hex$ is the sub-lattice of $\Lambda_{1,\omega}$
 consisting of the vertices with symbol $\circ$  
 of the tessellation obtained by $\Gamma_{3,3,3}$. We see that
 $\omega_1=2-\omega$ and $\omega_2=1+\omega$ are $\mathbb R$-linearly independent vectors
 wich belong to $Hex$ (see Figure 
 (\ref{centros-hexagonales}). Then
	
	\begin{figure}
		\begin{center} 
		\includegraphics[width=10cm]{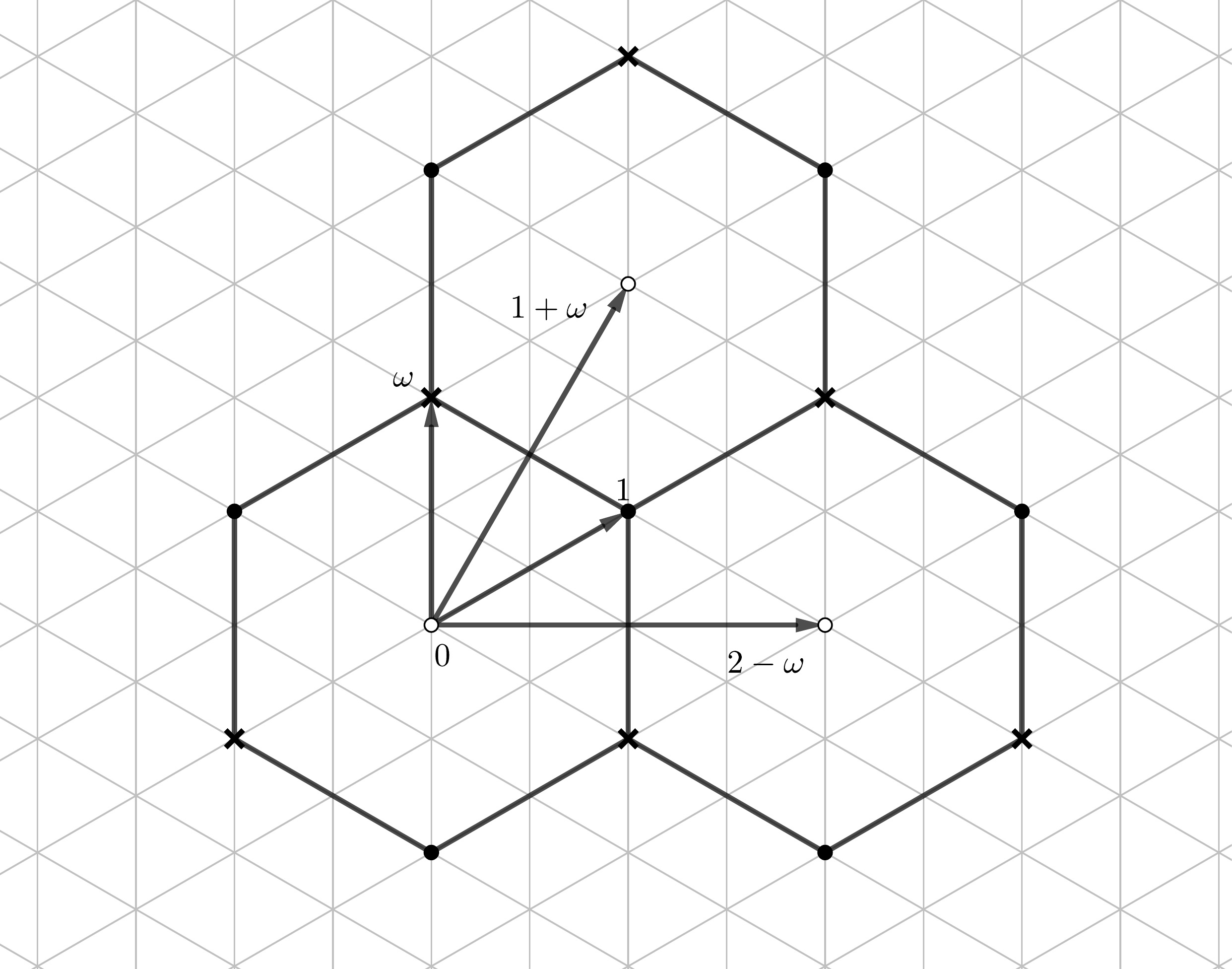}
		\caption{Generators of the sub-lattice formed by the vertices with symbol $\circ$.} 
		\label{centros-hexagonales} 
		\end{center}
	\end{figure}
	
	 \begin{equation}
		Hex=\langle 2-\omega,\ 1+\omega\rangle,
	\end{equation}
	
         Therefore any element $\alpha$ of $Hex$ can be written in terms of the base
        $1,\ \omega$ as follows
	
	\begin{equation}
		\alpha=m\omega_1+n\omega_2=2m+n+(n-m)\omega.
	\end{equation}
\end{remark}

\begin{remark}
	If $\omega_0=m+n\omega\in \Lambda_{1,\omega}$, then 
	$3\omega_0\in Hex$, because
	\begin{equation}
		3(m+n\omega)=(m-n)(2-\omega)+(m+n)(1+\omega).
	\end{equation}
	In addition, if $2\omega_0=2(m+n\omega)\in Hex$, then $\omega_0\in 
	Hex$.
\end{remark}

\par As a consequence of what has been proved in this subsection
we have the following theorem that gives conditions in order that
the parametrized (by a multiple constant of arc-length) 
$\ell(t)=1/2+t\omega_0$ parametrizes a primitive closed geodesic
in $2\Delta$.

\begin{theorem}\label{geodesica-cerrada-primitiva}
	Let  $\omega_0=m+n\omega$, with $m$ and $n$ relatively prime and $n$ odd. 
	Let $\ell(t)=1/2+t\omega_0$. If $\omega_0\in Hex$, then
	the curve $\gamma(t)=\pi(\ell(t))$ with $t\in [0,1]$ is a primitive closed geodesic in $2\Delta$. 
	If $\omega_0\notin Hex$, 
	$\gamma(t)=\pi(\ell(t))$ with $t\in [0,3]$, is a primitive closed geodesic in $2\Delta$.
\end{theorem}

\begin{corollary}
	The length of the geodesic 
	 $\gamma$ parametrized by
	$\pi(\ell(t))=\pi(1/2+t\omega_0)$, $\omega_0=m+n\omega$ with $\text{gcd}(m,n)=1$ and $n$ 
	odd, is equal to
	\begin{equation}
		\operatorname{length}\gamma=
		\begin{cases}
			|\omega_0|,\ \text{if}\ \omega_0\in Hex,\\
			3|\omega_0|,\ \text{if}\ \omega_0\notin Hex.
		\end{cases}
	\end{equation}
\end{corollary}
 
\par The following is the reciprocal of Theorem 
\ref{geodesica-cerrada-primitiva}: Let $\gamma$ be a closed smooth 
geodesic in $2\Delta$ based on $1/2$ and not passing through any 
vertex); let $\ell$ be its lifting with respect to the canonical 
projection $\pi\colon \mathbb{C}\to\mathbb{C}/\Gamma_{3,3,3}$, with 
base point $1/2$ (the lifting exists because $\pi$ is a covering map 
outside of the vertices of the tessellation). Since $\pi$ is a local 
isometry outside of the vertices $\ell$ is also a geodesic (with the 
euclidean metric of $\mathbb{C}$) which passes through $1/2$ and it 
does not contain any vertex of the tessellation given by 
$\Gamma_{3,3,3}$. Hence $\ell$ is a straight line which does not 
contain points in $\Lambda_{1,\omega}$. Since the image of $\ell$ under 
$\pi$ is closed curve it must connect $1/2$ with $1/2+\omega_0'$, for 
some $\omega_0'\in \Gamma_{3,3,3}$. Since $\pi$ is a local Riemannian 
isometry outside of the vertices the angles between the straight line 
and the edges which contain $1/2$ and $1/2+\omega_0'$ are the same and, 
therefore the edges are parallel and $\omega_0'\in \Lambda_{1,\omega}$. 
Hence $\omega_0'$ is a vertex with symbol $\circ$ of the decorated 
tessellation determined by $\Gamma_{3,3,3}$. 

\par Proposition \ref{interseccion-recta-reticula} implies that we can 
choose $\omega_0=m_0+n_0\omega$ with minimal modulus which generates 
the same straight line as $\omega_0'$. Then $\text{gcd}(m_0,n_0)=1$ and 
since $\ell$ does not contain points in the lattice it follows that 
$n_0$ is odd. We summarize all of the above paragraph with the 
following:

\begin{theorem}\label{teorema-2Delta}
	A smooth closed curve $\gamma(t)$ en $2\Delta$, which passes 
	through $1/2$, is a parametrized (by a constant multiple of 
	arc-length) primitive closed geodesic in $2\Delta$ if and only if 
	$\gamma(t)$ is given by $\pi(\ell(t))=\pi(1/2+t\omega_0)$ as in 
	Theorem \ref{geodesica-cerrada-primitiva}.
\end{theorem}

\begin{proposition}\label{interseccion-recta-reticula}
	Let $\omega_0=m_0+n_0\omega$ with $\text{gcd}(m_0,n_0)=1$. Then the 
	intersection of the line $\ell(t)=t \omega_0$ with the lattice 
	$\Lambda_{1,\omega}$ is equal to the set of $k\omega_0$, with $k\in 
	\mathbb{Z}$. Hence the line $t\mapsto1/2+t\omega_0$ intersects the 
	translate of the lattice  $\Lambda_{1,\omega}$, 
	$1/2+\Lambda_{1,\omega}$, in the set of points of the form 
	$1/2+k\omega_0$, with $k\in \mathbb{Z}$.
\end{proposition}
 
\begin{proof}
	Clearly $k\omega_0\in \Lambda_{1,\omega}$ for each $k\in 
	\mathbb{Z}$. Reciprocally, if $t\omega_0=m_1+n_1\omega$ for some 
	integers $m_1$, $n_1$, one has $tm_0=m_1$ and $t n_0=n_1$. Since 
	$\text{gcd}(m_0,n_0)=1$ there exist $a,b\in\mathbb{Z}$ such that 
	$am_0+bn_0=1$. Hence $t=am_1+bn_1\in \mathbb{Z}$.
\end{proof}

\subsection{Primitive closed geodesics in surfaces with an euclidean 
triangulated structure}\label{seccion-geodesica-cerrada-primitiva} 

\par Let $X$ be a Riemann surface obtained from a decorated euclidean 
triangulated structure. Let $f\colon X\to 2\Delta$ be its associated 
Belyi function.

\par Let $p$ be a point in $f^{-1}(1/2)\subset{X}$, \ie it is a middle 
point of an edge of the dessin. If $\alpha(t)$ is a parametrized closed 
smooth geodesic in $X$ based in $p$, and which does not pass by a 
vertex of the triangulation then $\tilde\alpha(t)=f(\alpha(t))$ is a 
parametrized smooth closed geodesic in $2\Delta$ based in $1/2$; then 
there exists $\omega_0$ satisfying the conditions of Teorema 
\ref{geodesica-cerrada-primitiva} and such that $f(\alpha)=\pi(\ell)$ 
with $\ell(t)=1/2+t\omega_0$; on the other hand if $\Psi$ is the 
development of $\ell$ in $X$ respect to $p$, one has that 
$\tilde{\gamma}=\Psi\circ\ell$ is the lifting of $f\circ\alpha$ based 
in $1/2$, then by unicity of based liftings it follows that 
$\tilde{\gamma}=\alpha$. We will see below, in Corollary 
\ref{corolario-length} that the corresponding geodesic 
$\tilde{\gamma}=\Psi\circ\ell$ is also a smooth closed geodesic. \par 
Let us suppose that we have fixed a labeling with symbols $a_i$ 
($i=1,\cdots,\text{deg}(f)$) of the edges of the graph corresponding to 
the dessin d'enfant of $X$ associated to the Belyi function 
$f:X\to2\Delta$. 
 
We denote by $\sigma_0$ and $\sigma_1$ the permutations associated to 
the dessin (permutations around the vertices with symbols $\circ$ and 
$\bullet$, respectively). Given $\gamma\in 
\pi_1(2\Delta-\{0,1,\omega\},1/2)$ one has the associated permutation 
$\sigma_{\gamma}$ of the fibre $f^{-1}(1/2)$ given by 
$\sigma_\gamma(\tilde{x})=\tilde{\gamma}(1)$, where $\tilde{\gamma}$ is 
a lifting of $\gamma$ based on $\tilde{x}$, for any  
$\tilde{x}\in{f^{-1}(1/2)}$. If we fix a labeling we can think of 
$\sigma_{\gamma}$ as an element of the symmetric group $S_d$, with 
$d=\deg f$. Of course one may also consider that $\sigma_{\gamma}$ is a 
permutation of the edges of the dessin of $X$, we will use this 
identification implicitly in the following discussion.

Let us recall that the homomorphism $\mu\colon 
\pi_1(2\Delta-\{0,1,\omega\},1/2)\to S_d$ defined by 
$\mu(\gamma)=\sigma_{\gamma^{-1}}$ is a group homomorphism called the 
\emph{monodromy homomorphism} of the dessin determined by the Belyi 
function$f$ and its image is called the \emph{monodromy group} of $f$. 

With the labeling of the edges fixed as before we denote by $p_i$ the 
midpoint of the edge $a_i$. For a curve $\gamma$ based in $1/2$ we 
denote by $\tilde{\gamma}_{p_i}$ its lifting based in $p_i$. 

\begin{theorem}
Let $X$ be a Riemann surface obtained from a decorated  equilateral triangulation.
 Let $f\colon X\to 2\Delta$ be the associated Belyi function.
Let us fix a labeling of the finite set of points $f^{-1}(1/2)$, with the symbols
$a_{1},\ldots,a_{d}$ where $d=\deg\ f$.

\par  Let $\gamma$ be a closed geodesic in $2\Delta$ which is based in
$1/2$ and defined by the line $\ell(t)=1/2+t\omega_{0}$ as inTheorem \ref{geodesica-cerrada-primitiva}. Let $a_i$ 
be an edge of the dessin of $f$ with midpoint $p_{i}\in f^{-1}(1/2)$. If
$\tau_{i}$ is the cycle which contains $i$ in the cyclic decomposition of the permutation
$\sigma_\gamma\in \mathrm{Sym}(f^{-1}(1/2)) \cong S_{d}$
(the isomorphism $\cong$ is deduced from the labeling), then the development
of $\ell(t)$ in $X$ based in $p_{i}$ can be expressed (up to a una reparametrization) as:

\begin{equation*}
\tilde{\gamma}=\tilde{\gamma}_{p_i}\cdot  
\tilde{\gamma}_{\sigma_\gamma(p_i)}\cdots
\tilde{\gamma}_{\sigma^{r-1}_\gamma(p_i)}
\end{equation*}

\noindent where $r=\operatorname{ord} \tau_{i}$ \ie the length of the cycle
$\tau_{i}$.
\end{theorem}

\begin{proof}
	There exists $\omega_0=m+n\omega$, with $\text{gcd}(m,n)=1$ and $n$ odd such that	
	$\gamma(t)=\pi(\ell(t))$, donde $\ell(t)=1/2+t\omega_0$. If we consider the development 
	$\Psi$ of $\ell$ with respect to $p_i\in a_i$, we obtain $\tilde{\gamma}=\Psi(\ell)$. 
	
	\par 
	Suppose that $\omega_0\in Hex$. Then, for each integer $k\geq 0$ one defines the segment 
         $\ell_k:[0,1]\to \C$, $\ell_k(t)=\ell(t+k)$ (i,e, the restriction of $\ell$ to the interval $[k,k+1]$). 
	Let $\tilde{\gamma}_k(t)=\Psi(\ell_k(t))$ and
	$\gamma_k(t)=\pi(\ell_k(t))$, for $t\in[k,k+1]$.
	Since locally $f=\pi\circ \Psi^{-1}$, $\tilde{\gamma}_k(t)$ is a lifting of
	$\gamma_k(t)$, with $t\in [0,1]$.
	
	\par We have 
	\begin{equation}
		\tilde{\gamma}_k(0)=\tilde{\gamma}_{k-1}(1)=
		\sigma_{\gamma}(\tilde{\gamma}_{k-1}(0)),     
	\end{equation} 
         hence 
	\begin{equation}
		\tilde{\gamma}_k(0)=\sigma_\gamma^k(\tilde{\gamma}_0(0))= 
		\sigma_\gamma^k(p_i). 
	\end{equation}
	Therefore,
	\begin{equation}
		\tilde{\gamma}_k(1)=\tilde{\gamma}_{k+1}(0)=
		\sigma_{\gamma}^{k+1}(p_i),     
	\end{equation} 
	
	\begin{equation}
		\tilde{\gamma}_{r-1}(1)=\sigma_{\gamma}^{r}(p_i)=p_i.     
	\end{equation}
	Since $\Psi(\ell)(t)= \tilde{\gamma}_0\cdots
	\tilde{\gamma}_{r-1}(t)$ with $t\in[0,r]$, it follows that 
	\begin{equation}
		\tilde{\gamma}(t)=\tilde{\gamma}_0\cdots\tilde{\gamma}_{r-1}(t),\ 
		t\in [0,r].
	\end{equation}
	Also, since  $\Psi\ell(0)=\Psi\ell(r)$ and the tangent vectors at $0$ and $r$ are the same
	 it follows that $\tilde{\gamma}$ parametrized in the interval $[0,r]$ is a parametrized smooth closed geodesic.
	
	\par We observe that $\tilde{\gamma}_k=\tilde{\gamma}_{\sigma^k(p_i)}$ 
	because
	\begin{equation}
		\tilde{\gamma}_{\sigma^k(p_i)}(0)=\sigma^k_\gamma(p_i)=
		\tilde{\gamma}_k(0).
	\end{equation}
	Therefore
	\begin{equation}
		\tilde{\gamma}(t)=\tilde{\gamma}_{p_i}\cdot  
		\tilde{\gamma}_{\sigma_\gamma(p_i)}\cdots 
		\tilde{\gamma}_{\sigma^{r-1}_\gamma(p_i)}(t),\quad t\in [0,r].
	\end{equation}
	
	\par If $\omega_0\notin Hex$, con $\text{gcd}(m,n)=1$ and $n$ odd
	we can apply an argument analogous to the previous one but with
	 $3\omega_0\in Hex$, and considering now $\ell_k(t)=\ell(t)$ with $t\in 
	[3k,3(k+1)]$ for each integer$k\geq 0$. We obtain: 
	\begin{equation}
		\tilde{\gamma}(t)=\tilde{\gamma}_{p_i}\cdot  
		\tilde{\gamma}_{\sigma_\gamma(p_i)}\cdots 
		\tilde{\gamma}_{\sigma^{r-1}_\gamma(p_i)}(t),\quad t\in [0,3r].
	\end{equation}
\end{proof}

From the proof of the previous theorem it follows the following
corollary:

\begin{corollary}\label{corolario-length}
	Let $\gamma$ be a smooth closed geodesic of $2\Delta$ given by the 
	parametrization $\ell(t)=1/2+t\omega_0$, $\omega_0=m+n\omega$ with 
	$\text{gcd}(m,n)=1$ and $n$ odd. If $a_i$ is an edge of the dessin 
	of $X$ and $\Psi$ is the development of $\ell$ with respect to 
	$p_i$, then the geodesic $\tilde{\gamma}(t)=\Psi(\ell(t))$ is 
	smooth and closed and has length 
	$\operatorname{length}\tilde{\gamma}=\operatorname{ord}\sigma_{a_i} 
	\operatorname{length}\gamma$, where $\sigma_{a_i}$ is the cycle of 
	$\sigma_\gamma$ which contained $a_i$. Hence \begin{equation}
		\operatorname{length} \tilde{\gamma}=
		\begin{cases}
			\operatorname{ord}\sigma_{a_i}|\omega_0|, \quad \omega_0\in 
			Hex\\
			3\operatorname{ord}\sigma_{a_i}|\omega_0|, \quad 
			\omega_0\notin Hex
		\end{cases}
	\end{equation}
	where $\sigma_{a_i}$ is the cycle of $\sigma_\gamma$ which contains $a_i$.
\end{corollary}

\subsection{Cyclic decomposition of $\sigma_{\gamma}$.} In what follows 
we will describe an algorithm useful to calculate the cyclic 
decomposition of $\sigma_\gamma$.

Let us consider $\ell(t)=1/2+t\omega_0$ as before, with 
$\text{gcd}(m,n)=1$ and $n$ odd. If we suppose that $\omega_0\in Hex$, 
then $\gamma(t)=\pi(\ell(t))$ with $t\in[0,1]$ is a closed geodesic in 
$2\Delta$.

\par Suppose that for $t\in [0,1]$, $\ell(t)$, intersects the hexagons 
$H_1,\ldots,H_k$ of the hexagonal tessellation associated to 
$\Gamma_{3,3,3}$ (in that order) and with decorated vertices with the 
symbols $\circ$, $\bullet$. We have that $\ell$ intersects each hexagon 
in 2, 3 or 4 triangles (triangles of the tessellation determined by 
$\Gamma_{3,3,3}$).

\par Suppose that $\ell$ intersects one of $H_j$ in 4 triangles 
$\Delta_1,\Delta_2,\Delta_3,\Delta_4$ and that they are intersected in 
that order. Let $a_i$ be th edge of $H_j$ which belongs to $\Delta_i$ 
and let $p_i$ be its midpoint. Consider the polygon in $\ell_{H_j}$ 
which connects $p_1,p_2,p_3,p_4$, according to this orden. Let us 
notice that $\ell_{H_j}\subset \bigcup_{i=1}^4\Delta_i$, and that 
$\bigcup_{i=1}^4\Delta_i$ is simply connected. Then $\ell\sim 
\ell_{H_j}$, through a free homotopy within $H_j$.

\par Using an argument similar to the previous one we can show that 
$\ell\sim \ell_{H_j}$ through a homotopy within $H_j$, when $\ell$ 
intersects  $H_j$ in 2 or 3 triangles. 

\par We conclude that:

\begin{equation}
	\ell\sim \ell_{H_1}\cdots \ell_{H_k}
\end{equation} 

\noindent by means of a homotopy that fixes the extreme points 
(remember we are taking the segment of $\ell(t)$ for $t\in[0,1]$) and 
the homotopy is inside the union of the triangles which intersects 
$\ell(t)$ for $t\in[0,1]$. If $\gamma_0$ and $\gamma_1$ are the 
generators of $\pi_1(2\Delta-\{0,1,\omega\},1/2)$, the natural loops 
around 0 and 1, respectively, then in $2\Delta$ one has:

\begin{equation}
	\gamma=\pi(\ell)=\prod \eta_i,\quad \eta_i\in 
	\{\gamma_0,\gamma_1,\gamma_1^{-1},\gamma_0^{-1}\}
\end{equation}
since a segment which connects the midpoints of two adjacent edges 
descends to $\gamma_0$, $\gamma_1$ or its inverses, according to the 
color of the common vertex and its orientation.

\par The previous arguments can also be applied when $\omega_0\notin 
Hex$, except that in this case we have to consider $\ell(t)$ with $t\in 
[0,3]$. Lets see an example which illustrates these ideas:

\begin{example}\label{ejemplo-3w}
	Let $\omega_0=3\omega\in Hex$, $\ell(t)=1/2+t\omega_0$, $t\in[0,1]$  
	(Figure (\ref{omega_0=3w})). In this case $\ell$ intersects two 
	hexagons $H_1$ and $H_2$, and in each hexagon intersects 3 
	triangles. Then $\ell(t)$, with $t\in [0,1]$ is homotopic to 
	$\ell_{H_1}\ell_{H_2}$ through a homotopy which fixes the extreme 
	points and within the union of the triangles it intersects. 
	Therefore:
	\begin{equation}
		\gamma=\gamma_1^{-1} \gamma_0^{-1}\gamma_1\gamma_0.
	\end{equation}
	
	\begin{figure}
		\begin{center} 
		\includegraphics[width=10cm]{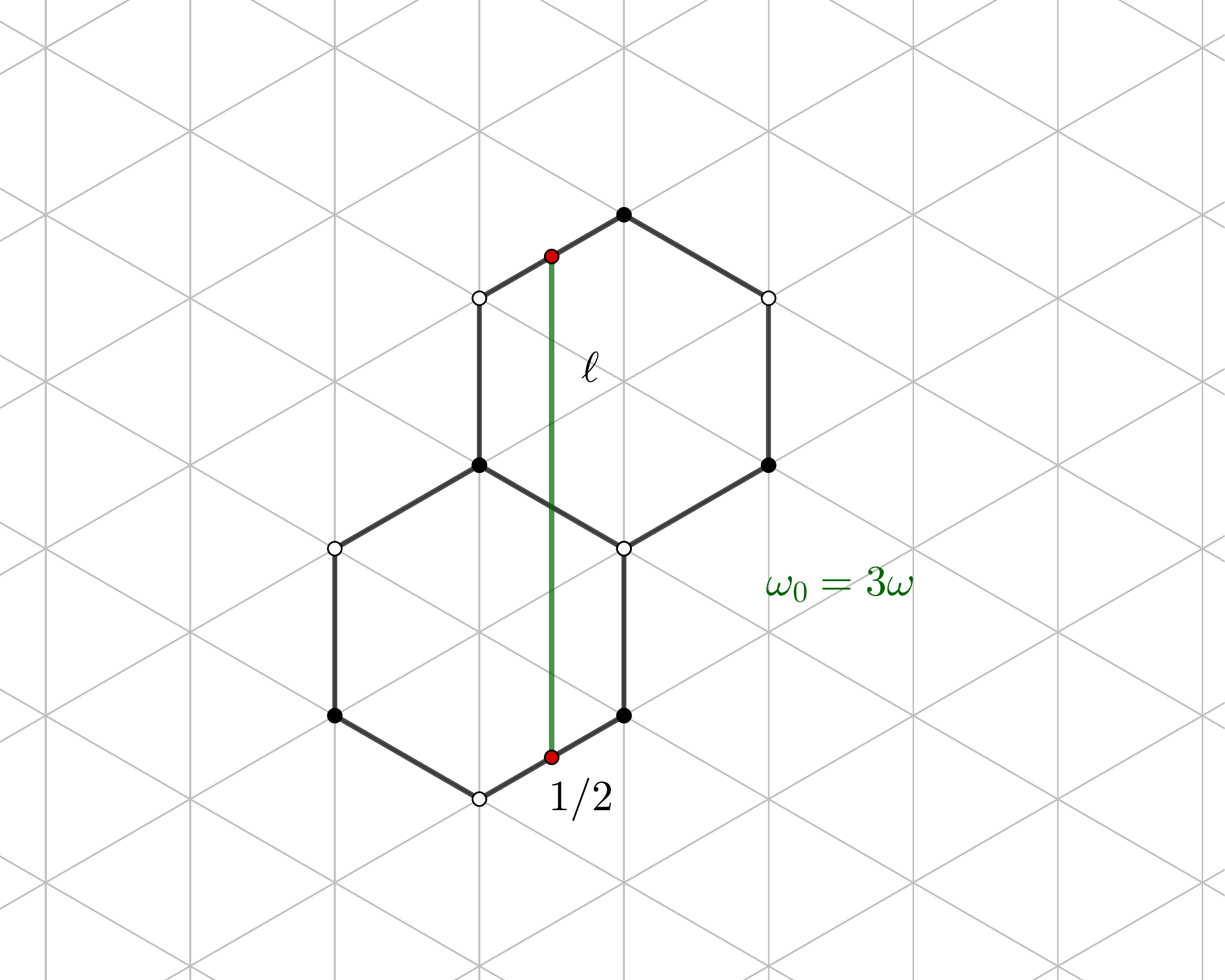}
		\caption{Factorization of $\gamma$ for $\omega_0=3\omega$.} 
		\label{omega_0=3w} 
		\end{center}
	\end{figure}
\end{example}

\par From the previous proof, as illustrated in Example 
\ref{ejemplo-3w}, we see that to find the factorization of 
$\gamma=\pi(\ell)$ it is enough to consider the triangles which which 
are intersected by the segment  $\ell(t)$, $t\in [0,1]$, if 
$\omega_0\in Hex$ or  $t\in [0,3]$ if $\omega_0\notin Hex$, and see how 
it permutes the edges decorated with the symbols $\circ$ and $\bullet$ 
in these triangles, as $t$ increases from 0 to 1. The permutation 
around a vertex of type $\circ$ contributes a factor $\gamma_0$ or 
$\gamma_0^{-1}$ and around a vertex of type $\bullet$ contributes a 
factor $\gamma_1$ or $\gamma_1^{-1}$, depending of the orientation. 
\par On the other hand, if $\gamma=\gamma_1\gamma_2$, then 
$\sigma_\gamma=\sigma_{\gamma_2}\sigma_{\gamma_1}$. Therefore if the 
factorization of  $\gamma$ is \begin{equation}
	\gamma=\eta_1\cdots\eta_n, \quad  \eta_i\in 
	\{\gamma_0,\gamma_1,\gamma_0^{-1},\gamma_1^{-1}\}
\end{equation}
it follows that

\begin{equation}\label{descomposicion-ciclica}
	\sigma_\gamma=\sigma_{\eta_n}\cdots\sigma_{\eta_1},\quad 
	\sigma_{\eta_i}\in \{\sigma_0,\sigma_1,\sigma_0^{-1},\sigma_1^{-1}\}
\end{equation}

recall that $\sigma_0=\sigma_{\gamma_0}$ and $\sigma_1=\sigma_{\gamma_1}$.
 
By hypothesis we know explicitly the permutations $\sigma_0$ and 
$\sigma_1$, so we can insert them in equation 
(\ref{descomposicion-ciclica}) and in this way we obtain explicitly its 
cyclic decomposition. In Table (\ref{direcciones-geodesicas}) we have 
computed the factorization of some geodesics 
$\gamma=\pi(1/2+t\omega_0)$ in the fundamental group 
$\pi_1(2\Delta-\{0,1,\omega\},1/2)$, following the algorithm described 
above, remembering that $t\in [0,1]$ if $\omega_0\in Hex$ and $t\in 
[0,3]$ otherwise.

\begin{center}
\begin{table}
\begin{tabular}{l|l} 
\hline
Directions $\omega_0$ &  Factorization of 
$\gamma$ in $\pi_1(2\Delta-\{0,1,\omega\},1/2)$\\ 
\hline \hline
1. $\omega_0=1+\omega\in Hex$ & $\gamma=\gamma_1^{-1}\gamma_0$.
\\
\\
2. $\omega_0=3\omega\in Hex,$ & 
$\gamma=\gamma_1^{-1}\gamma_0^{-1}\gamma_1\gamma_0.$
\\
\\
3. $\omega_0=2+3\omega\notin Hex$ & $\gamma=\gamma_1^{-1}\gamma_0 
\gamma_1^{-1}(\gamma_0^{-1})^2 \gamma_1^{-1}(\gamma_0^{-1})^2 
\gamma_1^{-1}\gamma_0^{-1}\gamma_1\gamma_0$\\ & \quad \quad 
$\gamma_1^2\gamma_0\gamma_1^2\gamma_0\gamma_1^{-1} \gamma_0$
\\
\\
4. $\omega_0=2+\omega\notin Hex$ & $\gamma=\gamma_1^{-1}\gamma_0 
\gamma_1^2\gamma_0\gamma_1\gamma_0^{-1}\gamma_1^{-1}(\gamma_0^{-1})^2
\gamma_1^{-1}\gamma_0$
\\
\\
5. $\omega_0=10+\omega\in Hex$ & $\gamma=\gamma_1^{-1}\gamma_0\gamma_1
\gamma_0^{-1}\gamma_1^{-1}\gamma_0\gamma_1\gamma_0^{-1}\gamma_1^{-1}
\gamma_0\gamma_1\gamma_0^{-1}\gamma_1^{-1}\gamma_0$
\\
\\
6. $\omega_0=13+7\omega\in Hex$ & $\gamma=\gamma_1^{-1}\gamma_0 
\gamma_1^2\gamma_0\gamma_1\gamma_0^{-1}\gamma_1^{-1}(\gamma_0^{-1})^2
\gamma_1^{-1}\gamma_0\gamma_1^{-1}\gamma_0\gamma_1^{-1}$\\ & \quad 
\quad $\gamma_0\gamma_1^2\gamma_0\gamma_1\gamma_0^{-1}\gamma_1^{-1}
(\gamma_0^{-1})^2 \gamma_1^{-1}\gamma_0$
\\
\\
7. $\omega_0=5-\omega\in Hex$ & $\gamma=\gamma_1\gamma_0^{-1}
\gamma_1^{-1}\gamma_0\gamma_1\gamma_0^{-1}$
\\
\\
8. $\omega_0=10+7\omega\in Hex$ &  $\gamma=\gamma_1^{-1}\gamma_0
\gamma_1^{-1}\gamma_0\gamma_1^2\gamma_0\gamma_1^2\gamma_0\gamma_1
\gamma_0^{-1}\gamma_1^{-1}(\gamma_0^{-1})^2\gamma_1\gamma_0^2$\\ & 
\quad \quad $\gamma_1^{-1}\gamma_0 \gamma_1^{-1} \gamma_0$
\\
\\
9. $\gamma_0 = 2+ 5\omega\in Hex$ & $\gamma=\gamma_1^{-1}
(\gamma_0^{-1})^2\gamma_1^{-1}\gamma_0^{-1}\gamma_1\gamma_0\gamma_1^2
\gamma_0$
\\
\hline
\end{tabular}
\caption{Some directions $\omega_0$ and the factorization of 
$\gamma(t)=\pi(1/2+t\omega_0)$ in $\pi_1(2\Delta-\{0,1,\omega\},1/2)$.}
\label{direcciones-geodesicas}
\end{table}
\end{center}

Next, we give an example where we apply the previous algorithm
and compute the lengths of some geodesics.

\begin{example}

Let $X$ be the surface with the decorated triangulated structure
obtained by taking the double of the plane annulus shown in 
Figure  
(\ref{dessin-diamante-Figure }). Let us fix the labeling of its
dessin d'enfant as shown on the right hand side of Figure  
(\ref{dessin-diamante-Figure }).

\begin{figure}
	\begin{center} 
	\includegraphics[width=13cm]{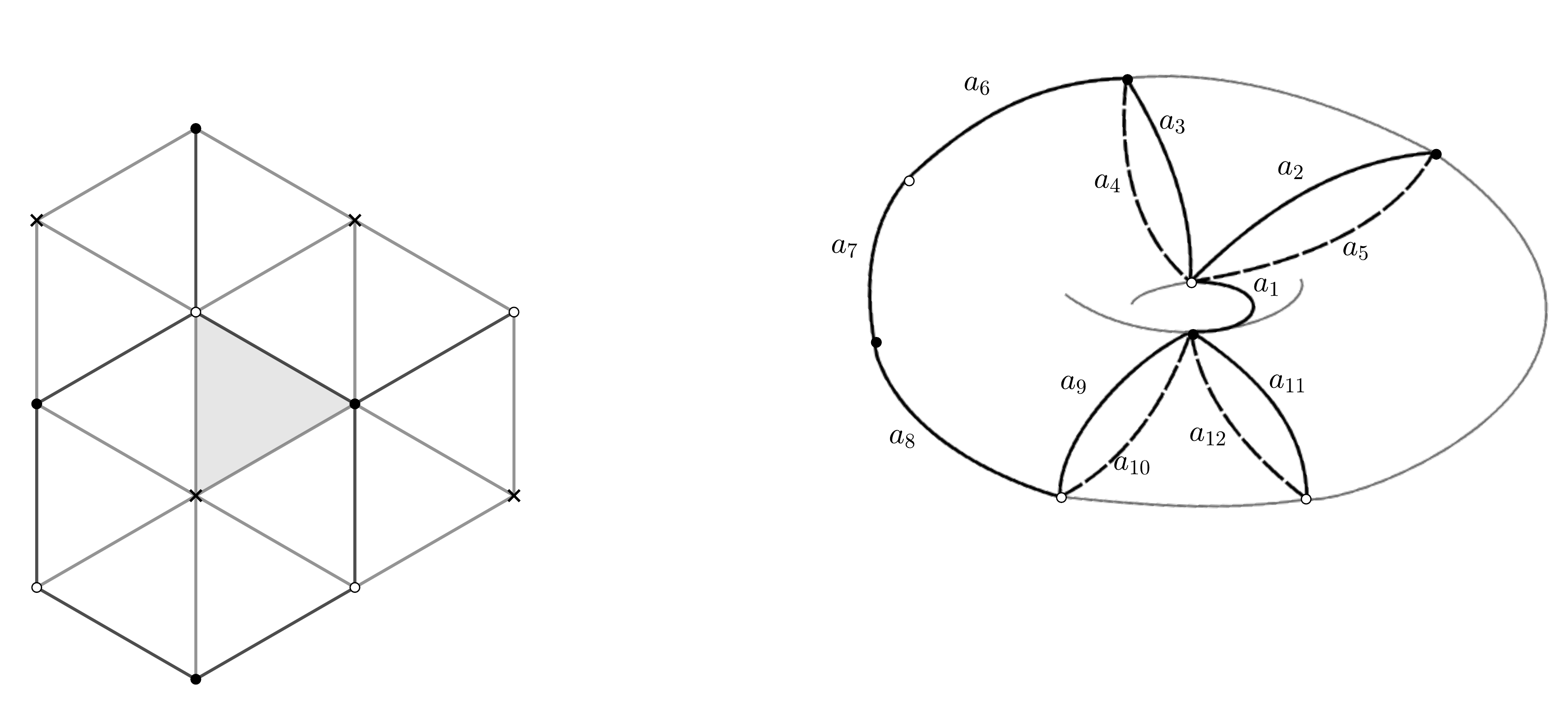}
	\caption{Doble of the annulus-diamond and its associated dessin.} 
	\label{dessin-diamante-Figure } 
	\end{center}
\end{figure}

The permutations associated to the dessin are	
\begin{eqnarray}\label{permutaciones-dessin-diamante}
	\sigma_{0}& = & (1\ 2\ 3\ 4\ 5)(9\ 8\ 10)(6\ 7)(11\ 12)\nonumber \\
	\sigma_{1}& = & (1\ 12\ 10\ 9\ 11)(3\ 4\ 6)(2\ 5)(7\ 8).
\end{eqnarray}

Let $\omega_0=5-\omega$. According to Table
(\ref{direcciones-geodesicas}) $\gamma(t)=\pi(1/2+t\omega_0)$  
factorizes in $\pi_1(2\Delta-\{0,1,\omega\},1/2)$, as follows 

\begin{equation}
	\gamma=\gamma_1\gamma_0^{-1}
\gamma_1^{-1}\gamma_0\gamma_1\gamma_0^{-1}.
\end{equation}

\noindent Hence $\sigma_{\gamma}$ factorizes as
\begin{equation}
	\sigma_{\gamma}=\sigma_0^{-1}\sigma_1\sigma_0\sigma_1^{-1}
	\sigma_0^{-1}\sigma_1.
\end{equation}

Replacing (\ref{permutaciones-dessin-diamante}) in the previous 
equation we obtain:

\begin{equation}
	\sigma_\gamma=(1\ 6\ 11\ 3\ 9\ 4\ 10\ 5\ 8)(2\ 7\ 12)
\end{equation}

\par \noindent  (we used the package \emph{Sympy} of \emph{Python} 
to obtain this product).

If $\tilde{\gamma}$ is the geodesic in $X$ with tangent vector 
$\omega_0=5-\omega$ at the midpoint of the edge $a_2$, see Figure 
(\ref{5-w-diamante}) we obtain that $\operatorname{ord}\sigma_{a_2}=3$ 
and $5-\omega\in Hex$. Hence by Corollary \ref{corolario-length} it 
follows that

\begin{figure}
	\begin{center} 
	\includegraphics[width=12cm]{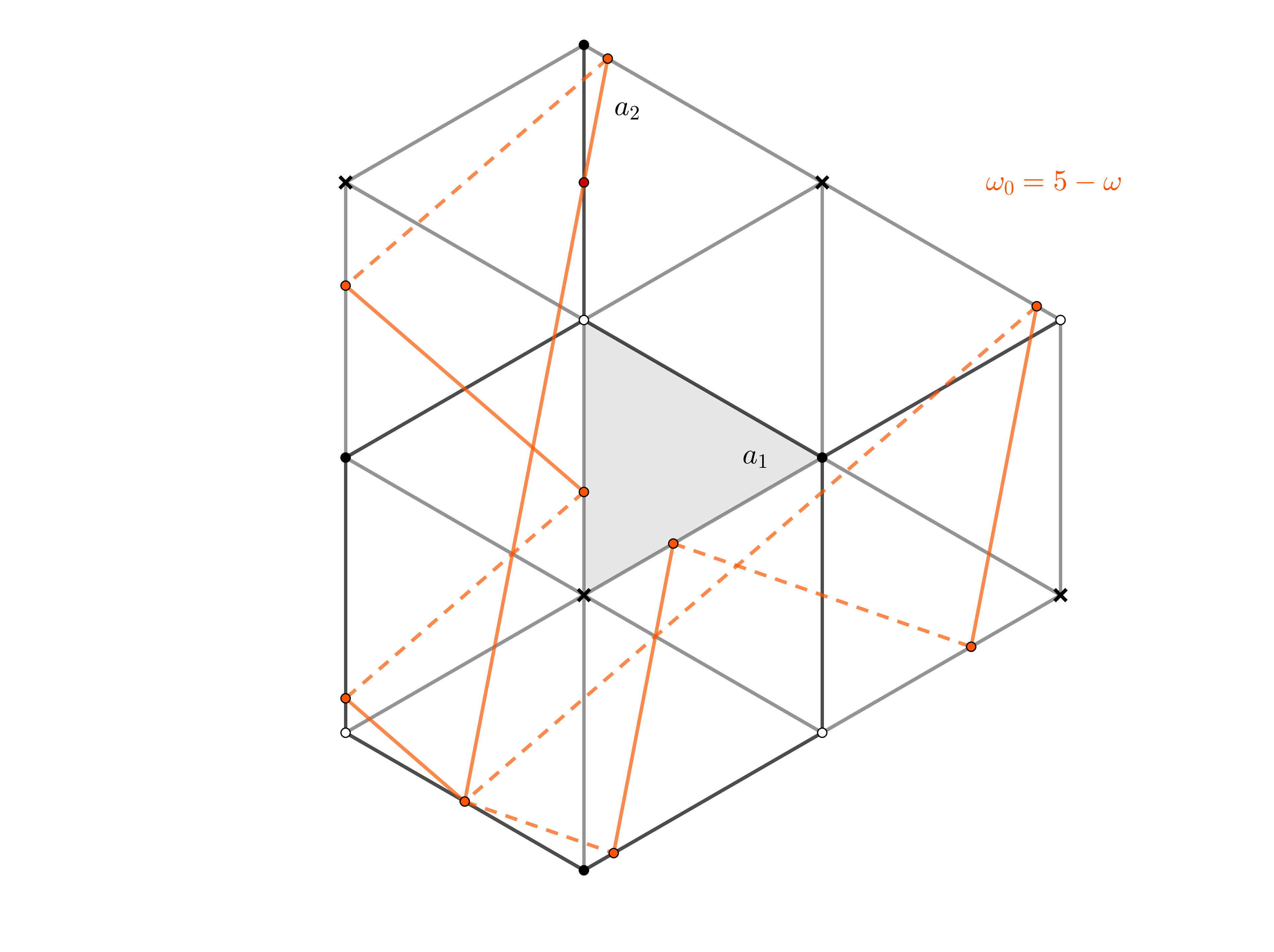}
	\caption{Geodesic $\tilde{\gamma}$ based at the midpoint of
	$a_2$ and with direction direccton $5-\omega$.} 
	\label{5-w-diamante} 
	\end{center}
\end{figure}

\begin{equation}
	\operatorname{long} \tilde{\gamma}=\operatorname{ord} 
	\sigma_{a_2}|\omega_0|=3|5-\omega|\approx 13.76.
\end{equation}

This length can be corroborated directly in Figure  
(\ref{5-w-diamante}).

\par Analogously, if $\omega_0=10+\omega$, then $\sigma_{\gamma}$ has 
the cyclic decomposition
\begin{equation}
	\sigma_\gamma = (1\ 8\ 5\ 2\ 12\ 11\ 6)(3\ 4)(9\ 10)(7).
\end{equation}

Hence, if $\tilde{\gamma}$ is the geodesic with  tangent  vector
$10+\omega$ in the midpoint of $a_7$ (Figure  (\ref{10+w-diamante})), 
one obtains:

\begin{equation}
	\operatorname{long}(\tilde{\gamma})= \operatorname{ord} 
\sigma_{a_7}|\omega_{0}|= |10+\omega|\approx 10.54.
\end{equation}

\begin{figure}
	\begin{center} 
	\includegraphics[width=12cm]{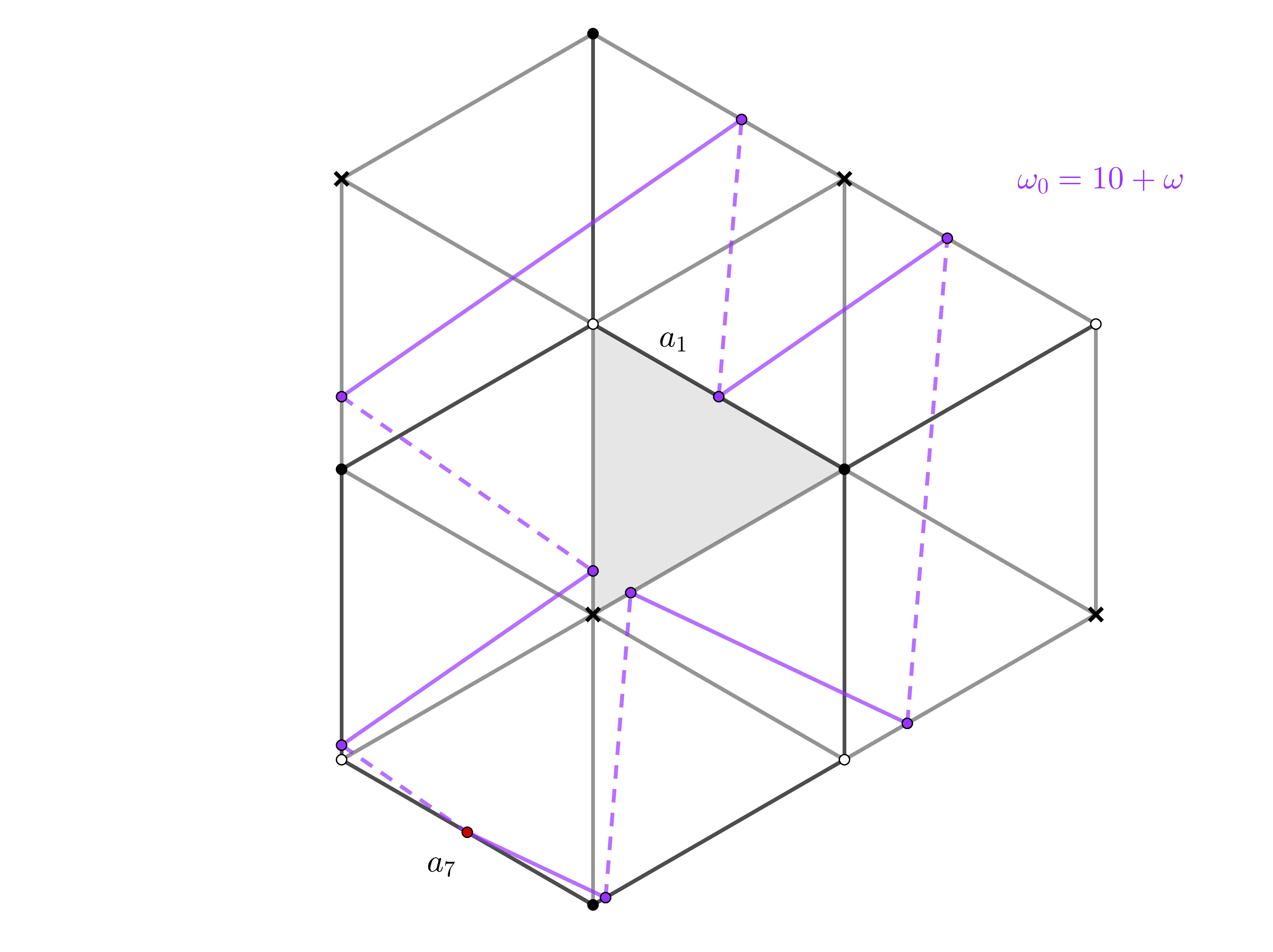}
	\caption{Geodesic $\tilde{\gamma}$ based in the midpoint of
	$a_1$ with direction $10+\omega$.} 
	\label{10+w-diamante} 
	\end{center}
\end{figure}

If $\tilde{\gamma}$ is based in $a_1$, $\operatorname{ord} 
\sigma_{a_1}=7$, and

\begin{equation}
	\operatorname{long}(\tilde{\gamma})= \operatorname{ord} 
\sigma_{a_1}|\omega_{0}|= 7|10+\omega|\approx 73.75.
\end{equation}

\par The cyclic decompositions of  $\sigma_\gamma$ in the previous  
examples, and others, are shown in Table
(\ref{tabla-factorizacion-gamma-diamante}).  

\begin{center}
\begin{table}
\begin{tabular}{l|l} 
\hline
Directions $\omega_0$ &  Factorization of 
$\sigma_\gamma\in S_{12}$\\ 
\hline \hline
1. $\omega_0=1+\omega\in Hex$ & $\sigma_\gamma=(1\ 12\ 2)(3\ 7\ 10\ 11\ 
8\ 6\ 5)$\\ \\
2. $\omega_0=3\omega\in Hex$ & $\sigma_\gamma=(1\ 9\ 6\ 5\ 11\ 8\ 4)
(2\ 7\ 12\ 3\ 10)$\\ \\
3. $\omega_0=2+\omega\notin Hex$ & $\sigma_\gamma=(2\ 4\ 11)(3\ 7\ 10)
(5\ 9\ 12)$\\ \\
4. $\omega_0=13+17\omega\in Hex$ & $\sigma_\gamma=(1\ 5\ 12\ 7\ 2\ 11)
(4\ 8\ 6\ 9)$\\ \\
5. $\omega_0=5-\omega\in Hex$ & $\sigma_\gamma=(1\ 6\ 11\ 3\ 9\ 4\ 10\ 
5\ 8)(2\ 7\ 12)$\\ \\
6. $\omega_0=10+\omega\in Hex$ & $\sigma_\gamma=(1\ 8\ 5\ 2\ 12\ 11\ 6)
(3\ 4)(9\ 10)$
\\
\hline
\end{tabular}
\caption{Some directions $\omega_0$ and the factorization of
$\sigma_\gamma$ with respect to the dessin given by
(\ref{permutaciones-dessin-diamante}).} 
\label{tabla-factorizacion-gamma-diamante}
\end{table}
\end{center}
\end{example} 

\begin{proposition}
	For each $k\geq 0$, $\omega_0=(3k+1)+\omega$ is contained in $Hex$. 
	Furthermore if $\gamma_{\omega_0}$ is the parametrized geodesic in 
	$2\Delta$ $t\mapsto\pi(1/2+t\omega_0)$, then 
	
	\begin{equation}\label{factorizacion-(3k+1)+w}
		\gamma_{\omega_0}=(\gamma_1^{-1}\gamma_0)(\gamma_1\gamma_0^{-1}
		\gamma_1^{-1}\gamma_0)^k.
	\end{equation} 
	
	Therefore, if $X$ is a Riemann surface endowed with a decorated 
	equilateral structure whose associated dessin d'enfant defines the 
	permutations $\sigma_0$ and $\sigma_1$, then $\sigma_\gamma$ is 
	factorized as follows:
	\begin{equation}
		\sigma_{\gamma_0}=(\sigma_0\sigma_1^{-1}\sigma_0^{-1}\sigma_1)^k
		\sigma_0\sigma_1^{-1}
	\end{equation}
\end{proposition}

\begin{proof}
	Observe that $(3k+1)+\omega=m(2-\omega)+n(1+\omega)$, with $m=k$ and
	$n=k+1$. This proves the first assertion.
	
	\par On the other hand 
	\begin{equation}
		\gamma_{\omega_0}=\gamma_{(3(k-1)+1)+\omega}(\gamma_1\gamma_0^{-1}
		\gamma_1^{-1}\gamma_0).
	\end{equation}
	
	Hence, if we apply the previous formula recursively we obtain
	\begin{equation}
		\gamma_{\omega_0}=\gamma_{1+\omega}(\gamma_1\gamma_0^{-1}
		\gamma_1^{-1}\gamma_0)^k.
	\end{equation}
	Since  $\gamma_{1+\omega}=\gamma_1^{-1}\gamma_0$, one has formula
	(\ref{factorizacion-(3k+1)+w}).
\end{proof}

\begin{remark}\label{observacion-sigma1}
	If $\sigma_1=(1)$ and $\omega_0=(3k+1)+\omega$, it follows from
	this proposition that $\sigma_{\gamma}=\sigma_0$.
\end{remark}

\begin{example}
	Consider the dessin on the Riemann sphere 
	given by  Belyi's function $z^n$. In this case the permutations 
	are
	\begin{eqnarray}
		\sigma_0 & = & (1\ 2\ \cdots n)\nonumber\\
		\sigma_1 & = & (1).
	\end{eqnarray} 
	Let us take the equilateral structure which turns each triangle of 
	the triangulation defined by the dessin into an equilateral 
	triangle. Then it follows from Remark \ref{observacion-sigma1} that 
	for any geodesic $\tilde{\gamma}$ with tangent vector at some 
	midpoint of an edge equal to $\omega_0=(3k+1)+\omega$ one has:	
	\begin{equation}
		\operatorname{long} \tilde{\gamma}=n|(3k+1)+\omega|.
	\end{equation}
\end{example}

\section{Conic hyperbolic structures}
 
\subsection{Conic hyperbolic metrics on compact surfaces with a decorated triangulation.}
 Let $X$ be a compact surface endowed with a decorated triangulation associated to
a dessin d'enfant. Then there exists a Belyi function $f_1\colon X\to 2\Delta$ 
which realizes the decorated triangulation. We recall that
 $\Delta$ is the equilateral triangle with vertices in
$0,1$ and $\omega$, where
$\omega=\exp(2\pi i/6)$; and $2\Delta$ is the sphere obtained by gluing two 
decorated triangles as explained at the beginning.

\par Let us consider a hyperbolic triangle contained in the  Poincar\'e
disk $\mathbb{D}$, which has vertices at
 $z_1$, $z_2$ and
$z_3$ such that: (i) $z_1=0$, (ii) the edge between $z_1$ and $z_2$ is
in the real axis and 
(iii) the triplet  $(z_1,z_2,z_3)$ defines a positive orientation of the triangle. 
Such a triangle is completely determined by its angles.
 If $\pi/m$, $\pi/n$ and $\pi/l$  are the angles at $z_1$, 
$z_2$ and $z_3$, respectively,  then the triangle will be denoted
by $\Delta_{m,n,l}$, see Figure  (\ref{TriangHip444}). 
We will also suppose that the triangle is decorated with
 $\circ,\ \bullet$, $\ast$ in $z_1$, $z_2$ and $z_3$ respectively. 
 Such a decoration determines a coloration into black and white triangles just 
 like was done in the euclidean case.
 
 \begin{figure}
	\begin{center} 
	\includegraphics[width=10cm]{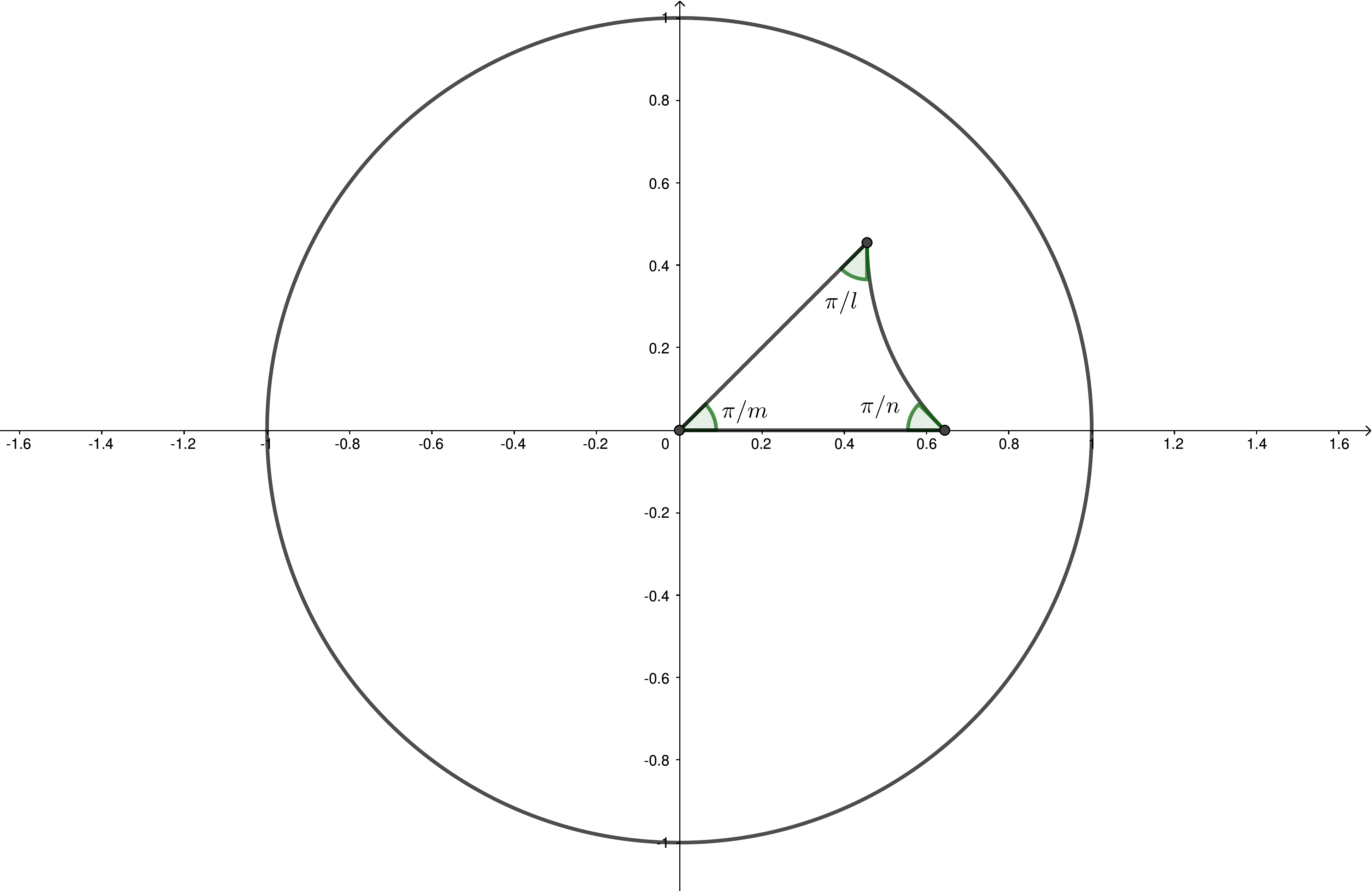}
	\caption{Hyperbolic triangle $\Delta_{m,n,l}$, with $m=n=l=4$.} 
	\label{TriangHip444} 
	\end{center}
\end{figure}
 
\par By an analogous construction as that of $2\Delta$ in the euclidean 
case we can define $2\Delta_{m,n,l}$: we use the hyperbolic reflection 
$R$ which fixes, point by point, the edge between $z_1$ and $z_2$; we 
identify the boundary of $\Delta_{m,n,l}$ with the boundary of its 
reflection $R(\Delta_{m,n,l})$, using $R$. This surface 
$2\Delta_{m,n,l}$ is homeomorphic to the 2-sphere and it has a natural 
hyperbolic metric outside of the vertices. And each of the two 
triangles is isometric to $\Delta_{m,n,l}$. In addition, as in the 
euclidean case, the triangulation induces a complex structure on 
$2\Delta_{m,n,l}$ which turns it into a Riemann surface. We will assume 
that the Riemann surface $2\Delta_{m,n,l}$ is endowed with this conic 
hyperbolic metric.

\begin{remark} The surface $2\Delta_{m,n,l}$ is conformal to 
	$\mathbb{H}/\Gamma_{m,n,l}$, where $\Gamma_{m,n,l}$ 
	is the triangular group generated by $\Delta_{m,n,l}$.
\end{remark}
 
\par There exists a conformal map which preserves the boundaries and 
send vertices to vertices respecting the decoration. By Schwarz 
Reflection Principle we can extend this map to its reflected triangles 
in the Poincar\'e disk. After identifications we have a conformal map 
on the quotients $\Phi\colon 2\Delta\to 2\Delta_{m,n,l}$, see Figure  
(\ref{PhiEucliHip}). 
 
 \begin{figure}
	\begin{center} 
	\includegraphics[width=12cm]{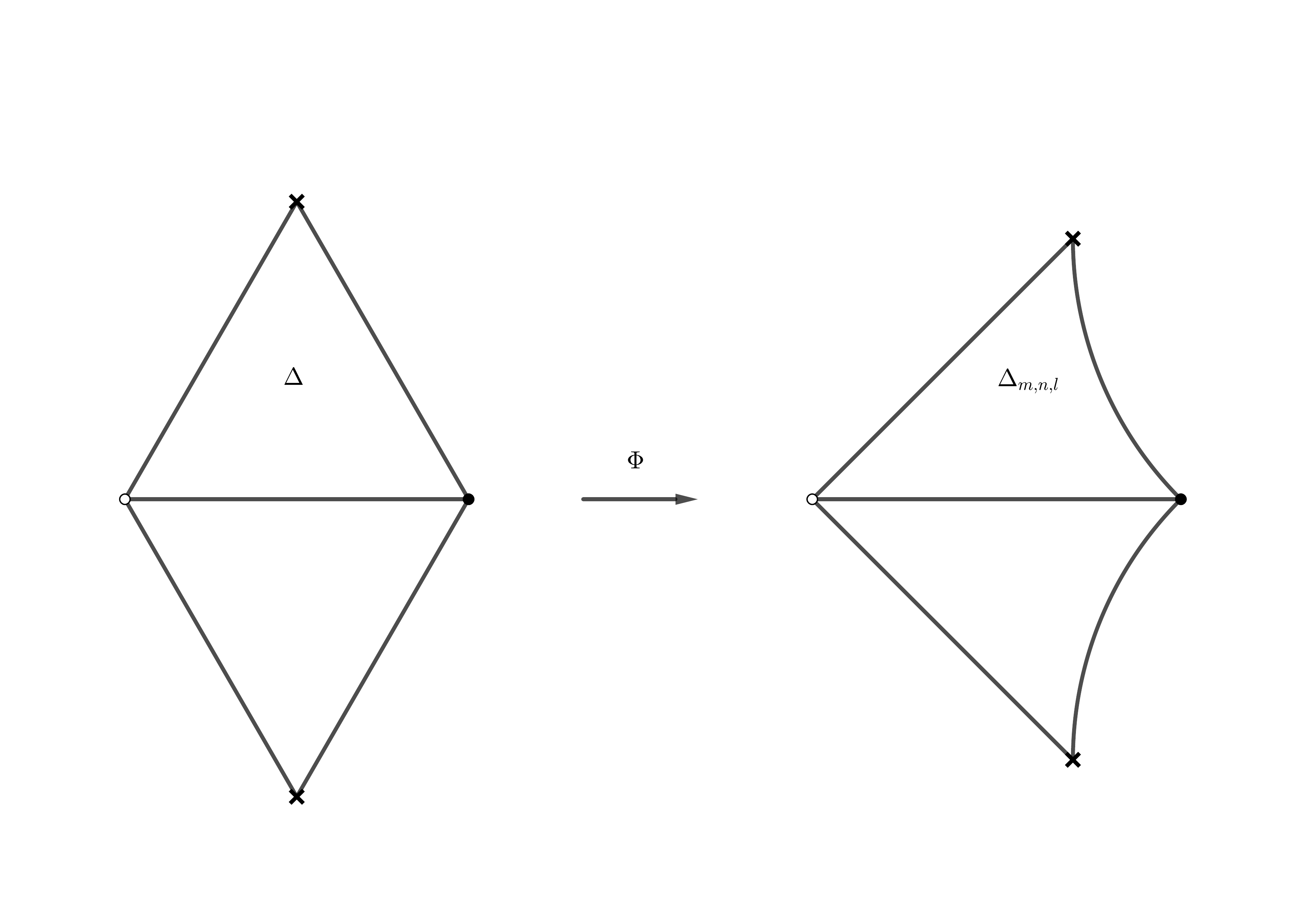}
	\caption{Conformal map between $2\Delta$ and $2\Delta_{m,n,l}$.} 
	\label{PhiEucliHip} 
	\end{center}
\end{figure}
 
\par Using $\Phi$ we construct a Belyi function $f=\Phi\circ f_1\colon 
X\to 2\Delta_{m,n,l}$. This is a holomorphic map with respect to the 
complex structures we have defined. Let us notice that this function 
realizes the decorated triangulation we had originally.
 
If we pull back to $X$, using $f$, the conic hyperbolic metric of 
$2\Delta_{m,n,l}$, the triangles of $X$ become isometric to 
$\Delta_{m,n,l}$. 
 
Summarizing, we obtain the following proposition:

\begin{proposition}\label{metrica-hiperbolica-conica}
	Let $X$ be a surface with a decorated triangulation induced by a 
	dessin d'enfant. For any positive integers $m,n,l$, such that  
	$1/m+1/n+1/l<1$, there exists a conic hyperbolic metric in $X$ such 
	that with this metric each triangle is isometric to to the 
	hyperbolic triangle $\Delta_{m,n,l}$. This metric is induced by a 
	Belyi function $f\colon X\to 2\Delta_{m,n,l}$  which realizes the 
	decorated triangulation.
\end{proposition}

\subsection{Representation of a Belyi function in the hyperbolic 
plane.}

\par Let $X$ be a surface with a decorated triangulation $T$ induced by 
a dessin d'enfant. Let $O(T)$ be the greatest valence of the vertices 
of the triangulation $T$. If $f\colon X\to \hat{\mathbb{C}}$ is the 
Belyi function that realizes $T$ one has $O(T)= \max_{p\in 
f^{-1}\{0,1,\infty\}}\operatorname{mult}_p(f).$

\par  Let $m$ be an integer, $m>3$, such that $m>O(T)$. 
Let us endow $X$ with the conic hyperbolic metric induced by a
Belyi function
$f\colon X\to 2\Delta_{m,m,m}$, which realizes the decorated triangulation 
(see Proposition \ref{metrica-hiperbolica-conica}).
 
Let $P$ be a (connected but not necessarily convex) polygon formed by a finite number
of triangles of the tessellation $T(m,m,m)$ corresponding to the triangular 
group $\Gamma_{m,m,m}$, generated by the triangle 
$\Delta_{m,m,m}$ \cite{Magnus}. 

\begin{definition}
A  \emph{hiperbolic peel}  of $X$ is a continuous map  
$\Psi\colon P\to X$ which satisfies the following conditions:

\begin{itemize} 
\item[(i)] Sends triangles of the tessellation  $T(m,m,m)$ onto triangles of $T$.
\item[(ii)] The restriction to each triangle of $T(m,m,m)$ is an isometry

\item[(iii)] It respects the decoration (we assume that the tessellation 
$T(m,m,m)$ is decorated compatibly with $\Delta_{m,m,m}$).
\item[(iv)] The function is injective in the interior of $P$ and each point of 
$X$ has at most two pre-images bajo $\Psi$. Thus the double pre-images
can occur only at $\partial{P}$, the boundary of $P$.
\end{itemize}
\end{definition}
\par If we define the equivalence relation in 
 $P$ as $x\sim y$ if and only 
$\Psi(x)=\Psi(y)$, we obtain a rule to glue some sides of the boundary
of the polygon by hyperbolic isometries that belong to the triangle group
 $\Gamma_{m,m,m}$. 
Condition (iv) implies that an edge on the boundary is paired with
at most another edge, but it is possible that certain edges are not paired with any other edge. 
\par Note that the peel $\Psi$ descends to the quotient to an embedding 
$\hat\Psi\colon (P/_\sim) \to X$. 

\begin{example}
	Figure  (\ref{cascara-omega}) shows the peel which corresponds to the elliptic curve 
	from Figure (\ref{hexagono-omega}). The function $\Psi$ maps each triangle onto 
	$\Delta_i$ in $T_i$ respecting the decorations. Here, as before, 
	we assume that the hexagon has the metric induced by the
	Belyi function $f\colon X\to 2\Delta_{4,4,4}$.
	
	\begin{figure}
	\begin{center} 
	\includegraphics[width=12cm]{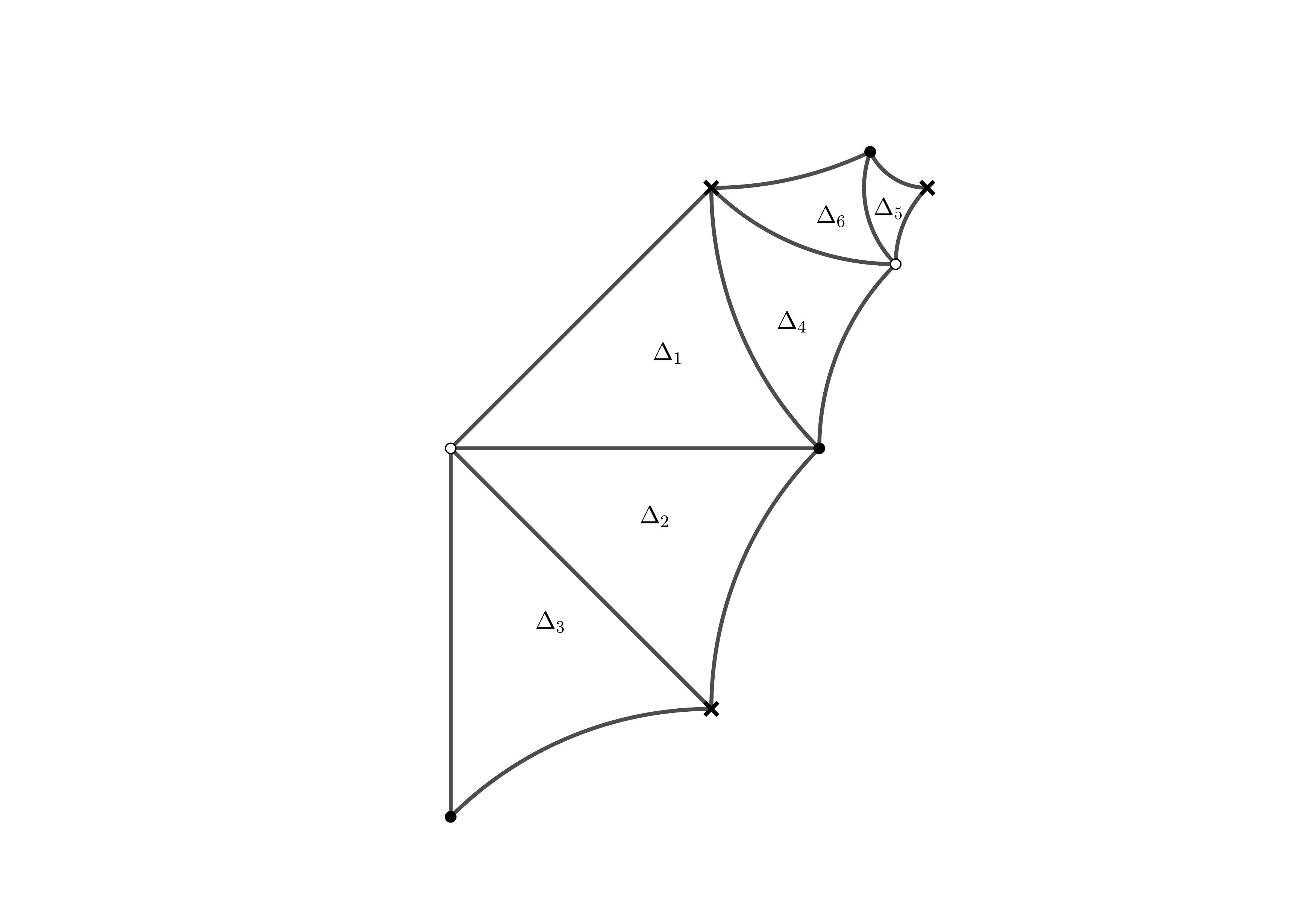}
	\caption{Peel associated to the elliptic curve 
	$\mathbb{C}/\Lambda_{1,\omega}$, with $\omega=\exp(2 \pi i /6)$. 
	In this case we take $m$ equal to $4$.} 
	\label{cascara-omega} 
	\end{center}
	\end{figure}
	
	\begin{figure} 
	\begin{center} 
	\includegraphics[width=12cm]{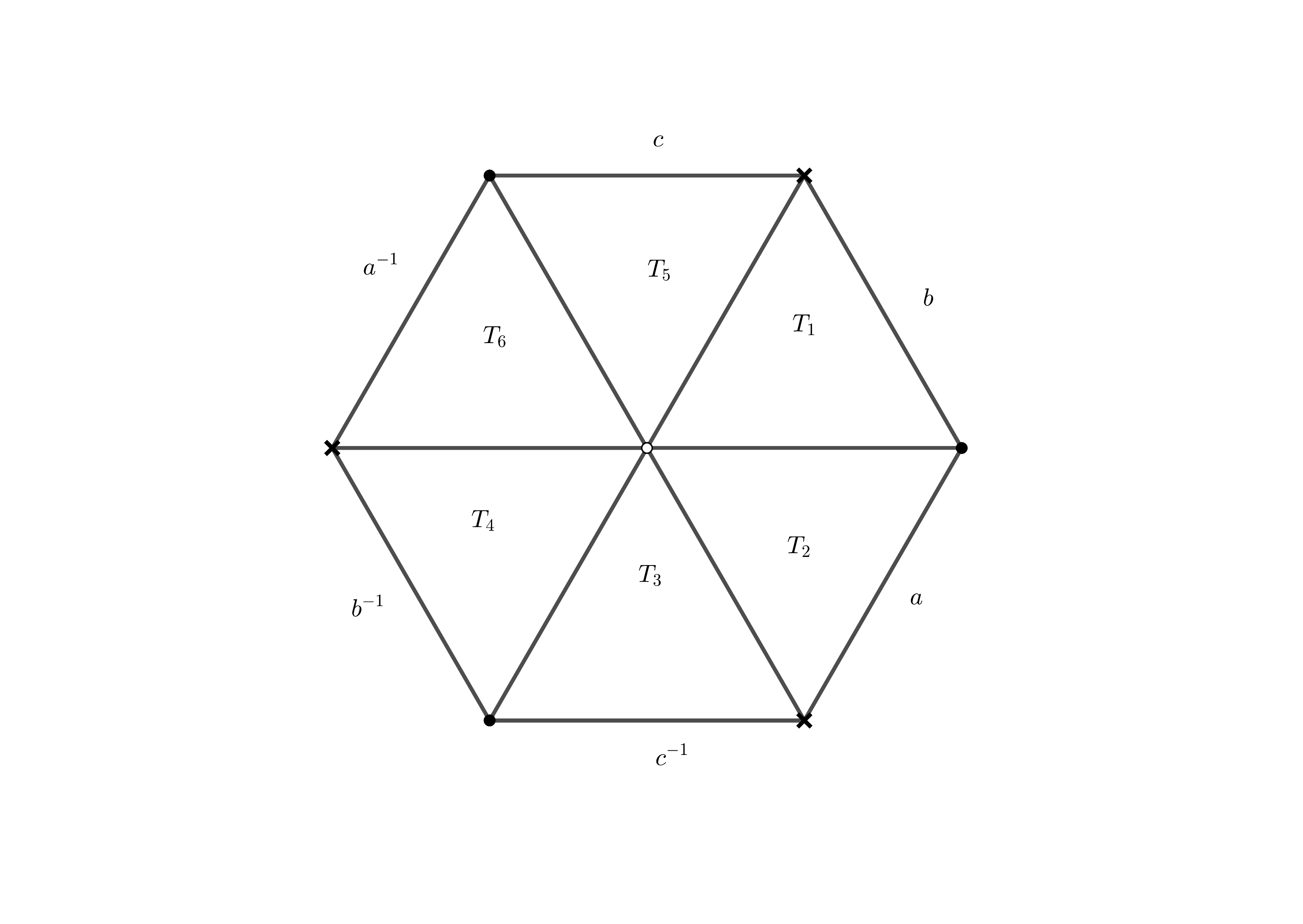} \caption{Hexagon 
	with sides identified, corresponding to
	$\mathbb{C}/\Lambda_{1,\omega}$, with $\omega=\exp(2 \pi i /6)$.} 
	\label{hexagono-omega} 
	\end{center}
	\end{figure}
\end{example}

\begin{remark}\label{no-vecindad}
	Each vertex $p$ of the polygon $P$ belongs to at most $O(T)$ 
	triangles of $P$; since $m > O(T)$, the union of such triangles is not a neighborhood 
	of  $p$. Therefore all the vertices of the triangulation of  
	 $P$ are in its boundary $\partial{P}$.
\end{remark}

\begin{remark}\label{si-todos}
	If all of the edges of $\partial{P}$ are pairwise identified then
	$\hat\Psi$ is surjective, hence it is a homeomorphism because
	$\Psi(P/_\sim)$ is a connected surface without boundary 
	contained in $X$.
\end{remark}

\begin{proposition}\label{no-todos}
	If one of the edges of $\partial{P}$ is not paired with some other edge
	then $\Psi$ can be extended to a polygon with one more triangle than the
	number of triangles of $P$.
\end{proposition}
\begin{proof}
	Let $a$ be an edge of $\partial{P}$ which is not paired with any 
	other edge. Let us add to $P$ the triangle $\Delta_1$, of the 
	tessellation, which has $a$ as an edge and which does not belong to 
	$P$. Such triangle must exist because $a$ is in the boundary of 
	$P$. Since $m>O(T)$ the edges of $\Delta_1$, which are distinct to  
	$a$, are not in $P$; this insures us that we don't use edges which 
	are paired with other edges of $\partial{P}$.

	\par One decorates the vertices of the added triangle $\Delta_1$ 
	according to the edge $a$ and color it according to the positive or 
	negative orientation determined by the ordering of the vertices.

	\par If $T_1$ is the triangle in $X$, adjacent to $\Psi(a)$, define 
	$\Psi\colon \Delta_1\to T_1$ as the isometry which preserves the 
	decoration we obtain a peel $\Psi\colon P\cup \Delta_1\to 
	\Psi(P)\cup T_1$ that has one more triangle than $P$.
\end{proof}

\begin{remark}
	It follows from Proposition \ref{no-todos} and by an argument of 
	maximality, since the triangulation has a finite number of 
	triangles, that there exist a peel $\Psi\colon P\to X$ such 
	that every edge in $\partial{P}$ is paired with another edge so 
	that the edge of the boundary are pairwise identified. Then, by 
	Proposition \ref{si-todos}, $\Psi$ descends to a homomorphism 
	$\hat\Psi\colon (P/_\sim) \to X$ which is a conformal map.
\end{remark}

Summarizing we have the following theorem:

\begin{theorem}
	If $X$ is a decorated surface with a conic hyperbolic metric induced by
	a Belyi function $f\colon 
	X\to 2\Delta_{m,m,m}$ which realizes the triangulation, and $m$ is 
	integer such that $m>3$ and $m>O(T)$. Then there is a peel
	$\Psi\colon P\to X$ which is maximal and which descends to a
	conformal mapping to the quotient $\hat\Psi\colon (P/_\sim) \to X$.
\end{theorem}

 $\hat\Psi\colon (P/_\sim)\to X$ is a conformal map which sends triangles onto 
 triangles of the respective triangulations. It also respect the decorations and
 it restriction to each triangle is an isometry. Also
$P/_\sim$ has Belyi function $f_P\colon (P/_\sim)\to 2\Delta_{m,m,m}$, defined
as the map which sends triangles onto
triangles, by an isometry which respects the decorations, including the colorations
of the triangles. By construction, the
map $\hat\Psi$ is an isomorphism of ramified coverings \ie the following diagram commutes:

\begin{equation}
	\xymatrix{ P/_\sim\ar^{\hat\Psi}[rr]\ar_{f_P}[rd] & &X\ar[ld]^{f}\\
	& 2\Delta_{m,m,m} & }
\end{equation}

\subsection{The case of ideal decorated  hyperbolic triangulations}

\par In the same fashion as in the case of hyperbolic triangles 
$\Delta_{m,n,l}$ with positive angles $\pi/m$, $\pi/n$ and $\pi/l$
we can consider the case when the triangle is ideal \ie its angles are
all equal to 0 ($m=n=l=\infty$). 

\par In particular we consider the ideal triangle in the upper half plane $\mathbb{H}$ 
that has as its vertices $0$, $1$ and $\infty$. 
We denote this triangle by $\Delta_{\infty,\infty,\infty}$,
see Figure  (\ref{oo}). 

\begin{figure}
	\begin{center} 
	\includegraphics[width=10cm]{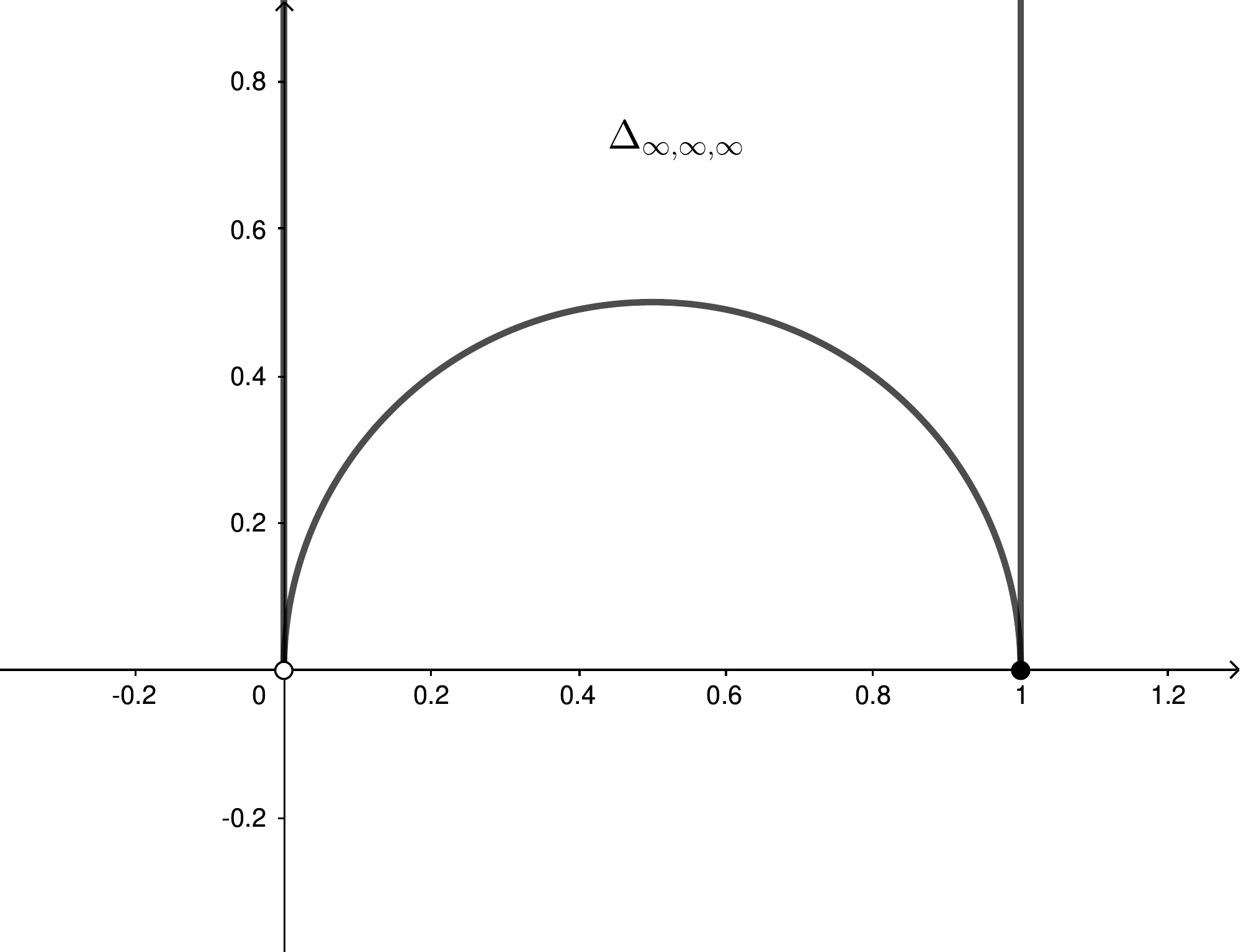}
	\caption{Ideal triangle with vertices in $0$, $1$ and $\infty$.} 
	\label{oo} 
	\end{center}
\end{figure}

 The vertices $0$, $1$ and $\infty$ are decorated with the symbols  
$\circ,\bullet,\ast$ respectively; then we use the hyperbolic reflection 
$R$ with respect to the edge which connects $\circ$ with $\bullet$, and
we identify the boundary of $\Delta_{\infty,\infty,\infty}$ with
the boundary of $R(\Delta_{\infty,\infty,\infty})$ by the reflection $R$. 

\par The surface of genus 0 constructed as the union 
along the boundary of two ideal triangles has a complete hyperbolic
metric and therefore the metric induces a complex structure. As a 
Riemann surface it is conformally 
equivalent to the Riemann sphere minus three points. This decorated surface
will be denoted by $2\Delta_{\infty,\infty,\infty}$.

\begin{remark}
	The surfsce $2\Delta_{\infty,\infty,\infty}$ is also conformally equivalent to 
	$\mathbb{H}/\Gamma(2)$ where the group $\Gamma(2)$ is the
	modulo 2 congruence 
	subgroup of $\text{PSL}(2,\Z)$ (we recall that
	$\Gamma(2)=\Gamma_{\infty,\infty,\infty}$).
\end{remark}

 If $X$ is a decorated triangulated surface, obtained from a dessin 
d'enfant, there exists a Belyi map $f\colon X\to 2\Delta_{\infty,\infty,\infty}$ 
which realizes the decorated triangulation. Therefore we can 
endow $X$  with a conic hyperbolic metric by pulling-back the 
conic hyperbolic metric of $2\Delta_{\infty,\infty,\infty}$, 
by means of $f$. 

\par Also in this case one has a maximal peel, $\Psi\colon 
P\to X$, which descends to the quotient to a conformal map
$\hat\Psi\colon (P/_\sim)\to X$. The definition of the peel is completely analogous
to the case of finite triangles. 

\emph{The following remark highlights the advantage of using ideal triangles}:
\begin{remark} In this case as $P$ is an ideal polygon it must be convex.
\end{remark} 

\par The ideal peel of $P$ can also be obtained from the uniformization
of the sphere with three punctures: we consider
$2\Delta_{\infty,\infty,\infty}$ as the quotient
$\mathbb{H}/\Gamma(2)$. If $f\colon X\to \mathbb{H}/\Gamma(2)$ is a 
 Belyi function, by properties of covering spaces 
 there is a subgroup $K_0$ of of finite index of $\Gamma(2)$ and a conformal map $\Psi$ 
 such that the following diagram is commutative:

\begin{equation}\label{diagramaGamma}
	\xymatrix{\mathbb{H}/K_0\ar^{\Psi}[rr] \ar_{\pi}[rd] & & X\ar^f[ld]\\
	& \mathbb{H}/\Gamma(2)&}
\end{equation}

\noindent where $\pi$ is the function which acts on right cosets as follows: $[z]_{K_0}\mapsto [z]_{\Gamma(2)}$. 
Then, $\Psi$ must send 
triangles onto triangles, preserve the decoration and be an isometry when
restricted to each triangle. Therefore if we consider a fundamental domain 
 $P$ of the action of  $K_0$, we obtain a peel 
$\Psi\colon P\to X$ which coincides, in the quotient, with the map $\Psi$ 
in Diagram (\ref{diagramaGamma}).

\begin{example}
	In Figure  (\ref{ejemplo-tres-cilindros}) we show three  
	parallelograms identified along their boundaries (as indicated in the figure). 
	The surface with boundary
	obtained after gluing  is a closed annulus. Such an annulus is decorated as
	indicated in the figure. If we take the double of this annulus we have a surface
	$X$, with a decorated triangulation and homeomorphic to the 2-torus.
	 
	\par In Figure  (\ref{ejemplo-tres-cilindros}) we label the triangles
	with the symbols $\Delta_i$, 
	$i=1,2,3,4,5,6$. In addition, we suppose that
	the triangles in the other part of the double are labeled in $X$
	with the symbols $\Delta_i'$,  $i=1,2,3,4,5,6$, in such a way that
	$\Delta_i'$ is the double of $\Delta_i$.
	
	\par We chose the (unique) conic hyperbolic metric on $X$
	which renders each triangle $\Delta_i$ and $\Delta_i'$ 
	isometric to the ideal hyperbolic triangle with vertices in
	$1,\rho,\rho^2$, where $\rho=\exp(2\pi i /3)$. We also suppose
	that the vertices have the original decoration.

	\begin{center}	
	\begin{figure} 
	\includegraphics[width=10cm]{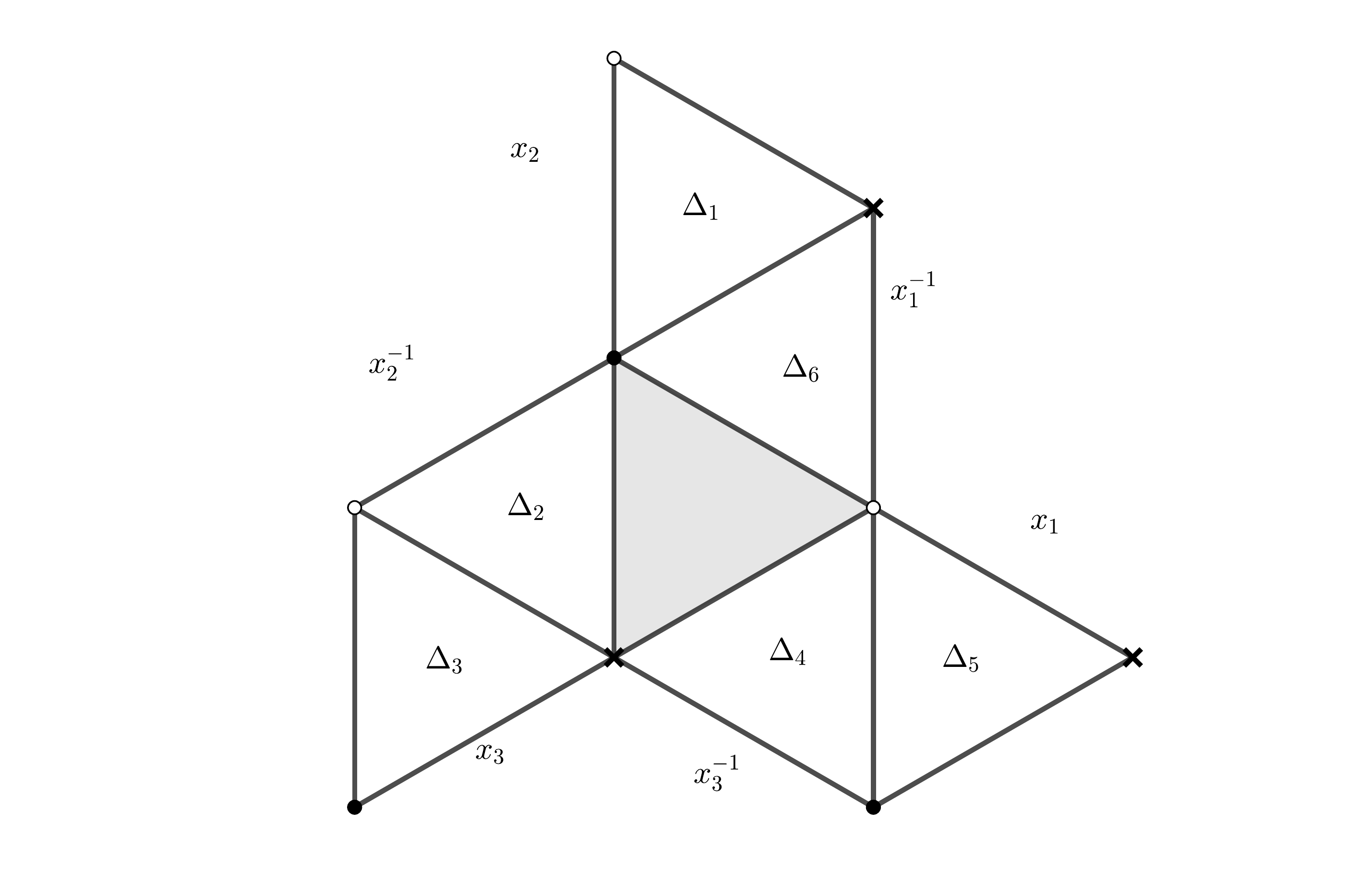}
	\caption{Parallelograms identified along their boundaries. 
	The sides marked with letters are identified according to the orientation. 
	After the identification one gets an annulus.} 
	\label{ejemplo-tres-cilindros} 
	\end{figure}
	\end{center}
	
	Let $P$ be the ideal decorated polygon in Figure 
	(\ref{cascara-chorizo}). Define $\Psi\colon P\to X$ as the function 
	which maps each triangle of $P$ (Figure (\ref{cascara-chorizo})) 
	onto a triangle of  $X$ by an isometry which preserves the 
	decoration and according with the labels. By definition $\Psi$ es a 
	peel of $X$. 
	
	\par We observe that $\Psi(\partial P)$ is a graph in the surface 
	and it is a union of edges of the triangulation. (Figure  
	(\ref{grafica-chorizo})).
	
	\begin{figure}
	\begin{center} 
	\includegraphics[width=14cm]{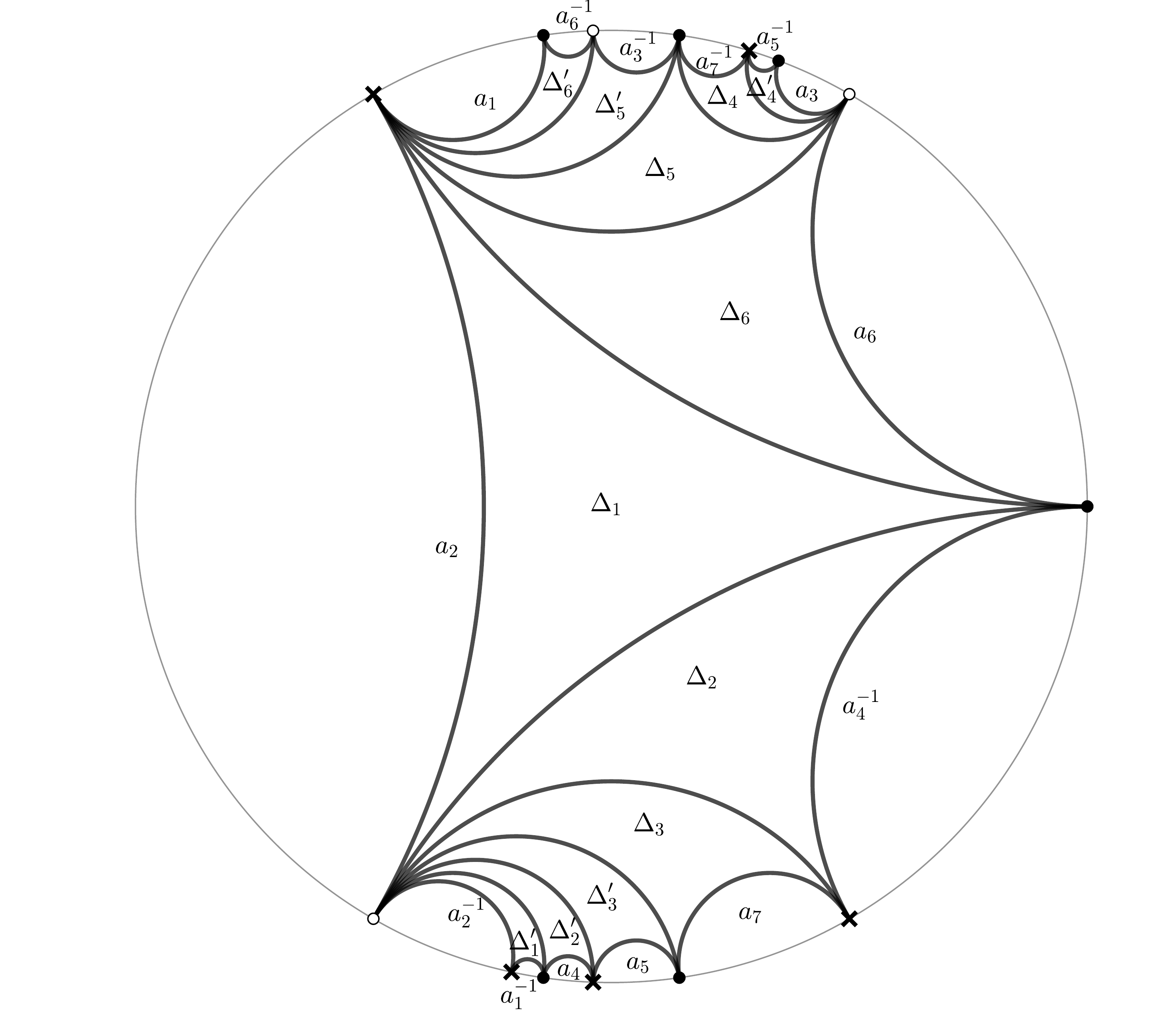}
	\caption{Example of an ideal peel. If we identify the edges of the 
	boundary as indicated in the figure we get an elliptic curve.} 
	\label{cascara-chorizo} 
	\end{center}
	\end{figure}
	
	\begin{figure}
	\begin{center} 
	\includegraphics[width=12cm]{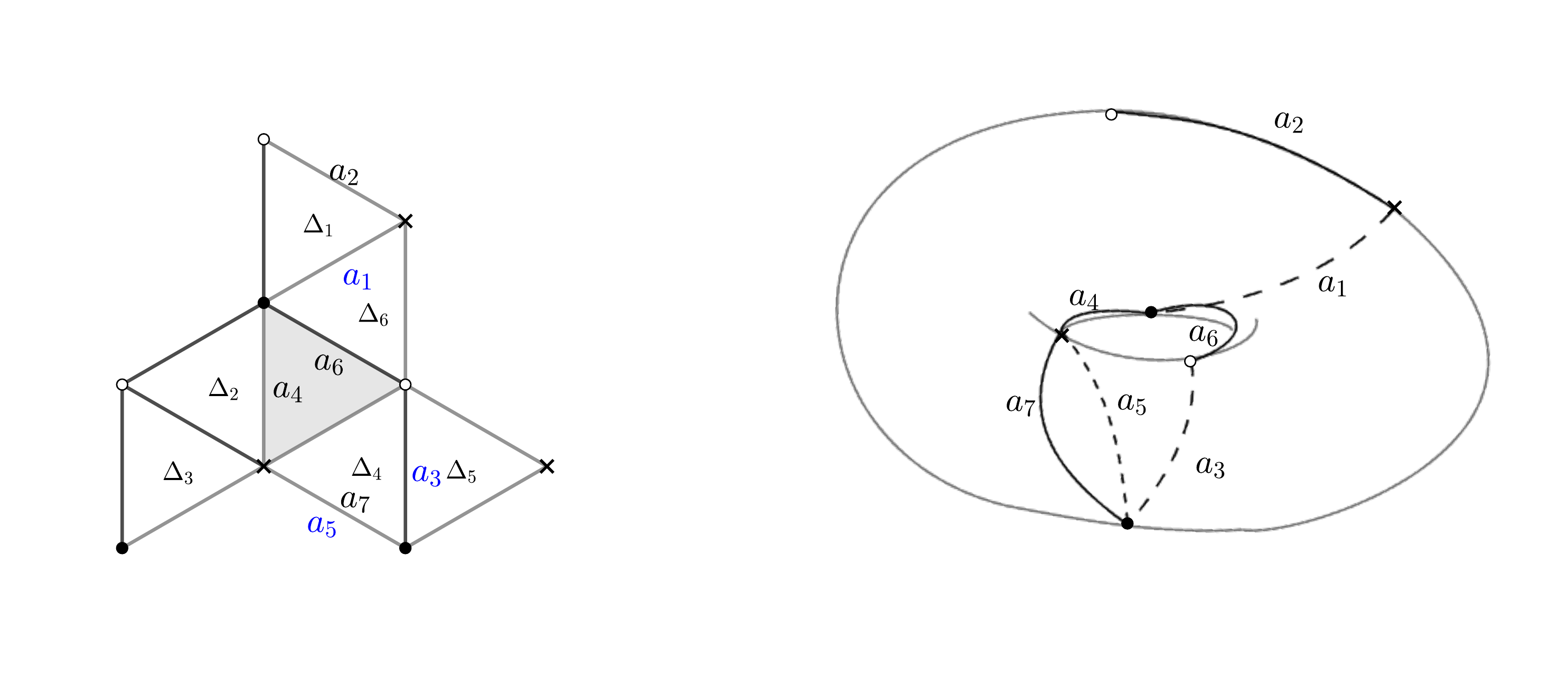}
	\caption{The image of $\partial P$ under $\Psi$ is a graph in 
	$X$ such that ita complement in $X$ is simply connected. In the right-hand-side
	of the picture the letters in blue indicates that the edge is in the reverse of the triangle.} 
	\label{grafica-chorizo} 
	\end{center}
	\end{figure}
\end{example}

\begin{definition}
	Let $X$ be a Riemann surface and $\Gamma$ a subgroup of
	$\operatorname{Aut} (X)$ (the group of holomorphic automorphisms of $X$). 
	We say that a region $F$ of $X$  is a fundamental domain
	of $\Gamma$ if: i)  any $z\in X$ is equivalent (\ie is in the same orbit) under the
	action of $\Gamma$ with an element $z'$ in the closure of $F$,
	ii) two distinct points in $F$ are  not equivalent under the action of $\Gamma$.
\end{definition}

\begin{remark}
	Let us suppose that $X$ is an elliptic curve (\ie a compact
	Riemann surface of genus 1). Suppose that the complex structure is
	given by an equilateral decorated hyperbolic triangulation
	with triangles isometric to $\Delta_{\infty,\infty,\infty}$. 
	
	Let $\Psi\colon P\to X$ be an ideal peel for
	$X$. Suppose 
	that $\Lambda_{1,\tau}=\left\{m+n\tau\,\,:\, m, n \in \Z  \right\}$ 
	is the lattice which
	uniformizes $X$ \ie $X\cong\C/\Lambda_{1,\tau}$ and  $\pi\colon \mathbb{C}\to X$ is
	the universal covering map such that $\operatorname{Deck}(\pi)$ (the group of deck transformations) 
	is equal to $\Lambda_{1,\tau}$. 
	
	Since $P$ is simply connected by covering space theory,
	we can find a lifting of $\Psi$ to a map from $P$ to $C$ \ie a map
	$\tilde{\Psi}$ such that the diagram (\ref{levantamiento-cascara}) is commutative:
	
	\begin{equation}\label{levantamiento-cascara}
		\xymatrix{ & \mathbb{C}\ar^{\pi}[d]\\
		P\ar^{\Psi}[r]\ar^{\tilde{\Psi}}[ru] & X}
	\end{equation}
	
	Wr claim that $\tilde{\Psi}(\mathring{P})$ is a fundamental domain
	for $\Lambda_{1,\tau}$, where $\mathring{P}$ 
	denotes the interior of  $P$. This is equivalent to proving that
	$\pi(\overline{\tilde{\Psi}(\mathring{P})})=X$ and that $\pi\colon 
	\tilde{\Psi}(\mathring{P})\to X$ is injective. To prove the first assertion, 
	note that by continuity 
	$\overline{\tilde{\Psi}(\mathring{P})}=\tilde{\Psi}(P)$, hence
	
	\begin{equation}
		\pi(\overline{\tilde{\Psi}(\mathring{P})})=\pi(\tilde{\Psi}(P))=
		\Psi(P)=X. 		
	\end{equation} 
		
	To prove the second assertion, let $x,y\in 
	\tilde{\Psi}(\mathring{P})$ be such that $\pi(x)=\pi(y)$. Then, there exist 
	$p_1,p_2\in \mathring{P}$ such that $x=\tilde{\Psi}(p_1)$ and
	$y=\tilde{\Psi}(p_2)$. Since the diagram
	(\ref{levantamiento-cascara}) commutes $\Psi(p_1)=\Psi(p_2)$, hence
	$p_1=p_2$ hence $x=y$. This proves that
	$\tilde{\Psi}(\mathring{P})$ is a fundamental domain for
	$\Lambda_{1,\tau}$.
\end{remark}

\section{Peels of a decorated equilateral
	 triangulation of an elliptic curve and its modulus $\tau$}
\begin{remark}
	 Let $X$ be a Riemann surface obtained from a decorated equilateral 
	 triangulation. We have a decorated graph whose edges and vertices 
	 are the edges and vertices of the triangulation. If $\Psi\colon 
	 P\to X$ is a peel for $X$, then $\mathcal{G}_P$ will denote the 
	 subgraph $\Psi(\partial P)$. In Figure  (\ref{grafica-chorizo}) 
	 (right-hand-side of the picture) we can see the graph 
	 $\mathcal{G}_P$ for the peel shown in (Figure 
	 (\ref{cascara-chorizo})).
\end{remark}

Recall that two topological spaces $X$ and $Y$ have the same homotopy 
type if there exist maps $f\colon X\to Y$ and $g\colon Y\to X$ such 
that $g\circ f$ is homotopic to the identity map $id_X$ and $f\circ g$ 
is homotopic to $id_Y$.

Two connected graphs (finite or not) have the same homotopy type if and only if
their fundamental group (which are free groups) are isomorphic.

\begin{proposition}
	Let us suppose that $X$ is an elliptic curve  obtained from a 
	decorated equilateral triangulation. If $\Psi\colon P\to X$ is a 
	peel for $X$. Then $\mathcal{G_P}$ is homotopic to a bouquet of two 
	circles. Hence, $\pi_1(\partial \mathcal{G}_P,x_0)$ is a free group 
	in two generators for any $x_0$ in $\partial \mathcal{G}_P$.
\end{proposition} 

\begin{proof}
	By graph theory, the graph  $\mathcal{G}_P$ 
	has the homotopy type of a bouquet of the following number
	of circles:

	\begin{equation}
		\beta_1(\mathcal{G}_P)=1-\chi(\mathcal{G}_P)=1-v+a
	\end{equation} 
	
	\noindent where $\chi(\mathcal{G}_P$) is the Euler  characteristic 
	of the graph. One has the cell decomposition of the torus $X$ given 
	by the vertices and edges of  $\mathcal{G}_P$ and the unique 2-cell 
	which is the complement of the graph on the torus (the interior of 
	$P$ is the unique 2-cell) we have that $\chi(X)=0=v-a+1$, hence 
	$1-\chi(\mathcal{G}_P)=2$. 
\end{proof}

\begin{remark} 
	The same proof can be used to show that, if instead of the torus, 
	$X$ is a surface of genus $g$ then $1-\chi(\mathcal{G}_P)=2g$ so 
	$\mathcal{G}_P$ and the first Betti number of  $\mathcal{G}_P$  is $2g$ so that it is homotopic to a bouquet of $2g$ circles.
\end{remark}

\begin{remark}
	Let us suppose that $X$ is an elliptic curve  obtained from a 
	decorated equilateral triangulation. Then if $\Psi\colon P\to X$ is 
	a peel for $X$. \par From the homology long exact sequence of the 
	pair $(X, \mathcal{G}_P)$, since $X-\mathcal{G}_P$ is contractible, 
	it follows that the inclusion $i\colon \mathcal{G}_P\to X$ induces 
	an isomorphism in the first homology groups: 
	
    \begin{equation}
        i_*\colon H_1(\mathcal{G}_P,\mathbb{Z})\to H_1(X,\mathbb{Z}).
    \end{equation} 

	\par Hence,

    \begin{equation}
        i_\#\colon \pi_1(\mathcal{G}_P,x_0)\to \pi_1(X,x_0)
    \end{equation}

	\noindent is an epimorphism from the free group in two generators onto the fundamental
	group of the elliptic curve which is isomorphic to $\Z\oplus\Z$. Therefore we can choose two curves 
	$\gamma_1$ and $\gamma_2$ based at a point $x_0\in\mathcal{G}_P$, totally contained in
	$\mathcal{G}_P$, and which generate $\pi_1(X,x_0)$.
\end{remark}
 
\begin{remark}
	Let $X$ be a Riemann surface with a marked base point 
	$x_0$ and $p\colon (\tilde{X},\tilde{x}_0)\to (X,x_0)$  be its 
	universal covering, with $\tilde{x}_0\in \tilde{X}$ in
	the fiber of $x_0$. By the theory of covering spaces
	it follows that for each
	$[\gamma]\in \pi_1(X,x_0)$ there exists a unique $g_\gamma\in 
	\operatorname{Deck}(\tilde{X},p)$ such that	
	$g_\gamma(\tilde{x}_0)=\tilde{\gamma}(1)$, where $\tilde{\gamma}$ 
	is the lifting of $\gamma$ based in $\tilde{x}_0$. Hence, we have	
	 a function $\pi_1(X,x_0)\to 
	\operatorname{Deck}(\tilde{X},p)$. It can be shown that this funcion
	is a group isomorphism.
	
\end{remark}
Therefore one has the following theorem:
\begin{theorem}
	Suppose that $X$ is an elliptic curve obtained by a decorated equilateral triangulation. 
	Let $p\colon \mathbb{C}\to X$ be its holomorphic universal covering. If $\Phi\colon P\to X$ 
	is a hyperbolic peel for $X$, the there exist two curves
	$\gamma_1$ and $\gamma_2$ en $\mathcal{G}_P$ such that
	$g_{\gamma_1}(z)=z+\omega_1$ and $g_{\gamma_2}(z)=z+\omega_2$ generate
	$\operatorname{Deck}(p)$. Hence: $\tau=\omega_1/\omega_2$ 
	uniformize to $X$ \ie $X$ has modulus $\tau$ and $X$ is biholomorphic to $\C/\Lambda_{1,\tau}$.
\end{theorem}

\section{Final remarks and questions}

The complex structures of surfaces $X$ deduced by a dessin $\mathcal{D}$ depend only on the combinatorics of the triangulation
$T_\mathcal{D}$ induced by the dessin or, equivalently, on the branched covering given by the associated Belyi map 
$f_{\mathcal D}:X\to \hat\C$. 
This in turn is given by a representation $\rho_\mathcal{D}:\pi_1(\bar\C-\{0,1,\infty\})\to{S_n}$ from the fundamental group of 
$\bar\C-\{0,1,\infty\}$ (which is a nonabelian free group in two generators) into the
symmetric group $S_n$ of permutations of $n$ objects. So we have the quadruple $(\mathcal{D}, T_\mathcal{D}, f_\mathcal{D}, \rho_\mathcal{D})$ where
each entity determines completely the other three.

In all that follows  $X$ will be an elliptic curve  obtained from a decorated equilateral triangulation induced by a dessin d'enfant $\mathcal{D}$. Suppose that $\mathcal{D}$ determines a Belyi function with 3 critical values (and not less than 3) which 
we always assume are 0, 1 and $\infty$.
Then $\mathcal{D}$ determines a decorated equilateral triangulation $T_{\mathcal{D}}$ by isometric equilateral 
triangles which are either euclidean or hyperbolic. 

there are several interesting questions
\begin{enumerate}
\item What properties must satisfy $\mathcal{D}$ (equivalently the combinatorics of $T_\mathcal{D}$ or $\rho_\mathcal{D}$) 
so that the elliptic curve is defined over $\Q$?
\item What properties must satisfy $\mathcal{D}$ (equivalently the combinatorics of $T_\mathcal{D}$ or $\rho_\mathcal{D}$)
so that the elliptic curve is modular?
\item  What properties must satisfy $\mathcal{D}$ so that the elliptic curve is defined over a quadratic extension 
$\Q(\sqrt{d})$ (for $d\in\Z$ not divisible by a square)?  In particular, what are the properties for $\mathcal{D}$ 
for $X$ to be an elliptic curve with complex multiplication? 
(partial results about this question are described in \cite{JW16} section 10.2).

\end{enumerate}
 We remark that if for an elliptic curve the conditions in item 1) imply the conditions in item 2), this would prove
 the celebrated Taniyama-Shimura conjecture, which in now a theorem. In 1995, Andrew Wiles and Richard Taylor 
 \cite{W, TW} proved a special case
 which was enough for Andrew Wiles to finally prove Fermat's Last Theorem. In 2001 the full conjecture was proven by
  Christophe Breuil, Brian Conrad, Fred Diamond and Richard Taylor \cite{BCDT}. This was one of the greatest 
  achievements in the field of mathematics of the last quarter of the XX century.
  
  \bigskip
  
   \noindent To prove that  for an elliptic curve the conditions in 1) imply the conditions in 2) is the \emph{Jugentraum} of the first  author (JJZ) and the \emph{Alterstraum} of the second (AV)!

\FloatBarrier


\begin{thebibliography}{123456789} 

\bibitem[Bel79]{Belyi} G. V. Belyi, \emph{Galois extensions of a 
maximal cyclotomic field}, Izv. Akad. Nauk SSSR Ser. Mat. 43(2) 
267–276, 479 (1979).

\bibitem[BCDT]{BCDT} Breuil C.; Conrad B.; Diamond F.; Taylor, R., \emph{On the modularity of elliptic curves over 
$\Q$: Wild 3-adic exercises}, Journal of the American Mathematical Society 14 (2001), pp. 843--939

\bibitem[Bost]{Bost} Bost, J-B. {\em Introduction to Compact Riemann Surfaces, Jacobians and Abelian Varieties}
From number theory to physics (Les Houches, 1989), 64--211, Springer, Berlin, 1992.

\bibitem[BNN95]{BNN95}  Brunet, R.; Nakamoto, A. ; Negami, S. \emph{Diagonal flips of triangulations on closed surfaces preserving specified properties}. J. Combin. Theory Ser. B 68 (1996), no. 2, 295--309.

\bibitem[CoItzWo94]{CIW} Cohen, Paula Beazley; Itzykson, Claude; Wolfart, J\"urgen 
{\em Fuchsian triangle groups and Grothendieck dessins. Variations on a theme of Bely\v{\i}}. Comm. Math. Phys. 163 (1994), no. 3, 605--627.

\bibitem[Esquisse]{Esquisse} Grothendieck, A.; {\em Esquisse d'un programme.} In L. Schneps \& P. Lochak (Eds.), 
{\em Geometric Galois Actions} (London Mathematical Society Lecture Note Series, pp. I-Vi). Cambridge: Cambridge University Press (1997).

\bibitem[For81]{Forster} Forster O., {\em Lectures on Riemann surfaces},
Graduate Texts in Mathematics, {\bf 81}, Springer-Verlag (1981).

\bibitem[GG12]{Girondo-Gonzalez} Girondo E. \& Gonz\'alez-Diez G., {\em 
Introduction to compact Riemann surfaces and dessins d'enfants}, London 
Mathematical Society Student Texts, {\bf 79}, Cambridge University 
Press (2012).

\bibitem[Gui14]{Gui14} Guillot, P. \emph{An elementary approach to dessins d'enfants and the Grothendieck-Teichm\"ller group.} Enseign. Math. 60 (2014), no. 3-4, 293--375. 

\bibitem[JW16]{JW16} Jones,G. A.; Wolfart, J. \emph{Dessins d'enfants on Riemann surfaces}. Springer Monographs in Mathematics. Springer, Cham, 2016. xiv+259 pp. 

\bibitem[Mag74]{Magnus} Magnus, W., {\em Noneuclidean tesselations and their groups}. Pure and Applied Mathematics, Vol. 61. Academic Press [A subsidiary of Harcourt Brace Jovanovich, Publishers], New York-London, 1974. xiv+207 pp. 

 \bibitem[ShVo90] {SV} Shabat, G. B.; Voevodsky, V. A. {\em Drawing curves osee number fields.} The Grothendieck Festschrift, Vol. III, 199--227, Progr. Math., 88, Birkh\"auser Boston, Boston, MA, 1990. 

\bibitem[ShZv94]{ShZv94} Shabat, G. B.; Zvonkin, A.\emph{Plane trees and algebraic numbers.} 
Jerusalem combinatorics 93, 233--275, Contemp. Math., 178, Amer. Math. Soc., Providence, RI, 1994.

\bibitem[Spr81]{Springer} Springer G., \emph{Introduction to Riemann 
	surfaces}, 2nd ed. (1981).
	
\bibitem[TW]{TW}Taylor, R,; Wiles, A.,
\emph{Ring-theoretic properties of certain Hecke algebras.}
Ann. of Math. (2) 141 (1995), no. 3, 553--572

\bibitem[Troy07]{Troyanov} Troyanov, M. {\em On the moduli space of singular euclidean surfaces}
 Handbook of Teichm\"ller theory. Vol. I, 507--540, IRMA Lect. Math. Theor. Phys., 11, Eur. Math. Soc., Z\"urich, 2007.

\bibitem[VSh89]{VS} Voevodski\v{\i}, V. A.; Shabat, G. B. {\em Equilateral triangulations of Riemann surfaces, and curves over algebraic number field}s. (Russian) Dokl. Akad. Nauk SSSR 304 (1989), no. 2, 265--268; translation in Soviet Math. Dokl. 39 (1989), no. 1, 38--41.

\bibitem[Wil95]{W} Wiles, A., \emph{Modular elliptic curves and Fermat's last theorem}. Ann. of Math. (2) 141 (1995), no. 3, 443--551


\end{thebibliography}
\end{document}